\documentclass[a4paper,12pt]{amsart}

\usepackage{float}
\usepackage{euscript,eufrak,verbatim}
\usepackage{graphicx}
\usepackage[usenames]{color}
\usepackage[colorlinks,linkcolor=red,anchorcolor=blue,citecolor=blue]{hyperref}
\usepackage{amsmath}
\usepackage{amsthm}
\usepackage[all]{xy}
\usepackage{graphicx} 
\usepackage{amssymb} 
\usepackage{amsfonts}

\usepackage{mathrsfs}
\usepackage{amscd}

\usepackage{tikz-cd}

\usepackage{bbm}

\makeatletter
\newcommand{\svdots}{%
  \vbox{\fontsize{\sf@size}{\sf@size pt}\linespread{0.3}\selectfont
    \kern0.2\baselineskip
    \hbox{.}\hbox{.}\hbox{.}%
    \kern0.1\baselineskip
  }%
}
\makeatother

\theoremstyle{plain}
\newtheorem{main theorem}{Main Theorem}
\newtheorem{theorem}{Theorem}[section]
\newtheorem{lemma}[theorem]{Lemma}

\newtheorem{corollary}[theorem]{Corollary}
\newtheorem{proposition}[theorem]{Proposition}
\newtheorem{claim}[theorem]{Claim}

\newtheorem{lemma-definition}[theorem]{Lemma-Definition}
\theoremstyle{definition}

\newtheorem{remark}[theorem]{Remark}
\newtheorem{example}[theorem]{Example}

\newtheorem{problem}[theorem]{Problem}

\makeatletter 
\@addtoreset{equation}{section}
\numberwithin{equation}{section}

\newcommand{\norm}[1]{\left\lVert#1\right\rVert}

\newcommand{\diam}{\mathrm{Diam}}
\newcommand{\supp}{\mathrm{supp}}
\newcommand{\h}{h_{\mathrm{top}}}
\newcommand{\FHh}{h_{\mathrm{FH}}}
\newcommand{\mesh}{\mathrm{mesh}}

\title[Intersecting random translates of Bedford--McMullen carpets]
{Weighted topological entropy and intersecting random translates of Bedford--McMullen carpets}

\author{Nima Alibabaei, Masaki Tsukamoto}

\address[Nima Alibabaei]
{Department of Mathematics, Kyoto University, Kyoto 606-8502, Japan}

\email{alibabaei.nima.28c@st.kyoto-u.ac.jp}

\address[Masaki Tsukamoto]
{Department of Mathematics, Kyoto University, Kyoto 606-8502, Japan}

\email{tsukamoto@math.kyoto-u.ac.jp}

\begin{document}



\thanks{N.A. was supported by JSPS KAKENHI Grant Number 25KJ1473.
M.T. was supported by JSPS KAKENHI JP25K06974.}

\begin{abstract}
We establish a relativised variational principle for the Feng--Huang weighted
topological entropy associated with a factor map between dynamical systems.
Combined with a recent theorem of Yin, this yields an almost-everywhere equivalence
between the Feng--Huang entropy and its combinatorial version on fibers.
As an application, we compute the Hausdorff dimension of the intersection of random
translates of two Bedford--McMullen carpets.
The resulting formula extends the Kenyon--Peres formula from the self-similar to the self-affine setting, 
and also points to a new problem concerning random matrix products.
\end{abstract}

\maketitle

\section{Background}  \label{section: Background}

The purpose of this paper is twofold.
First, we establish a relativised version of the Feng--Huang weighted variational principle \cite{Feng--Huang}.
Second, we apply it to the Hausdorff dimension of intersections of random translates of Bedford--McMullen carpets.
One of the most intriguing features of our work is that the resulting dimension formula 
points to a new problem concerning random matrix products.
In this section, we review the two main strands of research motivating these results.
Our main theorems are stated in \S\ref{section: main results}.
Since this paper is rather long, we include a brief roadmap in 
\S \ref{subsection: organization of the paper}.

Throughout the paper, a \textbf{dynamical system} means a pair $(X, T)$ such that 
$X$ is a compact metrizable space and $T$ is a continuous map from $X$ to itself.
(We do not assume that $T$ is a homeomorphism.)

\subsection{Weighted topological entropy} \label{subsection: weighted topological entropy}

The variational principle for topological entropy \cite{Goodwyn, Dinaburg, Goodman}
is one of the most basic results in topological dynamics:
Let $(X, T)$ be a dynamical system. Then we have
\begin{equation}  \label{eq: variational principle}
   \h(X, T) = \sup_{\mu\in \mathscr{M}^T(X)} h_{\mu}(T). 
\end{equation}   
Here $\h(X, T)$ and $h_{\mu}(T)$ denote the topological entropy and the Kolmogorov--Sinai entropy respectively, and 
$\mathscr{M}^T(X)$ is the set of all $T$-invariant Borel probability measures on $X$.

Motivated by the dimension theory of self-affine carpets and sponges, Feng and Huang \cite{Feng--Huang} discovered
an innovative idea; they introduced a 
\textit{weighted version} of \eqref{eq: variational principle}.
Their approach is based on the Hausdorff dimension-like definition of topological entropy 
developed by Bowen \cite{Bowen}.
Here we explain their theory using notations that differ slightly from the original ones.

Let $(X, T)$ and $(Y, S)$ be dynamical systems. 
Let $\pi\colon X\to Y$ be an equivariant continuous map, i.e. a continuous map satisfying 
$\pi\circ T = S\circ \pi$.
(We will sometimes write $\pi\colon (X, T)\to (Y, S)$ to clarify the underlying dynamics $T$ and $S$.)
For a measure $\mu\in \mathscr{M}^T(X)$ we denote by $\pi_*\mu \in \mathscr{M}^S(Y)$ the push-forward measure of $\mu$ by $\pi$.
Fix a real number $0\leq w \leq 1$.
Let $\mathbf{d}$ and $\mathbf{d}^\prime$ be metrics on $X$ and $Y$ respectively.
For $x\in X$, a natural number $n$ and a positive number $\varepsilon$,
we define the \textbf{$w$-weighted Bowen ball} $B^w_n(x, \varepsilon)$ as the set of $y\in X$ satisfying the following two conditions
\begin{align*}
 \mathbf{d}\left(T^k x, T^k y\right) &< \varepsilon \quad (0\leq k < \lceil wn\rceil) \\
 \mathbf{d}^\prime\left(S^k \pi(x), S^k \pi(y)\right) &< \varepsilon \quad (0 \leq k < n).
\end{align*}
Here $\lceil u\rceil$ denotes the least integer not less than $u\in \mathbb{R}$.
Roughly speaking, the orbit in $X$ is traced for only $\lceil wn\rceil$ steps,
while the projected orbit in $Y$ is traced for $n$ steps.
In the self-affine example below (Example \ref{example: Bedford--McMullen carpets}), 
this reflects the two different geometric scales.

Let $N$ be a natural number.
Let $\Omega \subset X$ be a subset.
We consider a covering of $\Omega$ by at most countably many $w$-weighted Bowen balls
\[  \Omega \subset \bigcup_{k=1}^\infty B^w_{n_k}(x_k, \varepsilon) \]
satisfying $n_k\geq N$ for all $k$.
For $s\geq 0$, we define $\Lambda_{N,\varepsilon}^{w, s}(\Omega)$ as the infimum of 
\[ \sum_{k=1}^\infty e^{-sn_k} \]
over all such coverings.
The quantity $\Lambda_{N,\varepsilon}^{w, s}(\Omega)$ is monotone in $N$. We define 
\[ \Lambda_\varepsilon^{w, s} (\Omega) = \lim_{N\to \infty} \Lambda_{N, \varepsilon}^{w, s}(\Omega). \]
As the parameter $s$ varies from $0$ to $\infty$, there exists a unique value of $s$
(denoted by $\FHh^w(\Omega, T, \varepsilon)$) at which $\Lambda_\varepsilon^{w, s} (\Omega)$ jumps from 
$\infty$ to $0$:
\[ \Lambda^{w,s}_{\varepsilon}(\Omega) = \begin{cases} \infty & (s <  \FHh^w(\Omega, T, \varepsilon)) \\
                                                                     0        & (s> \FHh^w(\Omega, T, \varepsilon)) \end{cases}. \]
The quantity $\FHh^w(\Omega, T, \varepsilon)$ is monotone in $\varepsilon$.
We define \textbf{Feng--Huang’s $w$-weighted topological entropy of $\Omega$} by 
\[ \FHh^w(\Omega, T) = \lim_{\varepsilon \to 0} \FHh^w(\Omega, T, \varepsilon). \]
Notice that this quantity depends not only on $\Omega$ but also on the map
$\pi\colon (X,T)\to (Y,S)$,
although this dependence is suppressed in the notation.
When we need to clarify the dependence on $\pi$, we will use the notation 
$\FHh^w(\Omega, \pi, T)$.
(The value of $\FHh^w(\Omega, T)$ is independent of the choices of metrics $\mathbf{d}$ and $\mathbf{d}^\prime$.
We assume that $\FHh^w(\Omega, T) = -\infty$ when $\Omega$ is empty.)
We also sometimes denote the $w$-weighted Bowen ball $B^w_n(x, \varepsilon)$ by $B^w_n(x, \pi, \varepsilon)$.

When $\Omega = X$, we denote $\FHh^w(X, T)$ by $\FHh^w(T)$.
Feng and Huang \cite{Feng--Huang} proved the following variational principle \cite[Theorem 1.4]{Feng--Huang}.

\begin{theorem}[Feng--Huang 2016] \label{theorem: Feng--Huang}
\[  \FHh^w(T) = \sup_{\mu\in \mathscr{M}^T(X)} \left(w h_\mu(T) + (1-w) h_{\pi_*\mu}(S)\right). \]
\end{theorem}

\begin{remark} \label{remark: variation of the definition of FH entropy}
In the definition of \(\FHh^w(\Omega,T)\), we consider at most countable coverings of
\(\Omega\) by \(w\)-weighted Bowen balls
\begin{equation}  \label{eq: covering in the definition of FH entropy}
   \Omega \subset \bigcup_{k} B^w_{n_k}(x_k,\varepsilon),
   \qquad n_k\geq N.
\end{equation}
The centers \(x_k\) are not required to belong to \(\Omega\).  
However, one obtains the same value of \(\FHh^w(\Omega,T)\) 
even if the definition is modified by requiring all centers \(x_k\) to belong to \(\Omega\).
Indeed, suppose that
\(\Omega\cap B^w_{n_k}(x_k,\varepsilon)\neq\emptyset\), and choose
\(y_k\in \Omega\cap B^w_{n_k}(x_k,\varepsilon)\).  Then
\[
B^w_{n_k}(x_k,\varepsilon)
\subset
B^w_{n_k}(y_k,2\varepsilon).
\]
Thus any cover as in \eqref{eq: covering in the definition of FH entropy}
can be replaced, after discarding the balls which do not meet \(\Omega\), by a
cover whose centers lie in \(\Omega\), at the cost of replacing
\(\varepsilon\) by \(2\varepsilon\).
Equivalently, if \(\widetilde{\Lambda}^{w,s}_{N,\varepsilon}(\Omega)\) denotes
the version of \(\Lambda^{w,s}_{N,\varepsilon}(\Omega)\) in which the centers of
the Bowen balls are required to lie in \(\Omega\), then
\[
\widetilde{\Lambda}^{w,s}_{N,2\varepsilon}(\Omega)
\leq
\Lambda^{w,s}_{N,\varepsilon}(\Omega)
\leq
\widetilde{\Lambda}^{w,s}_{N,\varepsilon}(\Omega).
\]
Hence the resulting critical value after taking the limit
\(\varepsilon\to0\) is unchanged.  We will use this observation in
Lemma \ref{lemma: reduction to previous Corollary} in \S \ref{subsection: useful corollary}.
\end{remark}

A key motivation for introducing weighted topological entropy lies in its connection to the dimension theory of self-affine carpets,
as illustrated in the following example.
This connection will also play an important role in the present work.

\begin{example}[Bedford--McMullen carpets]  \label{example: Bedford--McMullen carpets}
Let $\mathbb{T} = \mathbb{R}/\mathbb{Z}$ be the circle,
and let $\mathbb{T}^2 = \mathbb{R}^2/\mathbb{Z}^2$ be the torus.
Let $a\geq b>1$ be two integers. 
Define $T\colon \mathbb{T}^2\to \mathbb{T}^2$ and $S\colon \mathbb{T}\to \mathbb{T}$ by 
\[ T(x, y) = (ax, by), \quad S(y) = by. \]
We define an equivariant continuous map $\pi\colon \mathbb{T}^2\to \mathbb{T}$ by $\pi(x, y) = y$.
Let $X \subset \mathbb{T}^2$ be a $T$-invariant closed subset, and set $Y = \pi(X)$.
We set $w := \log_a b$ and consider the weighted topological entropy with respect to the weight $w$.
(The choice $w=\log_a b$ is natural because $a^{-wn}=b^{-n}$,
so the horizontal scale on $X$ and the vertical scale on $Y$ are matched.)
Then it is a direct consequence of the above definition that the weighted topological entropy of 
the map $\pi\colon (X,T) \to (Y,S)$ is connected to the Hausdorff dimension of $X$ by 
\begin{equation} \label{eq: Hausdorff dimension of self-affine carpet}
  \dim_{\mathrm{H}} X = \frac{\FHh^w(X, T)}{\log b}, \quad (w=\log_a b). 
\end{equation}  
Here the Hausdorff dimension is defined by using the standard metric on $\mathbb{T}^2$.
A basic example of $T$-invariant closed subsets is provided by Bedford--McMullen carpets as follows.
Let $D$ be a nonempty subset of $\{0,1,2,\dots, a-1\} \times \{0,1,2,\dots, b-1\}$. 
We define a \textbf{Bedford--McMullen carpet} by 
\[ X = \left\{\left(\sum_{n=1}^\infty \frac{x_n}{a^n}, \sum_{n=1}^\infty\frac{y_n}{b^n}\right)\middle|\, (x_n, y_n)\in D \text{ for all $n\geq1$}\right\}. \]
This is a $T$-invariant closed subset of $\mathbb{T}^2$.
Hence its Hausdorff dimension is connected to the weighted topological entropy by \eqref{eq: Hausdorff dimension of self-affine carpet}.
Bedford \cite{Bedford} and McMullen \cite{McMullen} calculated the Hausdorff dimension of $X$.
Their calculation shows that the Hausdorff dimension and Minkowski dimension do not coincide for $X$ in general. 
For example, let $a=3$, $b=2$ and $D = \{(0,0), (1,1), (2,0)\}$.
\begin{figure}[htbp]
    \centering
    \includegraphics[width=0.45\textwidth]{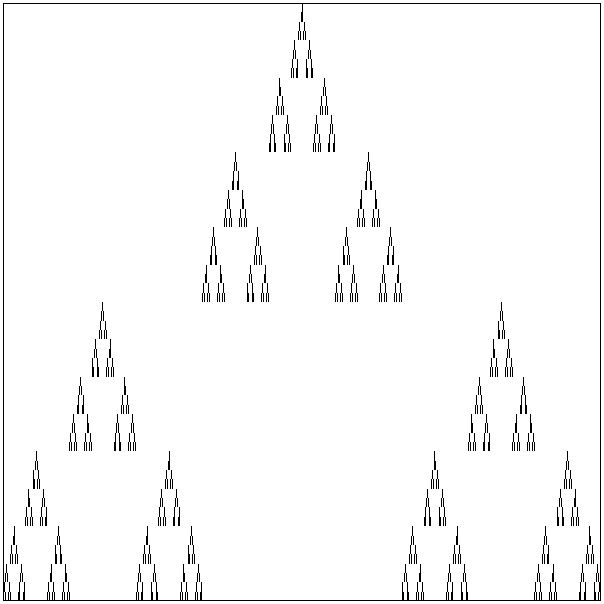}
    \caption{The Bedford--McMullen carpet for
    \(a=3\), \(b=2\) and \(D=\{(0,0),(1,1),(2,0)\}\).}
    \label{fig:BM-carpet-example}
\end{figure}
See Figure \ref{fig:BM-carpet-example}.\footnote{The figures in this paper were prepared with the assistance of
ChatGPT Pro (GPT-5.5).}
In this case, the Hausdorff dimension is given by 
\[ \dim_{\mathrm{H}} X = \log_2\left(1+2^{\log_3 2}\right) = 1.3496838201\dots,  \]
whereas the Minkowski dimension is given by 
$\dim_{\mathrm M} X = 2-\log_3 2 = 1.3690702464\dots$.
\end{example}

At the end of their paper \cite[\S~7.3, p.~441]{Feng--Huang}, 
Feng and Huang asked whether one can establish a relativised version of their variational principle. 
More precisely, their question is as follows.

\begin{problem} \label{problem: relativised}
Let $(X, T)$, $(Y, S)$ and $(Z, R)$ be dynamical systems.
Let $\pi\colon X\to Y$, $\rho\colon X\to Z$ and $\theta\colon Y\to Z$ be equivariant continuous maps such that 
$\rho = \theta \circ \pi$ and that $\rho$ is surjective.
\[
\begin{tikzcd}
(X,T) \arrow[r,"\pi"] \arrow[dr, two heads, "\rho"'] & (Y,S) \arrow[d,"\theta"] \\
& (Z,R)
\end{tikzcd}
\]
Let $0\leq w \leq 1$ and $\nu\in \mathscr{M}^R(Z)$. 
Can one prove that 
\begin{equation} \label{eq: relativised weighted variational principle}
  \begin{split}
   & \int_Z \FHh^w\left(\rho^{-1}(z),\pi, T\right) \, d\nu(z) \\
   & = \sup\left\{w h_\mu(T|R) + (1-w) h_{\pi_*\mu}(S|R) \middle|\, 
    \mu \in \mathscr{M}^T(X)  \text{ with }\rho_*\mu = \nu\right\}? 
   \end{split}
\end{equation}    
Here $h_\mu(T|R) = h_\mu(T) - h_{\nu}(R)$ and $h_{\pi_*\mu}(S|R) = h_{\pi_*\mu}(S) - h_{\nu}(R)$ denote the corresponding 
conditional Kolmogorov--Sinai entropy.
$\FHh^w\left(\rho^{-1}(z), \pi, T\right)$ is the $w$-weighted topological entropy of the set $\rho^{-1}(z)$
with respect to $\pi\colon (X, T)\to (Y, S)$.
\end{problem}

The purpose of the paper is to solve this question affirmatively and apply it to the dimension theory 
of certain fractal sets.

We would like to mention an important paper by Yin \cite{Yin}.
To our knowledge, this is the first paper to address Feng and Huang's question.
The approach of Yin \cite{Yin} and ours are different; in some sense, they are complementary.

We will directly establish a relativised variational principle for $\FHh^w\left(\rho^{-1}(z),\pi, T\right)$,
as this is the quantity most relevant to applications in fractal geometry.
On the other hand, Yin developed a relativised variational principle for an (a priori) different version of weighted topological entropy,
which we will denote by $\h^w\left(\Omega,\pi, T\right)$ in 
\S\ref{subsection: relativised variational principle for weighted topological entropy}.

By combining Yin’s variational principle with ours, we obtain the conclusion that
$\FHh^w\left(\rho^{-1}(z),\pi, T\right) = \h^w\left(\rho^{-1}(z),\pi, T\right)$ for $\nu$-almost every $z \in Z$
in the setting of Problem \ref{problem: relativised}.
The equivalence of these two quantities is a highly nontrivial and useful result.
We will explain this connection in more detail in \S\ref{subsection: relativised variational principle for weighted topological entropy}.

\subsection{Intersection of random translates of Cantor sets} \label{subsection: Intersection of random translates of Cantor sets}

Let 
\[ C = \left\{\sum_{n=1}^\infty \frac{x_n}{3^n} \in \mathbb{R} \middle|\, x_n\in \{0,2\}\right\} \]
be the standard middle-third Cantor set.
Hawkes \cite{Hawkes} discovered the following very curious result:
 \begin{theorem}[Hawkes 1975] \label{theorem: Hawkes}
   For almost every $t\in [0,1]$ with respect to the Lebesgue measure, the Hausdorff dimension of the intersection 
   $C\cap (C+t)$ is given by  
   \[ \dim_\mathrm{H}\left(C\cap (C+t)\right) = \frac{\log_3 2}{3}. \]
   Here $C+t = \{x+t\mid x\in C\}$.
 \end{theorem}
 Notice that $\dim_{\mathrm{H}} C = \log_3 2$. 
 Therefore, by considering the intersection with its own (random) translate, the dimension is reduced by a factor of $1/3$.
For some time, Hawkes’ theorem was regarded as an interesting isolated result.
Kenyon and Peres \cite{Kenyon--Peres_intersection}
placed Hawkes’ result in a broader context, namely that of Lyapunov exponents, and succeeded in greatly generalizing it.

Let $b>1$ be an integer.
Let $D_1, D_2\subset \{0,1,2,\dots, b-1\}$ be nonempty subsets, and define
\[ \Lambda_r = \left\{\sum_{n=1}^\infty \frac{x_n}{b^n} \in \mathbb{T} \middle|\, x_n\in D_r\right\} \quad (r=1,2). \]
Here we have regarded the Cantor set $\Lambda_r$ as a subset of the circle $\mathbb{T} = \mathbb{R}/\mathbb{Z}$.
We define $2\times 2$-matrices $A_0, A_1, \dots, A_{b-1}$ by 
\[ A_\tau(i, j) = \left|(D_1+i+\tau)\cap (D_2+jb)\right| \quad (i, j= 0,1). \]
Here $|\cdot|$ denotes the cardinality and $D_1+i+\tau = \{x+i+\tau\in \mathbb{Z}\mid x \in D_1\}$.
We define $\lambda$ as the top Lyapunov exponent of the unbiased random matrix products of $A_0, A_1, \dots, A_{b-1}$:
\[ \lambda = \lim_{n\to \infty} \frac{1}{n} \log \norm{A_{u_n}A_{u_{n-1}}\cdots A_{u_1}} \]
where $\norm{\cdot}$ is (any) norm of matrices and $u_1, u_2, u_3, \dots$ are i.i.d. random variables that take values in 
$\{0,1,2,\dots, b-1\}$ according to the uniform distribution.
This limit exists and is constant almost surely by the subadditive ergodic theorem.
Kenyon and Peres proved the following beautiful dimension formula \cite[Theorem 1.2]{Kenyon--Peres_intersection}.

\begin{theorem}[Kenyon--Peres 1991] \label{theorem: Kenyon--Peres}
For almost every $t\in \mathbb{T}$ with respect to the Lebesgue measure, the Hausdorff dimension of the intersection 
$(\Lambda_1+t)\cap \Lambda_2$ is given by 
\begin{equation} \label{eq: Kenyon--Peres formula}
    \dim_{\mathrm{H}}\left((\Lambda_1+t)\cap \Lambda_2\right) = \frac{\lambda}{\log b}. 
\end{equation}    
\end{theorem}

A central question motivating the present paper is 
\textit{how to extend the Kenyon--Peres formula \eqref{eq: Kenyon--Peres formula} to the self-affine setting}.
Namely, let $X_1, X_2 \subset \mathbb{T}^2$ be two
Bedford--McMullen carpets as introduced in Example \ref{example: Bedford--McMullen carpets}.
We would like to calculate the Hausdorff dimension of the intersection
\begin{equation*} 
  \dim_{\mathrm{H}}\left((X_1+\mathbf{t})\cap X_2\right) 
\end{equation*}  
for Lebesgue-a.e. $\mathbf{t} \in \mathbb{T}^2$.
(Here $X_1+\mathbf{t} = \{\mathbf{x}+\mathbf{t} \in \mathbb{T}^2\mid \mathbf{x}\in X_1\}$.)
This is a delicate problem since the Hausdorff dimension and Minkowski dimension do not coincide for 
$(X_1+\mathbf{t})\cap X_2$ in general.
(They coincide for $(\Lambda_1+t)\cap \Lambda_2$ in the setting of Theorem \ref{theorem: Kenyon--Peres};
this is an important step in the proof of Kenyon and Peres.)

Examining the argument of Kenyon and Peres (in particular, the proof of \cite[Proposition 2.6]{Kenyon--Peres_intersection}),  
one quickly sees that the relativised variational principle of Ledrappier and Walters \cite{Ledrappier--Walters} plays an important role
in the proof of Theorem \ref{theorem: Kenyon--Peres}.
Therefore, when attempting to extend the Kenyon--Peres formula to the self-affine setting, 
it is natural to expect that a relativised variational principle for weighted topological entropy 
(see Problem \ref{problem: relativised} above) will be useful.
Our strategy is first to solve Problem \ref{problem: relativised} and then to apply it to the analysis of the dimension 
$\dim_{\mathrm{H}}\left((X_1+\mathbf{t})\cap X_2\right)$.
Of course, the study of Problem \ref{problem: relativised} has its own interest and will hopefully find further applications in the future.

We conclude this section by mentioning two earlier approaches to translated intersections of Bedford--McMullen carpets
by Gui and Li \cite{Gui--Li} and Lu, Zou, and Wang \cite{Lu--Zou--Wang}.
Take a digit set $D \subset \{0,1,2,\dots, a-1\}\times \{0,1,2,\dots, b-1\}$, and let
$X$ be the Bedford--McMullen carpet defined by it. 
These two papers study the Hausdorff dimension of the intersection $(X+\mathbf{t})\cap X$.
(Strictly speaking, they considered $X$ as a subset of the plane $\mathbb{R}^2$ not of the torus $\mathbb{T}^2$.
But this difference is minor.)

Gui and Li \cite{Gui--Li} studied the problem through a level-dependent carpet model arising from the digit expansion of the translation, and 
obtained the explicit formula for $\dim_{\mathrm{H}}\left((X+\mathbf{t})\cap X\right)$ 
under suitable separation and coding assumptions\footnote{More specifically, assume
the following three assumptions: 
  \begin{enumerate}
   \item $\norm{d_1-d_2}_\infty\geq 2$ for any distinct $d_1, d_2\in D-D := \{x-y\mid x, y\in D\}$, where $\norm{\cdot}_\infty$ denotes the max norm.
   \item $t\in \mathbb{R}^2$ can be uniquely written as 
   $\mathbf{t}= \left(\sum_{n=1}^\infty \frac{\alpha_n}{a^n}, \sum_{n=1}^\infty\frac{\beta_n}{b^n}\right)$ by
   $(\alpha_n, \beta_n) \in D-D$.
   \item For each $d\in D-D$, the limit $\lim_{N \to \infty} \frac{1}{N}|\{1\leq n \leq N \mid (\alpha_n, \beta_n) = d\}|:=p_d$ exists.
  \end{enumerate}
Under these assumptions, the paper \cite[Corollary 3.3]{Gui--Li} proved that
\[ \dim_{\mathrm{H}}\left((X+\mathbf{t})\cap X\right) = \sum_{d\in D-D} p_d\log_b \left(\sum_{v\in q(D_d)} n_{v,d}^{\log_a b}\right), \]
where $q\colon \mathbb{R}^2\to \mathbb{R}$ is the projection to the second coordinate, $D_d := (D+d)\cap D$ and 
$n_{v,d} := \left|D_d\cap q^{-1}(v)\right|$. 
A beautiful point of their work is that 
the dimension formula is very explicit and easy to estimate.}.  
Their result has a spirit similar to the theorem of Hawkes (Theorem \ref{theorem: Hawkes}).
The difference between their work and ours is that we do not impose any restriction on the digit sets and
consider completely general ones. (On the other hand, the resulting formula is more difficult to estimate than the formula of
Gui and Li.)  

Lu, Zou, and Wang \cite{Lu--Zou--Wang} analyzed the problem from a different angle.
They proved that, for a general digit set $D\subset \{0,1,2,\dots, a-1\}\times \{0,1,2,\dots, b-1\}$, the intersection 
$(X+\mathbf{t})\cap X$ becomes a sofic affine invariant set if $\mathbf{t}\in \mathbb{Q}^2$.
A sofic affine invariant set is a class of fractal sets studied by another paper of Kenyon and Peres \cite{Kenyon--Peres_sofic} and its dimension formula was 
given there.
The point of \cite{Lu--Zou--Wang} is that their result concerns very \textit{structured} translations (that is, rational $\mathbf{t}$), 
whereas ours is formulated for \textit{random} translations.
Thus the two approaches are complementary.

\section{Main results} \label{section: main results}

In this section, we describe our main results.
We present the relativised variational principle for weighted topological entropy and its consequence in 
\S\ref{subsection: relativised variational principle for weighted topological entropy}.
We also explain our calculation of the Hausdorff dimension of the intersection of random translates of Bedford--McMullen carpets 
in \S\ref{subsection: intersecting random translates of Bedford--McMullen carpets}.

\subsection{Relativised variational principle for weighted topological entropy}
\label{subsection: relativised variational principle for weighted topological entropy}

Recall that, for $0\leq w\leq 1$, an equivariant continuous map $\pi\colon (X, T)\to (Y, S)$ between dynamical systems 
and a subset $\Omega \subset X$,
we have denoted by $\FHh^w(\Omega, T)$ (or $\FHh^w(\Omega, \pi,T)$ for clarifying the dependence on $\pi$) 
the $w$-weighted topological entropy of $\Omega$ (see \S \ref{subsection: weighted topological entropy}).
The following theorem is our first main result. This affirmatively solves Problem \ref{problem: relativised}.

\begin{theorem}  \label{theorem: relativised variational principle}
Let $0\leq w\leq 1$.
Let $(X, T)$, $(Y, S)$ and $(Z, R)$ be dynamical systems.
Let $\pi\colon X\to Y$, $\rho\colon X\to Z$ and $\theta\colon Y\to Z$ be equivariant continuous maps such that 
$\rho = \theta \circ \pi$ and that $\rho$ is surjective.
\[
\begin{tikzcd}
(X,T) \arrow[r,"\pi"] \arrow[dr, two heads, "\rho"'] & (Y,S) \arrow[d,"\theta"] \\
& (Z,R)
\end{tikzcd}
\]
For any $\nu\in \mathscr{M}^R(Z)$, we have
\begin{equation*} 
  \begin{split}
   & \int_Z \FHh^w\left(\rho^{-1}(z),\pi, T\right) \, d\nu(z) \\
   & = \sup\left\{w h_\mu(T|R) + (1-w) h_{\pi_*\mu}(S|R) \middle|\, 
    \mu \in \mathscr{M}^T(X)  \text{ with }\rho_*\mu = \nu\right\}.
   \end{split}
\end{equation*}    
\end{theorem}

Notice that $\rho_*\colon \mathscr{M}^T(X) \to \mathscr{M}^R(Z)$ is surjective since $\rho$ is assumed to be surjective.
(This follows from the Hahn--Banach theorem.)
We also remark that $\FHh^w\left(\rho^{-1}(z), \pi, T\right)$ is a measurable function of $z\in Z$.
(See Lemma \ref{lemma: measurability of fiber entropy} in 
\S \ref{section: proof of one direction of the relativised variational principle}.)

In order to state an important corollary of Theorem \ref{theorem: relativised variational principle}, 
we need to review the theory developed by Yin \cite{Yin}, whose approach is based on 
\cite{Barral--Feng_arXiv, Barral--Feng, Tsukamoto}.

Let $(X,T)$ be a dynamical system with a metric $\mathbf{d}$ on $X$.
For $N\geq 1$, we define a new metric $\mathbf{d}_N$ on $X$ by 
\[ \mathbf{d}_N(x, y) = \max_{0\leq n < N} \mathbf{d}\left(T^n x, T^n y\right). \]
We sometimes denote $\mathbf{d}_N$ by $\mathbf{d}^T_N$ to clarify the map $T$.
For a subset $U\subset X$, we denote by $\diam(U, \mathbf{d}_N)$ the diameter of $U$ with respect to $\mathbf{d}_N$.
For a positive number $\varepsilon$ and a subset $\Omega \subset X$, we define the 
\textbf{$\varepsilon$-covering number} $\#\left(\Omega, \mathbf{d}_N, \varepsilon\right)$ as the 
minimum cardinality $n$ of an open covering $\Omega \subset U_1\cup U_2\cup\dots \cup U_n$ that satisfies 
$\diam\left(U_i, \mathbf{d}_N\right) < \varepsilon$ for all $1\leq i \leq n$.
Let $(Y, S)$ be another dynamical system with a metric $\mathbf{d}^\prime$ on $Y$, and suppose that 
we are given an equivariant continuous map $\pi\colon X\to Y$.
Let $0\leq w\leq 1$, $\varepsilon >0$ and $\Omega \subset X$.
We set
\begin{equation*}
 \begin{split}
  & \#^w\left(\Omega, N, \varepsilon\right) \\
  & = \inf\left\{\sum_{j=1}^m \left(\#\left(\Omega \cap \pi^{-1}(V_j), \mathbf{d}_N, \varepsilon\right)\right)^w \middle|\, 
     \parbox{3in}{\centering $V_1, \dots, V_m$ are open sets of $Y$ that satisfy $\pi(\Omega) \subset V_1\cup \dots \cup V_m$ and 
             $\diam(V_j, \mathbf{d}^\prime_N) < \varepsilon$ for all $j$}\right\}. 
  \end{split}           
\end{equation*}             
We define 
\[ \h^w(\Omega, T) = \lim_{\varepsilon \to 0}\left(\limsup_{N\to \infty} \frac{1}{N}\log \#^w\left(\Omega, N, \varepsilon\right)\right). \]
We will occasionally use the notations $\#^w\left(\Omega,\pi, N, \varepsilon\right)$ and $\h^w(\Omega, \pi, T)$ 
to clarify the dependence on the map $\pi$.
When $\Omega= X$, we also denote $\h^w(X, T)$ by $\h^w(T)$.
This quantity was first introduced in \cite{Tsukamoto} as a new approach to the weighted variational principle.
(For symbolic systems, it was first considered by Barral and Feng \cite{Barral--Feng_arXiv, Barral--Feng}.)
There is no reason to hope that $\h^w(\Omega, T)$ coincides with $\FHh^w(\Omega, T)$; 
indeed they do not coincide in general even for a (non-invariant) closed subset $\Omega$.
However the paper \cite{Tsukamoto} proved that 
\[  \h^w(T) = \sup_{\mu\in \mathscr{M}^T(X)} \left(w h_\mu(T) + (1-w) h_{\pi_*\mu}(S)\right). \]
In particular $\h^w(T) = \FHh^w(T)$, and hence we have $\h^w(\Omega, T) = \FHh^w(\Omega, T)$ 
for any closed $T$-invariant subsets $\Omega$.

Yin \cite[Theorem 2.9]{Yin} proved the following theorem.

\begin{theorem}[Yin 2025] \label{theorem: Yin}
Let $0\leq w\leq 1$.
Let $(X, T)$, $(Y, S)$ and $(Z, R)$ be dynamical systems.
Let $\pi\colon X\to Y$, $\rho\colon X\to Z$ and $\theta\colon Y\to Z$ be equivariant continuous maps such that 
$\rho = \theta \circ \pi$ and that $\rho$ is surjective.
For any $\nu\in \mathscr{M}^R(Z)$, we have
\begin{equation*} 
  \begin{split}
   & \int_Z \h^w\left(\rho^{-1}(z),\pi, T\right) \, d\nu(z) \\
   & = \sup\left\{w h_\mu(T|R) + (1-w) h_{\pi_*\mu}(S|R) \middle|\, 
    \mu \in \mathscr{M}^T(X)  \text{ with }\rho_*\mu = \nu\right\}.
   \end{split}
\end{equation*}    
\end{theorem}

Combining this theorem with ours (Theorem \ref{theorem: relativised variational principle}), 
we obtain the following corollary.

\begin{corollary} \label{corollary: equivalence of two approaches}
Let $0\leq w\leq 1$.
Let $(X, T)$, $(Y, S)$ and $(Z, R)$ be dynamical systems.
Let $\pi\colon X\to Y$, $\rho\colon X\to Z$ and $\theta\colon Y\to Z$ be equivariant continuous maps such that 
$\rho = \theta \circ \pi$ and that $\rho$ is surjective.
For any $\nu\in \mathscr{M}^R(Z)$, we have
\[ \FHh^w(\rho^{-1}(z),\pi, T) = \h^w(\rho^{-1}(z),\pi, T) \]
for $\nu$-almost every point $z\in Z$.
\end{corollary}

\begin{proof}
Set $\Omega = \{z\in Z\mid \FHh^w(\rho^{-1}(z),\pi, T) \neq \h^w(\rho^{-1}(z),\pi, T)\}$.
When $\nu$ is ergodic, both $\FHh^w(\rho^{-1}(z), \pi, T)$ and $\h^w(\rho^{-1}(z), \pi, T)$ are constant $\nu$-almost everywhere.
(See Lemma \ref{lemma: measurability of fiber entropy} and Remark \ref{remark: R-invariance of combinatorial weighted entropy} in 
\S \ref{section: proof of one direction of the relativised variational principle}.)
Therefore Theorems \ref{theorem: relativised variational principle} and \ref{theorem: Yin} imply
$\FHh^w(\rho^{-1}(z),\pi, T) = \h^w(\rho^{-1}(z),\pi, T)$ for $\nu$-almost every $z\in Z$.
Namely we have $\nu(\Omega) = 0$.
This holds for every ergodic measure $\nu\in \mathscr{M}^R(Z)$.
Then it follows from the ergodic decomposition theorem that $\nu(\Omega)=0$ for 
every (not necessarily ergodic) $\nu\in \mathscr{M}^R(Z)$.
\end{proof}

Corollary \ref{corollary: equivalence of two approaches} is very useful because the two quantities are adapted to different purposes.
On the one hand, the quantity $\FHh^w(\rho^{-1}(z),\pi, T)$ is intrinsically related to Hausdorff dimension since its definition 
mimics that of Hausdorff dimension.
On the other hand, $\h^w(\rho^{-1}(z),\pi,T)$ is more combinatorial in nature
and is often easier to estimate.
Therefore the identity $\FHh^w(\rho^{-1}(z),\pi, T) = \h^w(\rho^{-1}(z),\pi, T)$ allows one to transfer information
between the geometric and combinatorial settings.
We will illustrate this idea by an application to the intersection of random translates of Bedford--McMullen carpets.

The original paper of Feng and Huang \cite{Feng--Huang} introduced not only the weighted topological entropy but also 
the weighted topological \textit{pressure}.
In the present paper, we focus on the case of entropy because this is the most relevant to the application in our mind.
However, we believe that our relativised variational principle can be generalized to the case of pressure.
(Yin \cite{Yin} studied the weighted pressure in his setting.)
Interested readers may pursue this direction.

\subsection{Intersecting random translates of Bedford--McMullen carpets}
\label{subsection: intersecting random translates of Bedford--McMullen carpets}

In this subsection, we extend the Kenyon--Peres formula to the self-affine setting.
Let $a\geq b >1$ be two integers, and set $A = \{0,1,2,\dots, a-1\}$ and $B = \{0,1,2,\dots,b-1\}$.
Let $D_1$ and $D_2$ be nonempty subsets of $A\times B$.
Let $X_r$ $(r=1,2)$ be the Bedford--McMullen carpets defined by $D_r$:
\[ X_r = \left\{\left(\sum_{n=1}^\infty \frac{x_n}{a^n}, \sum_{n=1}^\infty\frac{y_n}{b^n}\right) \in \mathbb{T}^2\middle|\, 
    (x_n, y_n)\in D_r \text{ for all $n\geq1$}\right\}. \]
Notice that this is defined as a subset of the torus $\mathbb{T}^2 = \mathbb{R}^2/\mathbb{Z}^2$.

To state our dimension formula, we need to introduce a family of $4\times4$-matrices 
whose rows and columns are indexed by the set
$I := \{(0,0),(0,1),(1,0),(1,1)\}$.
Let $q\colon \mathbb{R}^2\to \mathbb{R}$ be the projection to the second coordinate.
For $\tau\in A\times B$ and $v\in B$, 
we define a $4\times 4$-matrix $Q_{\tau, v}$ by 
\[ Q_{\tau, v}\bigl((i,j), (k,\ell)\bigr) 
   = \left|\left(D_1+\tau+(i,j)\right)\cap \left(D_2+(ka,\ell b)\right)\cap q^{-1}(\ell b+ v)\right|, \]
where $(i,j), (k,\ell)\in I$.   
Here $Q_{\tau, v}\bigl((i,j), (k,\ell)\bigr)$ denotes the $\left((i,j), (k,\ell)\right)$-entry of $Q_{\tau, v}$,
namely 
 \[
Q_{\tau, v} =
\begin{pmatrix}
Q_{\tau,v}\bigl((0,0),(0,0)\bigr) & Q_{\tau, v}\bigl((0,0),(0,1)\bigr) & Q_{\tau, v}\bigl((0,0),(1,0)\bigr) & Q_{\tau, v}\bigl((0,0),(1,1)\bigr) \\
Q_{\tau, v}\bigl((0,1),(0,0)\bigr) & Q_{\tau, v}\bigl((0,1),(0,1)\bigr) & Q_{\tau, v}\bigl((0,1),(1,0)\bigr) & Q_{\tau, v}\bigl((0,1),(1,1)\bigr) \\
Q_{\tau, v}\bigl((1,0),(0,0)\bigr) & Q_{\tau, v}\bigl((1,0),(0,1)\bigr) & Q_{\tau, v}\bigl((1,0),(1,0)\bigr) & Q_{\tau, v}\bigl((1,0),(1,1)\bigr) \\
Q_{\tau, v}\bigl((1,1),(0,0)\bigr) & Q_{\tau, v}\bigl((1,1),(0,1)\bigr) & Q_{\tau, v}\bigl((1,1),(1,0)\bigr) & Q_{\tau, v}\bigl((1,1),(1,1)\bigr)
\end{pmatrix}.
\]
The binary indices $(i, j), (k,\ell)\in I$ record the carry parameters that arise in the
digit-by-digit description of the intersection.
See Examples \ref{example: sharing positive eigenvector} and
 \ref{example: numerical calculation} below for concrete examples of these matrices.

Let $\tau_1, \tau_2,\tau_3, \dots$ be i.i.d. random variables that take values in $A\times B$ according to the uniform distribution.
We consider the following limit:
\begin{equation} \label{eq: new random matrix product}
 \lambda = \lim_{n\to \infty} 
 \frac{1}{n}\log\left(\sum_{v_1, \dots, v_n\in B} \norm{Q_{\tau_n, v_n} Q_{\tau_{n-1}, v_{n-1}}\cdots Q_{\tau_1, v_1}}^{\log_a b}\right).
\end{equation} 
This limit almost surely exists and is constant by the subadditive ergodic theorem.
We denote this constant by $\lambda$.

Our second main result is the following theorem.

\begin{theorem}  \label{theorem: intersection of Bedford--McMullen carpets}
For almost every $\mathbf{t} \in \mathbb{T}^2$ with respect to the Lebesgue measure, the Hausdorff dimension of the intersection 
$(X_1+\mathbf{t})\cap X_2$ is given by 
\begin{equation} \label{eq: Hausdorff dimension of intersection of carpets}
   \dim_{\mathrm{H}}\left((X_1+\mathbf{t})\cap X_2\right) = \frac{\lambda}{\log b}. 
\end{equation}   
\end{theorem}

Here is the outline of the proof.
The key point is that the same fibre entropy admits both a geometric and a combinatorial description.
For $r=1,2$ we define $T_r\colon X_r \to X_r$ by $T_r(x_1, x_2) = (ax_1, bx_2)$, and we set
$T= T_1\times T_2\colon X_1\times X_2\to X_1\times X_2$.
Let $\pi_r\colon X_r\to \mathbb{T}$ be the projection to the second coordinate $(r=1,2)$, 
and set $Y_r = \pi_r(X_r)$.
We define $\pi\colon X_1\times X_2\to Y_1\times Y_2$ and $\rho\colon X_1\times X_2\to \mathbb{T}^2$ by 
\[ \pi(\mathbf{x}, \mathbf{y}) = \left(\pi_1(\mathbf{x}), \pi_2(\mathbf{y})\right), \quad \rho(\mathbf{x}, \mathbf{y}) = \mathbf{y}-\mathbf{x}.\]
Notice that the fiber $\rho^{-1}(\mathbf{t})$ is naturally identified with the intersection $(X_1+\mathbf{t})\cap X_2$.
Set $w= \log_a b$ and apply (a slight variant of) Corollary \ref{corollary: equivalence of two approaches} to 
$X:=X_1\times X_2$, $Y:= Y_1\times Y_2$ and $Z:= \mathbb{T}^2$.
(Indeed, the current setting is different from that of Corollary \ref{corollary: equivalence of two approaches}
because there is no relevant map from $Y_1\times Y_2$ to $\mathbb{T}^2$.
We will explain this modification in Corollary \ref{corollary: geometric and combinatorial entropies}.)
Then we obtain
\[ \FHh^w\left(\rho^{-1}(\mathbf{t}),\pi, T\right) = \h^w\left(\rho^{-1}(\mathbf{t}),\pi, T\right) \]
for almost every $\mathbf{t}\in \mathbb{T}^2$.
The theorem therefore follows from the two identities
\[
\FHh^w(\rho^{-1}(\mathbf{t}),\pi,T)
 =(\log b)\dim_{\mathrm H}((X_1+\mathbf{t})\cap X_2)
\]
and
\[
\h^w(\rho^{-1}(\mathbf{t}),\pi,T)=\lambda \quad (\text{for almost every $\mathbf{t}\in \mathbb{T}^2$}).
\]
The first identity is geometric and immediately follows from the definition, 
while the second one is combinatorial and comes from a
digit-by-digit description of the intersection.
Indeed, writing
\[
\mathbf{t}=\left(\sum_{n=1}^\infty \frac{\alpha_n}{a^n},
\sum_{n=1}^\infty \frac{\beta_n}{b^n}\right),
\]
the digits $\tau_n=(\alpha_n,\beta_n)\in A\times B$ are i.i.d.\ and uniformly distributed
for Lebesgue-a.e.\ $\mathbf{t}$, and the corresponding overlap transitions are encoded by the
matrices $Q_{\tau_n,v_n}$.

\begin{remark}
In the setting of Theorem \ref{theorem: intersection of Bedford--McMullen carpets},
the intersection $(X_1+\mathbf{t})\cap X_2$ may become empty.
For dealing with this case, we assume that the Hausdorff dimension of the empty set is $-\infty$.
Then the identity \eqref{eq: Hausdorff dimension of intersection of carpets} still holds for this case.
Namely, if $(X_1+\mathbf{t})\cap X_2$ is empty for Lebesgue almost every $\mathbf{t}$ then 
\[ \sum_{v_1, \dots, v_n\in B} \norm{Q_{\tau_n, v_n} Q_{\tau_{n-1}, v_{n-1}}\cdots Q_{\tau_1, v_1}}^{\log_a b}=0 \]
almost surely for sufficiently large $n$. (We assume that $\log 0 = -\infty$.)
See Example \ref{example: finite time extinction} below for an interesting example.
\end{remark}

Theorem \ref{theorem: intersection of Bedford--McMullen carpets} can be seen as a generalization of the Kenyon--Peres formula.
For $\tau \in A\times B$, we set 
\begin{equation} \label{eq: Q_tau in introduction}
  Q_\tau = \sum_{v\in B} Q_{\tau, v}.
\end{equation}  
This is a $4\times 4$-matrix whose entries are given by
\[ Q_\tau\left((i,j), (k,\ell)\right) =\left|\left(D_1+\tau+(i,j)\right)\cap \left(D_2+(ka,\ell b)\right)\right| . \]
Since we have $(x+y)^w \leq x^w + y^w$ for nonnegative numbers $x$ and $y$, 
the quantity $\lambda$ in \eqref{eq: new random matrix product}
is bounded from below by 
\[ \frac{\lambda}{\log_a b} \geq  \lim_{n\to \infty} \frac{1}{n} \log \norm{Q_{\tau_n} Q_{\tau_{n-1}}\dots Q_{\tau_1}}, \]
where the right-hand side is the top Lyapunov exponent of the unbiased random matrix products of $Q_\tau$ $(\tau\in A\times B)$.
Moreover, when 
$a=b$ (namely, in the self-similar case), this inequality becomes an equality, 
and the above dimension formula \eqref{eq: Hausdorff dimension of intersection of carpets} reduces to 
the planar case of the Kenyon--Peres formula \cite[Corollary~5.8]{Kenyon--Peres_intersection}, 
discussed in \S\ref{subsection: Intersection of random translates of Cantor sets}.

In the case of $a>b$, the calculation of the limit \eqref{eq: new random matrix product} does not fall 
within the standard framework of Lyapunov exponents.
Indeed, the expression involves a sum of powers of norms of matrix products,
rather than the norm of a single random product,
and therefore lies outside the usual multiplicative-ergodic framework.
The rigorous computation of \(\lambda\) appears to lead to an interesting problem
in the theory of random matrix products.
Of course, one can numerically estimate the value of $\lambda$ by Monte Carlo methods.
However, a computation accompanied by certified error bounds requires new ideas.

The following is a simple, yet interesting, example.

\begin{example}[Finite-time extinction] \label{example: finite time extinction}
Let \(a=3\), \(b=2\), and
\[
D_1=D_2=\{(0,0),(1,1),(2,0)\}.
\]
Probably this is the most famous example of Bedford--McMullen carpets.
See Figure \ref{fig:BM-carpet-example} in \S \ref{subsection: weighted topological entropy}.
We have $\dim_{\mathrm{H}} X_1 = \dim_{\mathrm H} X_2 = \log_2\left(1+2^{\log_3 2}\right) \approx 1.34968$.
We show that, in this example, the generic intersection is empty:
\[
(X_1+\mathbf t)\cap X_2=\emptyset
\qquad
\text{for Lebesgue-a.e. } \mathbf t\in\mathbb T^2.
\]
With respect to the ordering
\(I=\{(0,0),(0,1),(1,0),(1,1)\}\), a direct computation gives
\[
Q_{(0,0)}
=
\begin{pmatrix}
3&0&0&0\\
0&0&0&0\\
0&0&1&0\\
1&1&0&0
\end{pmatrix},
\qquad
Q_{(0,1)}
=
\begin{pmatrix}
0&0&0&0\\
0&3&0&0\\
1&1&0&0\\
0&0&0&1
\end{pmatrix},
\]
where \(Q_\tau\) is the matrix defined by \eqref{eq: Q_tau in introduction}.
These two matrices satisfy
\[
Q_{(0,1)}Q_{(0,0)}Q_{(0,1)}Q_{(0,0)}=0.
\]
\(\tau_1,\tau_2, \tau_3, \dots\) are i.i.d. and uniformly distributed on
\(A\times B\).
Hence by the Borel--Cantelli lemma, the word
\[
(0,0),(0,1),(0,0),(0,1)
\]
appears in the sequence \(\tau_1,\tau_2, \tau_3, \dots\) almost surely.
Whenever this word appears, the corresponding product of the matrices
\(Q_\tau\) contains the zero block
\(Q_{(0,1)}Q_{(0,0)}Q_{(0,1)}Q_{(0,0)}\)
and therefore
\(Q_{\tau_n}Q_{\tau_{n-1}}\cdots Q_{\tau_1}=0\)
for all sufficiently large \(n\). 
It follows that
\[
Q_{\tau_n,v_n}Q_{\tau_{n-1},v_{n-1}}\cdots Q_{\tau_1,v_1}=0
\]
for every \(v_1,\dots,v_n\in B\), for all sufficiently large \(n\).
Thus 
\(\lambda=-\infty\).
Consequently, under the convention \(\dim_{\mathrm H}\emptyset=-\infty\),
Theorem \ref{theorem: intersection of Bedford--McMullen carpets} gives
\[
\dim_{\mathrm H}\bigl((X_1+\mathbf t)\cap X_2\bigr)=-\infty
\]
for Lebesgue-a.e. \(\mathbf t\in\mathbb T^2\). Equivalently,
\( (X_1+\mathbf t)\cap X_2=\emptyset \)
for Lebesgue-a.e. \(\mathbf t\in\mathbb T^2\).
This example shows that Theorem \ref{theorem: intersection of Bedford--McMullen carpets} can be used to 
detect the empty intersection \( (X_1+\mathbf t)\cap X_2=\emptyset \).
Although this is only a superficial application of the theorem, 
it is sometimes useful.
\end{example}

In general, computing $\lambda$ is highly nontrivial.
However, the following example illustrates an exceptional case where $\lambda$ can be computed explicitly.

\begin{figure}[htbp]
    \centering
    \includegraphics[width=0.85\textwidth]{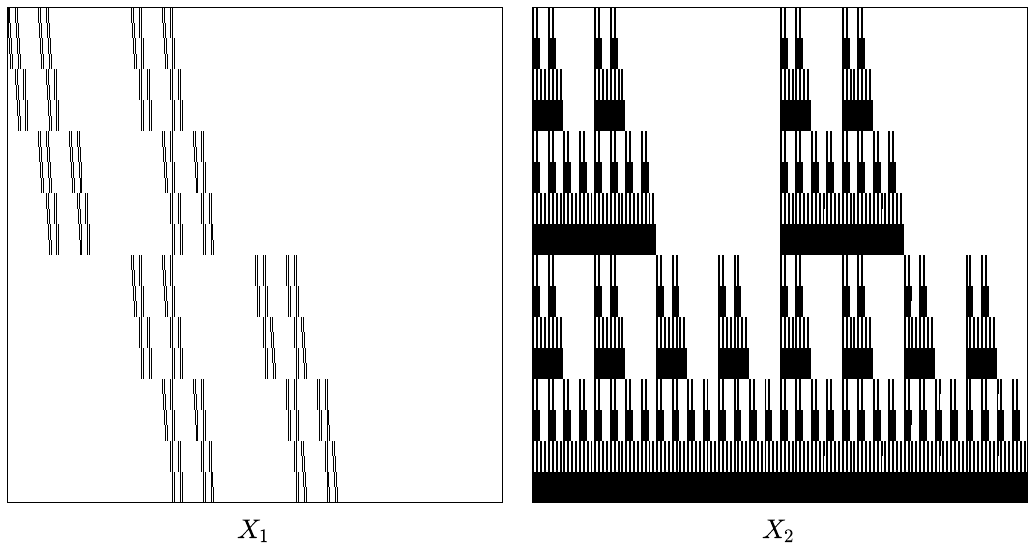}
    \caption{The Bedford--McMullen carpets $X_1$ and $X_2$ for
    \(a=4\), \(b=2\) and the digit sets \(D_1 = \{(1,0),(2,0),(0,1),(1,1)\}\) and \(D_2 = \{(0,0), (1,0), (2,0), (3,0), (0,1), (2,1)\}\).}
    \label{fig:BM-carpet-example-Perron}
\end{figure}

\begin{example}[Sharing positive eigenvector]  \label{example: sharing positive eigenvector}
The constant $\lambda$ can be easily computed if the matrices $Q_{\tau, v}$ share a common positive eigenvector (i.e. Perron eigenvector).
Namely, if there exists a component-wise positive vector $u\in \mathbb{R}^4$ for which 
\[ \forall \tau\in A\times B,\, \forall v\in B: \quad  Q_{\tau, v} u = c_{\tau, v} u \quad (c_{\tau, v}>0), \]
then $\norm{Q_{\tau_n, v_n}\cdots Q_{\tau_1, v_1}} \asymp \prod_{k=1}^n c_{\tau_k, v_k}$ and hence 
$\lambda = \frac{1}{ab} \sum_{\tau\in A\times B} \log \left(\sum_{v\in B} c_{\tau, v}^{\log_a b}\right)$.
In fact, one can prove that, if such a vector $u$ exists, then the eigenvector $u$ and eigenvalues $c_{\tau, v}$ must have the forms
$u = s(1,1,1,1)^T$ for some $s>0$ and $c_{\tau, v} = c_v$. 
In particular, $\lambda = \log \left(\sum_{v\in B} c_{v}^{\log_a b}\right)$ and hence 
\[ \dim_{\mathrm H} \bigl((X_1+\mathbf t)\cap X_2\bigr)
= \log_b \left(\sum_{v\in B} c_v^{\log_a b}\right) \quad  \text{for Lebesgue a.e. } \mathbf{t} \in \mathbb{T}^2. \]
(The above fact follows from the combinatorial structure of $Q_{\tau, v}$.
The details of the proof would take us rather far from the main purpose of the present paper. 
We therefore omit the details here. 
We plan to give them in a subsequent paper.)
More concretely, 
let 
$a=4$, $b=2$ and
\[ D_1 =\{(1,0),(2,0),(0,1),(1,1)\}, \quad 
D_2 =\{(0,0),(1,0),(2,0),(3,0),(0,1),(2,1)\}. \]
See Figure \ref{fig:BM-carpet-example-Perron}.
In this case we have 
$\dim_{\mathrm H} X_1 = \dim_{\mathrm M} X_1 = 3/2$,
$\dim_{\mathrm H} X_2 = \log_2 (2+ \sqrt{2}) \approx 1.771$ and 
$\dim_{\mathrm M} X_2 = 1 + \log_4 3 \approx 1.792$.
With respect to the ordering \(I=\{(0,0),(0,1),(1,0),(1,1)\}\), 
the sixteen matrices \(Q_{\tau,v}\) are given as follows:
{\scriptsize
\[
\begin{array}{cccc}
Q_{(0,0),0}=
\left(\begin{smallmatrix}
2&0&0&0\\
0&2&0&0\\
2&0&0&0\\
0&2&0&0
\end{smallmatrix}\right)
&
Q_{(0,0),1}=
\left(\begin{smallmatrix}
1&0&0&0\\
1&0&0&0\\
1&0&0&0\\
1&0&0&0
\end{smallmatrix}\right)
&
Q_{(0,1),0}=
\left(\begin{smallmatrix}
0&2&0&0\\
0&2&0&0\\
0&2&0&0\\
0&2&0&0
\end{smallmatrix}\right)
&
Q_{(0,1),1}=
\left(\begin{smallmatrix}
1&0&0&0\\
0&1&0&0\\
1&0&0&0\\
0&1&0&0
\end{smallmatrix}\right)
\\[2.5em]
Q_{(1,0),0}=
\left(\begin{smallmatrix}
2&0&0&0\\
0&2&0&0\\
1&0&1&0\\
0&2&0&0
\end{smallmatrix}\right)
&
Q_{(1,0),1}=
\left(\begin{smallmatrix}
1&0&0&0\\
1&0&0&0\\
1&0&0&0\\
0&0&1&0
\end{smallmatrix}\right)
&
Q_{(1,1),0}=
\left(\begin{smallmatrix}
0&2&0&0\\
0&2&0&0\\
0&2&0&0\\
0&1&0&1
\end{smallmatrix}\right)
&
Q_{(1,1),1}=
\left(\begin{smallmatrix}
1&0&0&0\\
0&1&0&0\\
0&0&1&0\\
0&1&0&0
\end{smallmatrix}\right)
\\[2.5em]
Q_{(2,0),0}=
\left(\begin{smallmatrix}
1&0&1&0\\
0&2&0&0\\
0&0&2&0\\
0&1&0&1
\end{smallmatrix}\right)
&
Q_{(2,0),1}=
\left(\begin{smallmatrix}
1&0&0&0\\
0&0&1&0\\
0&0&1&0\\
0&0&1&0
\end{smallmatrix}\right)
&
Q_{(2,1),0}=
\left(\begin{smallmatrix}
0&2&0&0\\
0&1&0&1\\
0&1&0&1\\
0&0&0&2
\end{smallmatrix}\right)
&
Q_{(2,1),1}=
\left(\begin{smallmatrix}
0&0&1&0\\
0&1&0&0\\
0&0&1&0\\
0&0&0&1
\end{smallmatrix}\right)
\\[2.5em]
Q_{(3,0),0}=
\left(\begin{smallmatrix}
0&0&2&0\\
0&1&0&1\\
0&0&2&0\\
0&0&0&2
\end{smallmatrix}\right)
&
Q_{(3,0),1}=
\left(\begin{smallmatrix}
0&0&1&0\\
0&0&1&0\\
0&0&1&0\\
0&0&1&0
\end{smallmatrix}\right)
&
Q_{(3,1),0}=
\left(\begin{smallmatrix}
0&1&0&1\\
0&0&0&2\\
0&0&0&2\\
0&0&0&2
\end{smallmatrix}\right)
&
Q_{(3,1),1}=
\left(\begin{smallmatrix}
0&0&1&0\\
0&0&0&1\\
0&0&1&0\\
0&0&0&1
\end{smallmatrix}\right).
\end{array}
\]
}
These matrices share a common positive eigenvector\footnote{We note here a curious fact. 
If we interchange $D_1$ and $D_2$, then the resulting matrices $Q_{\tau,v}$ no longer share a common positive eigenvector. 
On the other hand, the dimension of the generic intersection
$\dim_{\mathrm H}\bigl((X_1+\mathbf{t})\cap X_2\bigr)$
is invariant under this interchange. 
This is obvious from the geometric viewpoint. 
However, this symmetry is hidden at the level of the matrices $Q_{\tau,v}$.}:
Let $u = (1,1,1,1)^T$. Then 
\[ \forall \tau \in A\times B: \quad  Q_{\tau, 0} u = 2 u, \quad Q_{\tau, 1} u = u. \]
Therefore we have $\lambda = \log (1+\sqrt{2})$ and 
\[ \dim_{\mathrm H}
\bigl((X_1+\mathbf t)\cap X_2\bigr)
= \log_2 \left(1+ \sqrt{2}\right) \approx 1.27155 \quad  \text{for Lebesgue a.e. } \mathbf{t} \in \mathbb{T}^2. \]
Interestingly, this implies that
\[ \dim_{\mathrm H} \bigl((X_1+\mathbf t)\cap X_2\bigr) = \dim_{\mathrm H} X_1 + \dim_{\mathrm H} X_2 -2
   \quad \text{for Lebesgue a.e. $\mathbf{t} \in \mathbb{T}^2$}.  \]
This means that the bound in the Marstrand slicing theorem\footnote{The Marstrand slicing theorem and product theorem 
(\cite[Theorem 1.6.1 and Theorem 3.2.1]{Bishop--Peres}) provide
\[ \dim_{\mathrm H} \bigl((X_1+\mathbf t)\cap X_2\bigr) \leq \dim_{\mathrm M} X_1 + \dim_{\mathrm H} X_2 -2 \]
for Lebesgue a.e. $\mathbf{t}\in \mathbb{T}^2$. 
We have $\dim_{\mathrm M} X_1 = \dim_{\mathrm H} X_1$ for this example.} 
is attained for these sets $X_1$ and $X_2$.
This phenomenon is not accidental; one can show that any example with a common positive eigenvector has the same property.
We also plan to give a detailed account of this in a subsequent paper.
\end{example}

\begin{figure}[htbp]
    \centering
    \includegraphics[width=0.45\textwidth]{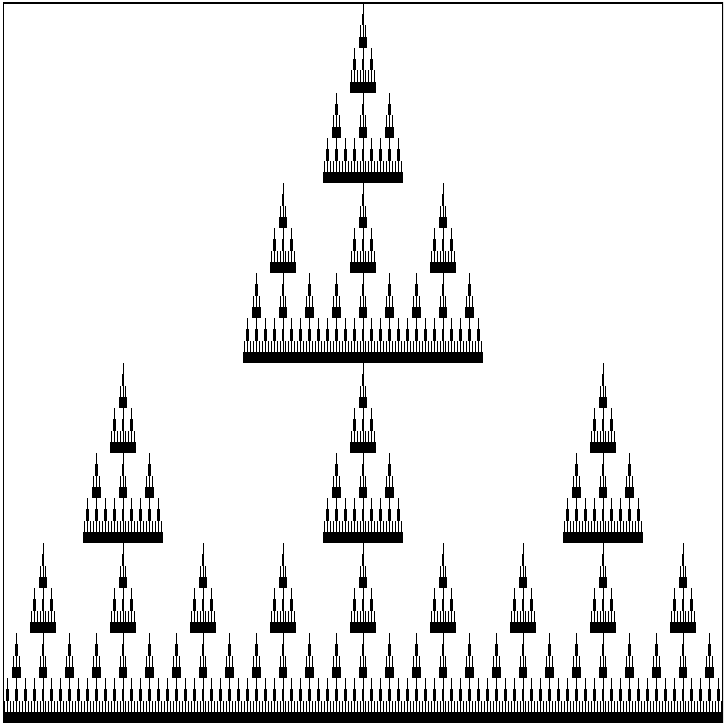}
    \caption{The Bedford--McMullen carpet for
    \(a=3\), \(b=2\) and the digit set \(\{(0,0),(1,0),(2,0),(1,1)\}\).}
    \label{fig:BM-carpet-example-positive}
\end{figure}

Apart from special situations such as finite-time extinction and the common
positive eigenvector case, a rigorous computation of \(\lambda\)
appears to be highly nontrivial.  We next present an example for which numerical
computations suggest that the generic intersection has positive Hausdorff
dimension.

\begin{example}[Numerical example] \label{example: numerical calculation}
Let \(a=3\), \(b=2\), and 
\[ D_1=D_2=\{(0,0),(1,0),(2,0),(1,1)\}. \] 
See Figure \ref{fig:BM-carpet-example-positive}.
In this case, $\dim_{\mathrm H} X_1 = \dim_{\mathrm H} X_2 = \log_2 3 \approx 1.585$
and $\dim_{\mathrm M} X_1 = \dim_{\mathrm M} X_2 = 1 + \log_3 2 \approx 1.6309$.
With respect to the ordering
\(I=\{(0,0),(0,1),(1,0),(1,1)\}\), 
the twelve matrices \(Q_{\tau,v}\) are given as follows:
{\scriptsize
\[
\begin{array}{cccc}
Q_{(0,0),0}=
\left(\begin{smallmatrix}
3&0&0&0\\
0&1&0&0\\
2&0&1&0\\
0&1&0&0
\end{smallmatrix}\right)
&
Q_{(0,0),1}=
\left(\begin{smallmatrix}
1&0&0&0\\
1&0&0&0\\
0&0&0&0\\
1&0&0&0
\end{smallmatrix}\right)
&
Q_{(0,1),0}=
\left(\begin{smallmatrix}
0&1&0&0\\
0&3&0&0\\
0&1&0&0\\
0&2&0&1
\end{smallmatrix}\right)
&
Q_{(0,1),1}=
\left(\begin{smallmatrix}
1&0&0&0\\
0&1&0&0\\
1&0&0&0\\
0&0&0&0
\end{smallmatrix}\right)
\\[2.5em]
Q_{(1,0),0}=
\left(\begin{smallmatrix}
2&0&1&0\\
0&1&0&0\\
1&0&2&0\\
0&0&0&1
\end{smallmatrix}\right)
&
Q_{(1,0),1}=
\left(\begin{smallmatrix}
0&0&0&0\\
1&0&0&0\\
0&0&0&0\\
0&0&1&0
\end{smallmatrix}\right)
&
Q_{(1,1),0}=
\left(\begin{smallmatrix}
0&1&0&0\\
0&2&0&1\\
0&0&0&1\\
0&1&0&2
\end{smallmatrix}\right)
&
Q_{(1,1),1}=
\left(\begin{smallmatrix}
1&0&0&0\\
0&0&0&0\\
0&0&1&0\\
0&0&0&0
\end{smallmatrix}\right)
\\[2.5em]
Q_{(2,0),0}=
\left(\begin{smallmatrix}
1&0&2&0\\
0&0&0&1\\
0&0&3&0\\
0&0&0&1
\end{smallmatrix}\right)
&
Q_{(2,0),1}=
\left(\begin{smallmatrix}
0&0&0&0\\
0&0&1&0\\
0&0&1&0\\
0&0&1&0
\end{smallmatrix}\right)
&
Q_{(2,1),0}=
\left(\begin{smallmatrix}
0&0&0&1\\
0&1&0&2\\
0&0&0&1\\
0&0&0&3
\end{smallmatrix}\right)
&
Q_{(2,1),1}=
\left(\begin{smallmatrix}
0&0&1&0\\
0&0&0&0\\
0&0&1&0\\
0&0&0&1
\end{smallmatrix}\right).
\end{array}
\]
}
These matrices do not share a common positive eigenvector.
A numerical computation\footnote{The numerical computation was carried out by simulating long typical sequences
\[
\tau_1,\tau_2,\dots \in A\times B
\]
with respect to the uniform Bernoulli measure, and estimating the corresponding logarithmic growth rate
\[
\frac{1}{n}\log \sum_{v_1,\dots,v_n\in B}
\left\|Q_{\tau_n,v_n}\cdots Q_{\tau_1,v_1}\right\|^{\log_3 2}.
\]
Since the sum over all \(v_1,\dots,v_n\) has \(2^n\) terms, we approximated this quantity by a sequential Monte Carlo method.} 
suggests that 
\[ \lambda \approx 0.775. \] 
Hence 
\begin{equation} \label{eq: numerical approximation of dimension}
  \dim_{\mathrm H}\bigl((X_1+\mathbf t)\cap X_2\bigr) = \frac{\lambda}{\log 2} \approx 1.12 
\end{equation}  
for Lebesgue-a.e. \(\mathbf t\in\mathbb T^2\).
We emphasize that this value is only a numerical approximation; no rigorous
error bound is claimed here.
The Marstrand slicing theorem and product theorem (\cite[Theorem 1.6.1 and Theorem 3.2.1]{Bishop--Peres})
provide
\[ \dim_{\mathrm H}\bigl((X_1+\mathbf t)\cap X_2\bigr) \leq \dim_{\mathrm H} X_1 + \dim_{\mathrm M} X_2 -2 
  = \log_2 3 + \log_3 2-1 \approx 1.215 \]
for Lebesgue a.e. $\mathbf{t} \in \mathbb{T}^2$.
Thus, assuming that the numerical approximation \eqref{eq: numerical approximation of dimension} is accurate, 
the dimension of the generic intersection is strictly smaller than the upper bound provided by Marstrand's theorems.
\end{example}

The present paper focuses on the general theory, and hence the rigorous computation of 
$\lambda$ is postponed to future work.
We hope to return to this problem in the near future.

\begin{remark}
The dimension formula \eqref{eq: Hausdorff dimension of intersection of carpets} 
appears to be a \lq\lq{}hybrid\rq\rq{} of Theorem \ref{theorem: Kenyon--Peres} (the original Kenyon--Peres formula) and 
another dimension formula of Kenyon and Peres \cite[Theorem~1.1]{Kenyon--Peres_sofic} about sofic affine invariant sets.
We expect that the method of \cite{Kenyon--Peres_sofic} (see also \cite{Alibabaei_sofic}) may be useful
for the rigorous calculation of $\lambda$.
Indeed, by adapting an observation of \cite[pp. 167--168]{Kenyon--Peres_sofic},
the constant $\lambda$ can be interpreted, at least formally, as the top Lyapunov exponent of
random products of \textit{infinite-dimensional matrices}:
Let $S^3$ be the $3$-dimensional sphere and $C(S^3)$ the space of continuous functions on it.
For each $\tau\in A\times B$ we define an operator $\mathcal{L}_{\tau}\colon C(S^3)\to C(S^3)$ by 
\[ \mathcal{L}_\tau f(x) = \sum_{v\in B} |Q_{\tau, v} x|^{\log_a b} f\left(\frac{Q_{\tau, v}x}{|Q_{\tau, v}x|}\right), \]
where the term
$|Q_{\tau, v} x|^{\log_a b} f\left(\frac{Q_{\tau,v}x}{|Q_{\tau,v}x|}\right)$
is defined to be zero if $Q_{\tau,v}x=0$.
A simple induction gives
\begin{align*}
 &\left(\mathcal{L}_{\tau_1}\mathcal{L}_{\tau_2}\cdots \mathcal{L}_{\tau_n}f\right)(x) \\
  &= \sum_{v_1, \dots, v_n\in B} \left|Q_{\tau_n, v_n}Q_{\tau_{n-1}, v_{n-1}}\cdots Q_{\tau_1, v_1} x\right|^{\log_a b}
     f\left(\frac{Q_{\tau_n, v_n}Q_{\tau_{n-1}, v_{n-1}}\cdots Q_{\tau_1, v_1} x}{\left|Q_{\tau_n, v_n}Q_{\tau_{n-1}, v_{n-1}}\cdots Q_{\tau_1, v_1} x\right|}
     \right). 
\end{align*}
In this formula, as above, the summand is interpreted as \(0\) whenever
\(Q_{\tau_n,v_n}\cdots Q_{\tau_1,v_1}x=0\).
Let $\mathbbm{1}$ denote the constant function on $S^3$ taking the value $1$.
Then 
\[ \lambda = \lim_{n\to \infty} \frac{1}{n} \log \norm{\mathcal{L}_{\tau_1}\mathcal{L}_{\tau_2}\cdots \mathcal{L}_{\tau_n} \mathbbm{1}}_\infty, \]
where $\tau_1, \tau_2, \tau_3, \dots$ are the i.i.d. random variables that uniformly take values in $A\times B$.
This expression of $\lambda$ is rather formal and does not immediately yield a practical method of computation.
(We will not use it in the rest of the present paper.)
Nevertheless, it may provide useful intuition for future developments.    
\end{remark}

\subsection{Organization of the paper} \label{subsection: organization of the paper}

We describe the organization of the rest of the paper.

\S \ref{section: preliminaries} contains preliminary material on
entropy theory. 
In \S \ref{section: proof of one direction of the relativised variational principle}, 
we prove one direction of the
relativised variational principle:
\begin{equation*} 
  \begin{split}
   & \sup\left\{w h_\mu(T|R) + (1-w) h_{\pi_*\mu}(S|R) \middle|\, 
    \mu \in \mathscr{M}^T(X)  \text{ with }\rho_*\mu = \nu\right\} \\
    & \leq \int_Z \FHh^w\left(\rho^{-1}(z),\pi, T\right) \, d\nu(z),
   \end{split}
\end{equation*}  
using a weighted relative Brin--Katok
formula of Wang and Huang \cite{Wang--Huang}.

The reverse inequality is proved in several steps.  
In \S \ref{section: proof of relativised variational principle under additional assumptions}, 
we establish
it under additional assumptions, namely in the zero-dimensional and ergodic
case.  In \S \ref{section: removing the ergodicity assumptions}, 
we remove the ergodicity assumption by a measurable
selection argument.  
In \S \ref{section: completion of the proof of relativised variational principle}, 
we remove the zero-dimensional assumption
using the theory of principal extensions, thereby completing the proof of the relativised
variational principle.  We also relate the Feng--Huang weighted entropy
to combinatorial weighted entropy; the consequence needed later is stated
in \S \ref{subsection: useful corollary}.

\S \ref{section: application to Bedford--McMullen carpets} 
is devoted to the application to Bedford--McMullen carpets.  
There we identify the geometric side of the theory with Hausdorff dimension and compute
the combinatorial side in terms of random matrix products.

Readers interested mainly in the Bedford--McMullen application may skip the
proof of the general relativised variational principle on a first reading.  
The application uses the preceding theory only through the result stated in
\S \ref{subsection: useful corollary}, 
and hence \S \ref{subsection: useful corollary} together with 
\S \ref{section: application to Bedford--McMullen carpets} can be read
independently of most of 
\S \ref{section: preliminaries}--\S \ref{subsection: proof of relativised variational principle}.

\subsection{Acknowledgement of AI assistance}

During the preparation of this paper, we used AI-assisted tools, in particular
ChatGPT Pro (GPT-5.5), to help check for possible mistakes, improve English expressions,
prepare figures, and write code for the numerical computations in
Example~\ref{example: numerical calculation}.  AI assistance was especially useful in
the exploratory search leading to Example~\ref{example: sharing positive eigenvector}.
All AI-generated suggestions
and outputs were checked and edited by the authors.

\section{Preliminaries}  \label{section: preliminaries}

In this section we prepare several facts about Kolmogorov--Sinai entropy and weighted topological entropy.
For the materials around Kolmogorov--Sinai entropy, basic references are the books of 
Walters \cite[Chapter 4]{Walters_book} and Downarowicz \cite[Chapters 1 and 2]{Downarowicz}.

Let $X$ be a compact metrizable space with its Borel $\sigma$-algebra $\mathcal{B}_X$.
Let $\mathscr{M}(X)$ be the set of Borel probability measures on $X$.
Take $\mu \in \mathscr{M}(X)$ and let 
$\alpha = \{A_1, \dots, A_a\}$ be a finite measurable partition of $X$.
We define the \textbf{Shannon entropy} by 
\[ H_\mu(\alpha) = -\sum_{i=1}^a \mu(A_i) \log \mu(A_i), \]
where we assume $0\log 0  = 0$.

For another finite measurable partition $\beta = \{B_1, \dots, B_b\}$, we define 
\[ \alpha\vee \beta  = \{A_i\cap B_j\mid 1\leq i \leq a, 1\leq j \leq b\}. \]
We have $H_\mu(\alpha\vee \beta) \leq H_\mu(\alpha) + H_\mu(\beta)$.

Let $\mathcal{A}$ be a sub-$\sigma$-algebra of $\mathcal{B}_X$.
For $\mu\in \mathscr{M}(X)$, we consider its disintegration by $\mathcal{A}$:
\[ \mu = \int_X \mu_x^{\mathcal{A}} d\mu(x).  \]
For a finite measurable partition $\alpha$ of $X$, we define 
\[ H_\mu(\alpha|\mathcal{A}) = \int_X H_{\mu^{\mathcal{A}}_x}(\alpha) d\mu(x). \]
For another finite measurable partition $\beta$ of $X$, we have 
\[ H_\mu(\alpha\vee \beta|\mathcal{A}) \leq H_\mu(\alpha|\mathcal{A}) + H_\mu(\beta|\mathcal{A}). \]

The following lemma is standard, but we include the proof for completeness.

\begin{lemma} \label{lemma: concavity of entropy and reverse}
Let $X$ be a compact metrizable space, $\alpha$ a finite measurable partition of $X$ and 
$\mathcal{A}$ a sub-$\sigma$-algebra of $\mathcal{B}_X$.
For $0\leq t \leq 1$ and $\mu_0, \mu_1\in \mathscr{M}(X)$, we have
\begin{equation} \label{eq: concavity of entropy}
 (1-t) H_{\mu_0}(\alpha|\mathcal{A}) + t H_{\mu_1}(\alpha|\mathcal{A})  \leq 
    H_{(1-t)\mu_0 + t\mu_1}(\alpha|\mathcal{A}), 
\end{equation}
\begin{equation}  \label{eq: reverse inequality for concavity of entropy}
     H_{(1-t)\mu_0 + t\mu_1}(\alpha|\mathcal{A}) \leq (1-t) H_{\mu_0}(\alpha|\mathcal{A}) + t H_{\mu_1}(\alpha|\mathcal{A}) + h(t), 
\end{equation}     
where $h(t) = -(1-t) \log (1-t) - t\log t$.
(We assume $h(0) = h(1) = 0$.)
\end{lemma}

\begin{proof}
The cases \(t=0\) and \(t=1\) are trivial, so we assume \(0<t<1\).
Set $\mu = (1-t)\mu_0 + t \mu_1$.
We consider a space $X^\prime := X\times \{0,1\}$ with a probability measure $\mu^\prime$ defined by 
\[ \mu^\prime(A\times \{0\}) = (1-t) \mu_0(A), \quad 
    \mu^\prime(A\times \{1\}) = t \mu_1(A), \quad (A\in \mathcal{B}_X). \]
We define a partition $\alpha^\prime$ and a $\sigma$-algebra $\mathcal{A}^\prime$ of $X^\prime$ by 
\[ \alpha^\prime = \{A\times \{0,1\}\mid A\in \alpha\}, \quad 
     \mathcal{A}^\prime = \{A\times \{0,1\}\mid A\in \mathcal{A}\}. \]
Since $\mu^\prime(A\times \{0,1\}) = (1-t)\mu_0(A) + t \mu_1(A)$ for $A\in \mathcal{B}_X$, we have 
\[ H_{\mu^\prime}(\alpha^\prime\mid \mathcal{A}^\prime)  = H_\mu(\alpha\mid \mathcal{A}). \]     
     
Let $\beta = \{X\times \{0\}, X\times \{1\}\}$. This is a partition of $X^\prime$.
Let $\mathcal{A}^\prime \vee \beta$ be the $\sigma$-algebra generated by $\mathcal{A}^\prime$ and $\beta$.
Explicitly 
\[ \mathcal{A}^\prime \vee \beta = \{A_0 \times \{0\}\cup A_1 \times \{1\} \mid A_0, A_1 \in \mathcal{A}\}. \]
Then, for \(\mu_0\)-a.e. \(x\in X\) and for \(\mu_1\)-a.e. \(x\in X\), respectively,
\[
(\mu^\prime)^{\mathcal{A}^\prime \vee \beta}_{(x,0)}
=
(\mu_0)^{\mathcal A}_x\times\delta_0,
\qquad
(\mu^\prime)^{\mathcal{A}^\prime \vee \beta}_{(x,1)}
=
(\mu_1)^{\mathcal A}_x\times\delta_1.
\]
($\delta_0$ and $\delta_1$ denote the Dirac measures at $0$ and $1$ respectively.)
Hence 
\[ H_{\mu^\prime}(\alpha^\prime\mid \mathcal{A}^\prime \vee \beta)
   = (1-t) H_{\mu_0}(\alpha\mid \mathcal{A}) + t H_{\mu_1}(\alpha\mid \mathcal{A}). \]
Since conditioning reduces entropy, we obtain 
\[ H_{\mu^\prime}(\alpha^\prime\mid \mathcal{A}^\prime) 
  \geq  H_{\mu^\prime}(\alpha^\prime\mid \mathcal{A}^\prime \vee \beta). \]
Recall $H_{\mu^\prime}(\alpha^\prime\mid \mathcal{A}^\prime)  = H_\mu(\alpha\mid \mathcal{A})$.
This provides the inequality \eqref{eq: concavity of entropy}.

On the other hand, we have 
$H_{\mu^\prime}(\alpha^\prime\mid \mathcal{A}^\prime) 
\leq H_{\mu^\prime}(\alpha^\prime \vee \beta \mid \mathcal{A}^\prime)$.
By the chain rule
\begin{align*}
   H_{\mu^\prime}(\alpha^\prime \vee \beta \mid \mathcal{A}^\prime) 
    & = H_{\mu^\prime}(\beta\mid \mathcal{A}^\prime) + H_{\mu^\prime}(\alpha^\prime\mid \mathcal{A}^\prime \vee \beta) \\
    & = H_{\mu^\prime}(\beta\mid \mathcal{A}^\prime)
     + (1-t) H_{\mu_0}(\alpha\mid \mathcal{A}) + t H_{\mu_1}(\alpha\mid \mathcal{A}). 
\end{align*}
Since $H_{\mu^\prime}(\beta\mid \mathcal{A}^\prime) \leq H_{\mu^\prime}(\beta) = h(t)$, we obtain the inequality 
\eqref{eq: reverse inequality for concavity of entropy}.
\end{proof}

Let $X$ be a compact metrizable space, and 
let $T\colon X\to X$ be a continuous map.
We define $\mathscr{M}^T(X)$ as the set of $T$-invariant Borel probability measures on $X$.
We also define $\mathscr{M}^T_{\mathrm{erg}}(X)$ as the set of ergodic measures.
For $0\leq m <n$ and a finite measurable partition $\alpha$ of $X$, we set 
\[ \alpha_m^n = \bigvee_{j=m}^n T^{-j}\alpha. \]
For $\mu\in \mathscr{M}^T(X)$, we define 
\[ h_\mu(T,\alpha) = \lim_{n\to \infty} \frac{1}{n} H_\mu(\alpha_0^{n-1}). \]
We define the \textbf{Kolmogorov--Sinai entropy} $h_\mu(T)$ as the supremum of $h_\mu(T,\alpha)$ over all finite 
measurable partitions $\alpha$ of $X$.

Let $(Y, S)$ be another dynamical system and $\mathcal{B}_Y$ the Borel $\sigma$-algebra of $Y$.
Let $\pi\colon X\to Y$ be an equivariant continuous map.
For $\mu\in \mathscr{M}^T(X)$ and a finite measurable partition $\alpha$ of $X$, we define 
\[  h_\mu(T|S, \alpha) = \lim_{n\to \infty} \frac{1}{n} H_\mu(\alpha_0^{n-1}|\pi^{-1}\mathcal{B}_Y), \]
where $\pi^{-1}\mathcal{B}_Y$ denotes the pull-back of $\mathcal{B}_Y$ by $\pi$.
We define $h_\mu(T|S)$ as the supremum of $h_\mu(T|S, \alpha)$ over all finite measurable partitions $\alpha$
of $X$.

We have \cite[Lemma 3.1]{Ledrappier--Walters}
\[ h_\mu(T) = h_{\pi_*\mu}(S) + h_\mu(T|S). \]
In particular, if $h_\mu(T) < \infty$ then we can write $h_\mu(T|S) = h_\mu(T) - h_{\pi_*\mu}(S)$.
(More generally, given a third dynamical system $(Z, R)$ and an equivariant continuous map $\theta\colon Y\to Z$,
we have $h_\mu(T|R) = h_{\pi_*\mu}(S|R) + h_\mu(T|S)$.)

The following lemma was proved in \cite[Lemma 3.2]{Ledrappier--Walters}.

\begin{lemma} \label{lemma: basic properties of relative measure theoretic entropy}
Let $\pi\colon (X, T)\to (Y, S)$ be an equivariant continuous map between dynamical systems
$(X, T)$ and $(Y, S)$.
Let $\alpha$ be a finite measurable partition of $X$.
\begin{enumerate}
  \item Let $\{\mu_n\}_{n=1}^\infty \subset \mathscr{M}(X)$ be a sequence of Borel probability measures on $X$
  that converges to some $\mu\in \mathscr{M}(X)$ in the weak$^*$ topology. 
  If $\mu(\partial A) = 0$ for all $A\in \alpha$ then we have 
  \[ \limsup_{n\to \infty} H_{\mu_n}(\alpha|\pi^{-1}\mathcal{B}_Y) \leq H_\mu(\alpha|\pi^{-1}\mathcal{B}_Y). \]
  \item Let $\mu\in \mathscr{M}^T(X)$ and we consider its ergodic decomposition 
  \[ \mu = \int_{\mathscr{M}^T_{\mathrm{erg}}(X)} \lambda dp(\lambda), \]
  where $p$ is a Borel probability measure on $\mathscr{M}^T_{\mathrm{erg}}(X)$.
  Then we have 
  \[ h_\mu(T|S) = \int_{\mathscr{M}^T_{\mathrm{erg}}(X)} h_\lambda(T|S) dp(\lambda). \]
\end{enumerate}
\end{lemma}

In the rest of this section we prepare elementary facts on the weighted topological entropy.
First we consider the definition of $\varepsilon$-covering number.
Let $(X, \mathbf{d})$ be a compact metric space.
For $\varepsilon >0$ and $\Omega\subset X$, we defined $\#(\Omega, \mathbf{d}, \varepsilon)$ as the minimum $n\geq 1$
for which there exists an open covering $\Omega \subset U_1\cup \dots \cup U_n$ satisfying $\diam(U_i, \mathbf{d}) < \varepsilon$
for all $1\leq i \leq n$.
Now it is easy to check that we can replace \lq\lq{}open covering\rq\rq{} with \lq\lq{}closed covering\rq\rq{} in this definition.
Namely $\#(\Omega, \mathbf{d}, \varepsilon)$ is the minimum $n\geq 1$ for which there exist closed subsets $E_1, \dots E_n$ of $X$
that satisfy $\Omega\subset E_1\cup \dots \cup E_n$ and $\diam (E_i, \mathbf{d}) < \varepsilon$ $(1\leq i \leq n)$.

Let $\pi\colon (X, T)\to (Y, S)$ be an equivariant continuous map between dynamical systems, and let 
$\mathbf{d}$ and $\mathbf{d}^\prime$ be metrics on $X$ and $Y$ respectively.
For $0\leq w \leq 1$ and $\Omega \subset X$, we defined $\#^w(\Omega, N, \varepsilon)$ as the minimum of 
$\sum_{j=1}^m \left(\#(\Omega \cap \pi^{-1}(V_j), \mathbf{d}_N, \varepsilon)\right)^w$ over all open covering 
$\pi(\Omega) \subset V_1\cup \dots \cup V_m$ with $\diam(V_j, \mathbf{d}^\prime_N) < \varepsilon$ $(1\leq j \leq m)$.
The following lemma will be used in \S \ref{section: application to Bedford--McMullen carpets}.

\begin{lemma} \label{lemma: weighted covering number}
If $\Omega\subset X$ is closed, then 
\begin{equation*}
 \begin{split}
  & \#^w\left(\Omega, N, \varepsilon\right) \\
  & = \inf\left\{\sum_{j=1}^m \left(\#\left(\Omega \cap \pi^{-1}(E_j), \mathbf{d}_N, \varepsilon\right)\right)^w \middle|\, 
     \parbox{3in}{\centering $E_1, \dots, E_m$ are closed sets of $Y$ that satisfy $\pi(\Omega) \subset E_1\cup \dots \cup E_m$ and 
             $\diam(E_j, \mathbf{d}^\prime_N) < \varepsilon$ for all $j$}\right\}. 
  \end{split}           
\end{equation*}
\end{lemma}

\begin{proof}
We denote the right-hand side of the above identity by $\#^w\left(\Omega, N, \varepsilon\right)^\prime$.
Let $V_1, \dots, V_m$ be open subsets of $Y$ that satisfy $\pi(\Omega)\subset V_1\cup \dots \cup V_m$ and 
$\diam(V_j, \mathbf{d}_N^\prime) < \varepsilon$.
Since $\pi(\Omega)$ is compact, there exist closed subsets $E_1, \dots, E_m$ of $Y$ that satisfy 
$E_i\subset V_i$ and $\pi(\Omega) \subset E_1\cup \dots \cup E_m$.
Then 
\[ \sum_{j=1}^m \left(\#\left(\Omega \cap \pi^{-1}(E_j), \mathbf{d}_N, \varepsilon\right)\right)^w \leq 
  \sum_{j =1}^m \left(\#(\Omega \cap \pi^{-1}(V_j), \mathbf{d}_N, \varepsilon)\right)^w. \]
It follows that $\#^w\left(\Omega, N, \varepsilon\right)^\prime \leq \#^w\left(\Omega, N, \varepsilon\right)$.
The reverse inequality is easy.
\end{proof}

Next we consider the Feng--Huang entropy.
Let $0\leq w \leq 1$.
Let $(X_i, T_i)$ and $(Y_i, S_i)$ be dynamical systems $(i=1,2)$, and
$\pi_i\colon X_i\to Y_i$ equivariant continuous maps.
Let $f \colon X_1\to X_2$ and $g \colon Y_1\to Y_2$ be also equivariant continuous maps.
Suppose that $f$ is surjective and $g\circ \pi_1 = \pi_2\circ f$.
Namely we have the following commutative diagram.

\[
\begin{tikzcd}
(X_1,T_1) \arrow[r,"\pi_1"] \arrow[d, two heads,"f"'] &  (Y_1,S_1)    \arrow[d,"g"] \\
 (X_2,T_2) \arrow[r,"\pi_2"'] & (Y_2,S_2)
\end{tikzcd}
\]

\begin{lemma} \label{lemma: factor map and weighted topological entropy}
For any $\Omega \subset X_2$ we have 
$\FHh^w(\Omega,\pi_2, T_2) \leq \FHh^w\left(f^{-1}(\Omega), \pi_1, T_1\right)$.
\end{lemma}

\begin{proof}
Let $\mathbf{d}_{X_i}$ and $\mathbf{d}_{Y_i}$ be metrics on $X_i$ and $Y_i$ respectively.
For any $\varepsilon>0$ there exists $\delta>0$ such that 
  \begin{itemize}
    \item $\mathbf{d}_{X_1}(x, x^\prime) < \delta$ $\Longrightarrow$ $\mathbf{d}_{X_2}(f(x), f(x^\prime))< \varepsilon$,
    \item  $\mathbf{d}_{Y_1}(y, y^\prime) < \delta$ $\Longrightarrow$ $\mathbf{d}_{Y_2}(g(y), g(y^\prime))< \varepsilon$.
  \end{itemize}
Then we have $f\left(B^w_n(x,\pi_1, \delta)\right) \subset B^w_n(f(x),\pi_2, \varepsilon)$ for any $n\geq 1$ and $x\in X_1$.

Let $s\geq 0$ and $N\geq 1$. If we have an (at most) countable covering 
$f^{-1}(\Omega) \subset \bigcup_j B^w_{n_j}(x_j, \pi_1, \delta)$ with $n_j\geq N$ then 
$\Omega \subset \bigcup_j B_{n_j}^w(f(x_j),\pi_2, \varepsilon)$ since $f$ is surjective. 
Hence $\Lambda^{w,s}_{N,\varepsilon}(\Omega) \leq \sum_j e^{-n_j s}$.
Therefore 
\[ \Lambda^{w,s}_{N,\varepsilon}(\Omega) \leq \Lambda^{w,s}_{N,\delta}(f^{-1}(\Omega)). \] 
Letting $N\to \infty$, we obtain $\Lambda^{w,s}_{\varepsilon}(\Omega) \leq \Lambda^{w,s}_{\delta}(f^{-1}(\Omega))$
and $\FHh^w(\Omega,T_2,  \varepsilon) \leq \FHh^w(f^{-1}(\Omega),T_1, \delta)$.
Letting $\delta\to 0$ and $\varepsilon \to 0$, we conclude 
$\FHh^w(\Omega,\pi_2, T_2) \leq \FHh^w\left(f^{-1}(\Omega), \pi_1, T_1\right)$.
\end{proof}

Let $(Z, R)$ be a dynamical system.
Let $\rho\colon X_2\to Z$ and $\theta\colon Y_2\to Z$ be equivariant continuous maps
such that $\rho = \theta\circ \pi_2$.
\[
\begin{tikzcd}
(X_1,T_1) \arrow[r,"\pi_1"] \arrow[d, two heads, "f"']
  & (Y_1,S_1) \arrow[d,"g"] \\
(X_2,T_2) \arrow[r,"\pi_2"'] \arrow[dr, "\rho"']
  & (Y_2,S_2) \arrow[d,"\theta"] \\
& (Z,R)
\end{tikzcd}
\]

\begin{corollary} \label{corollary: factor map and weighted topological entropy}
In the above setting, $\FHh^w\left(\rho^{-1}(z), \pi_2, T_2\right) \leq \FHh^w\left((\rho\circ f)^{-1}(z), \pi_1, T_1\right)$
for all $z\in Z$.
\end{corollary}

\begin{proof}
Since $(\rho\circ f)^{-1}(z) = f^{-1}\left(\rho^{-1}(z)\right)$, this is an immediate corollary of 
Lemma \ref{lemma: factor map and weighted topological entropy}.
\end{proof}

\section{Proof of one direction of the relativised variational principle} \label{section: proof of one direction of the relativised variational principle}

Throughout this section we assume the following conditions.
\begin{enumerate}
  \item $0\leq w\leq 1$.
  \item $(X, T)$, $(Y, S)$ and $(Z, R)$ are dynamical systems.
  \item $\pi\colon X\to Y$, $\rho\colon X\to Z$ and $\theta\colon Y\to Z$ are equivariant continuous maps with
$\rho = \theta \circ \pi$. 
  \item $\rho$ is surjective.
\end{enumerate}
\[
\begin{tikzcd}
(X,T) \arrow[r,"\pi"] \arrow[dr, two heads, "\rho"'] & (Y,S) \arrow[d,"\theta"] \\
& (Z,R)
\end{tikzcd}
\]
Fix metrics $\mathbf{d}$ and $\mathbf{d}^\prime$ on $X$ and $Y$ respectively.

In this section, we prove the following inequality from Theorem \ref{theorem: relativised variational principle}:
\begin{equation}  \label{eq: half of the variational principle}
  w h_\mu(T|R) + (1-w) h_{\pi_*\mu}(S|R) \leq \int_Z \FHh^w\left(\rho^{-1}(z),\pi, T\right) \, d(\rho_*\mu)(z)  
\end{equation}  
for any measure $\mu \in \mathscr{M}^T(X)$.

\begin{lemma} \label{lemma: measurability of fiber entropy}
 Set $\varphi(z) = \FHh^w\left(\rho^{-1}(z),\pi, T\right)$ for $z\in Z$.
 \begin{enumerate}
    \item $\varphi(z)$ is a measurable function.
    \item For any invariant probability measure $\nu\in \mathscr{M}^R(Z)$, 
we have $\varphi(Rz) = \varphi(z)$ for $\nu$-almost every $z\in Z$.
In particular, if $\nu$ is ergodic, then $\varphi$ is constant $\nu$-almost everywhere.
 \end{enumerate}
\end{lemma}

\begin{proof}
(1) It is enough to prove that, for any $s, N, \varepsilon$, 
the quantity $\Lambda^{w, s}_{N,\varepsilon}\left(\rho^{-1}(z)\right)$ is measurable in $z\in Z$.
Indeed, once this is proved, 
then $\Lambda^{w, s}_{\varepsilon}\left(\rho^{-1}(z)\right) = \lim_{N\to \infty} \Lambda^{w, s}_{N,\varepsilon}\left(\rho^{-1}(z)\right)$ 
is also measurable in $z$, and hence for any $s\geq 0$ 
the set 
\[ \{z\mid \FHh^w\left(\rho^{-1}(z), T, \varepsilon\right) \leq  s\} 
  = \bigcap_{n=1}^\infty \{z\mid \Lambda^{w, s+\frac{1}{n}}_{\varepsilon}\left(\rho^{-1}(z)\right) = 0\} \]
is a measurable subset of $Z$.
Namely $\FHh^w\left(\rho^{-1}(z), T, \varepsilon\right)$ is a measurable function in $z$.
Then $\varphi(z) = \lim_{\varepsilon \to 0}\FHh^w\left(\rho^{-1}(z), T, \varepsilon\right)$ is also measurable.

Take any $C>0$. We will prove that $\{z\mid \Lambda^{w, s}_{N,\varepsilon}\left(\rho^{-1}(z)\right) < C\}$ is an open subset of $Z$.
Suppose $\Lambda^{w, s}_{N,\varepsilon}\left(\rho^{-1}(z)\right) < C$ for some $z\in Z$.
Then there exists an at most countable cover $\{B^w_{n_k}(x_k, \pi, \varepsilon)\}$ of $\rho^{-1}(z)$ that satisfies $n_k\geq N$ and 
\[ \sum_k e^{-sn_k} < C. \] 
Since $\rho^{-1}(z)$ is compact and 
$\bigcup_k B^w_{n_k}(x_k,\pi,\varepsilon)$ is an open neighbourhood of $\rho^{-1}(z)$,
if $z^\prime\in Z$ is sufficiently close to $z\in Z$ then $\rho^{-1}(z^\prime)$ is also covered by $\{B^w_{n_k}(x_k, \pi, \varepsilon)\}$.
Then we obtain $\Lambda^{w,s}_{N, \varepsilon}\left(\rho^{-1}(z^\prime)\right) < C$. 

(2) First we prove that $\varphi(Rz) \geq \varphi(z)$ for all $z\in Z$.
Let $\varepsilon>0$.
Take a finite cover $X = B^w_1(p_1, \pi, \varepsilon)\cup \dots \cup B^w_1(p_A, \pi, \varepsilon)$.
For any natural number $n$, $1\leq a \leq A$ and $x\in X$, there is a point $x^\prime\in X$ that satisfies
\[ 
B_1^w(p_a, \pi, \varepsilon) \cap T^{-1}\left(B^w_n(x, \pi, \varepsilon)\right) 
\subset B_n^w(x^\prime, \pi, 2\varepsilon).
\]
(When the left-hand side is empty, this is trivial. 
Otherwise, any point $x^\prime \in B_1^w(p_a, \pi, \varepsilon) \cap T^{-1}\left(B^w_n(x, \pi, \varepsilon)\right)$
satisfies the claim.)
Since $T\rho^{-1}(z) \subset \rho^{-1}(Rz)$,
it follows that 
\[ 
  \Lambda^{w, s}_{N, 2\varepsilon}\left(\rho^{-1}(z)\right) \leq A \Lambda^{w, s}_{N,\varepsilon}\left(\rho^{-1}(Rz)\right).
 \] 
Noting that $A$ is independent of $N$, we can let $N\to \infty$ and obtain 
$\Lambda^{w, s}_{2\varepsilon}\left(\rho^{-1}(z)\right) \leq A \Lambda^{w, s}_{\varepsilon}\left(\rho^{-1}(Rz)\right)$.
Then $\FHh^w\left(\rho^{-1}(z), T, 2\varepsilon\right) \leq \FHh^w\left(\rho^{-1}(Rz), T, \varepsilon\right)$.
Letting $\varepsilon \to 0$, we conclude $\varphi(z) \leq \varphi(Rz)$.

Let $\varphi_n(z) := \min\left(\varphi(z), n\right)$ for natural numbers $n$.
This is a bounded measurable function.
We have $\varphi_n(Rz) \geq \varphi_n(z)$. 
Since $\nu$ is $R$-invariant,
\[ \int_Z \varphi_n(Rz) d\nu(z) = \int_Z \varphi_n(z) d\nu(z). \]
Hence $\varphi_n(Rz) = \varphi_n(z)$ for $\nu$-almost every $z\in Z$.
Letting $n\to \infty$, we conclude that $\varphi(Rz) = \varphi(z)$ $\nu$-almost everywhere.
\end{proof}

\begin{remark} \label{remark: R-invariance of combinatorial weighted entropy}
The quantity $\psi(z) := \h^w\left(\rho^{-1}(z), \pi, T\right)$ (the combinatorial version of weighted topological entropy)
introduced in (\ref{subsection: relativised variational principle for weighted topological entropy})
has the same properties. 
Namely it is measurable and $R$-invariant $\nu$-almost everywhere.
The measurability was proved in \cite[Proposition 2.2]{Yin}.
Here we provide a proof of the $R$-invariance for completeness.
Let $\varepsilon>0$. We fix open coverings $X = P_1\cup\dots \cup P_A$ and $Y=Q_1\cup \dots \cup Q_B$
that satisfy $\diam(P_a, \mathbf{d}) < \varepsilon$ and $\diam(Q_b, \mathbf{d}^\prime) < \varepsilon$ for all 
$1\leq a \leq A$ and $1\leq b \leq B$. Let $N$ be a natural number.
For any $U\subset X$ with $\diam(U, \mathbf{d}_N) < \varepsilon$, 
we have $\diam (T^{-1}U \cap P_a, \mathbf{d}_N) < \varepsilon$.
Hence for any $\Omega\subset X$ we have $\#\left(T^{-1}\Omega, \mathbf{d}_N, \varepsilon\right) 
\leq A \#\left(\Omega, \mathbf{d}_N, \varepsilon\right)$.
Let $\pi\left(\rho^{-1}(Rz)\right) \subset V_1\cup \dots \cup V_m$ be an open covering with 
$\diam(V_i, \mathbf{d}^\prime_N) < \varepsilon$ for any $1\leq i \leq m$.
Then 
$\pi\left(\rho^{-1}(z)\right) \subset S^{-1}\left(\pi\left(\rho^{-1}(Rz)\right)\right) \subset \bigcup_{b, i} Q_b \cap S^{-1}V_i$.
Each $Q_b\cap S^{-1}V_i$ has diameter smaller than $\varepsilon$ with respect to $\mathbf{d}^\prime_N$.  
We have $\pi^{-1}(Q_b\cap S^{-1}V_i)\cap \rho^{-1}(z) \subset T^{-1}(\pi^{-1}(V_i) \cap \rho^{-1}(Rz))$
and hence
\[ \#\left(\pi^{-1}(Q_b\cap S^{-1}V_i)\cap \rho^{-1}(z), \mathbf{d}_N, \varepsilon\right) 
   \leq A \#\left(\pi^{-1}(V_i) \cap \rho^{-1}(Rz), \mathbf{d}_N, \varepsilon\right). \]
It follows that 
\[ \#^w\left(\rho^{-1}(z), N, \varepsilon\right) \leq (A^w B) \#^w\left(\rho^{-1}(Rz), N, \varepsilon\right). \]
Taking the logarithm, dividing by $N$ and letting $N\to \infty$,
\[ \limsup_{N\to \infty} \frac{\log \#^w\left(\rho^{-1}(z), N, \varepsilon\right)}{N} 
     \leq \limsup_{N\to \infty}  \frac{\log \#^w\left(\rho^{-1}(Rz), N, \varepsilon\right)}{N}. \]
Letting $\varepsilon \to 0$, we obtain $\psi(z) \leq \psi(Rz)$.
Then it follows from the same argument as in Lemma \ref{lemma: measurability of fiber entropy} (2)
that $\psi$ is $R$-invariant $\nu$-almost everywhere for any $\nu\in \mathscr{M}^R(Z)$.
\end{remark}

The proof of \eqref{eq: half of the variational principle} is essentially an application of the following theorem of 
Wang and Huang \cite[Theorem 5.5]{Wang--Huang}.
It may be viewed as a weighted relative version of the Brin--Katok formula.
We state only the special case needed here.

\begin{theorem}[Wang--Huang 2019] \label{theorem: relativised and weighted Brin--Katok formula}
Assume that $R\colon Z\to Z$ is a homeomorphism.
Let $\mu\in \mathscr{M}^T(X)$, and set $\nu := \rho_*\mu \in \mathscr{M}^R(Z)$.
Let
\[ \mu = \int_Z \mu_z \, d\nu(z) \]
be the disintegration of $\mu$ by the map $\rho\colon X\to Z$, where 
$\mu_z$ is a probability measure supported on $\rho^{-1}(z)$ for $\nu$-almost every $z\in Z$.
Then we have
\begin{equation*}
  \begin{split}
  & \int_X \lim_{\varepsilon\to 0} \left(\limsup_{n\to \infty} \frac{-\log \mu_{\rho(x)}(B^w_n(x,\pi, \varepsilon))}{n}\right) d\mu(x) \\
   & = \int_X \lim_{\varepsilon\to 0} \left(\liminf_{n\to \infty} \frac{-\log \mu_{\rho(x)}(B^w_n(x,\pi, \varepsilon))}{n}\right) d\mu(x) \\
    & =  w h_\mu(T|R) +(1-w) h_{\pi_*\mu}(S|R).
  \end{split}
\end{equation*}   
\end{theorem}

Here $B^w_n(x, \pi, \varepsilon)$ denotes the $w$-weighted Bowen ball with respect to the map 
$\pi\colon (X, T)\to (Y,S)$.
When $w=1$ and both $Y$ and $Z$ are trivial (i.e. one-point space), Theorem \ref{theorem: relativised and weighted Brin--Katok formula}
reduces to the classical Brin--Katok formula \cite{Brin--Katok}.

The assumption that \(R\colon Z\to Z\) is a homeomorphism is imposed only in order to apply
\cite[Theorem~5.5]{Wang--Huang} literally.  Their theorem is formulated for
disintegrations with respect to \textit{invariant} sub-\(\sigma\)-algebras.  
When \(R\) is a homeomorphism, 
the Borel $\sigma$-algebra $\mathcal{B}_Z$ is $R$-invariant ($R^{-1}\mathcal{B}_Z = \mathcal{B}_Z$), 
and then the conditioning \(\sigma\)-algebra required in their theorem is invariant.

It might be possible to remove the invertibility assumption on $R$ from the 
statement of Theorem \ref{theorem: relativised and weighted Brin--Katok formula}.
However this is not trivial and we do not discuss it here\footnote{Notice that, if $R$ is not a homeomorphism,
the $\sigma$-algebra $\mathcal{B}_Z$ is not invariant in general (it is only \textbf{subinvariant}; 
$R^{-1} \mathcal{B}_Z\subset \mathcal{B}_Z$), 
and hence 
the situation is more complicated.
Nevertheless, it is known that the relative Shannon--McMillan--Breiman theorem holds for 
subinvariant conditioning $\sigma$-algebras as well \cite[Theorem B.0.1 in Appendix B]{Downarowicz}.
Therefore one possible strategy to remove the invertibility assumption from Theorem \ref{theorem: relativised and weighted Brin--Katok formula}
is to establish the weighted version of this relative Shannon--McMillan-Breiman theorem 
and to apply it to the current setting.}.

We can now prove the desired inequality.

\begin{proposition} \label{proposition: half of the variational principle}
For any $\mu \in \mathscr{M}^T(X)$ we have 
\begin{equation} \label{eq: one direction of variational principle in the proposition}
   w h_\mu(T|R) + (1-w) h_{\pi_*\mu}(S|R) \leq \int_Z \FHh^w\left(\rho^{-1}(z),\pi, T\right) \, d(\rho_*\mu)(z). 
\end{equation}   
\end{proposition}

\begin{proof}
We first reduce the proof to the case where $R\colon Z\to Z$ is a homeomorphism.
(This reduction is a standard application of the natural extension.  
It is somewhat lengthy and may be skipped on a first reading.)
As a preliminary step, we observe that we may assume, without loss of generality, that \(R\) is surjective.
Indeed, every $\nu\in \mathscr{M}^R(Z)$ has full measure on $\bigcap_{n=1}^\infty R^n Z$.
Hence we can replace $Z$ by $\bigcap_{n=1}^\infty R^n Z$, and $X$ and $Y$ by $\rho^{-1}\left(\bigcap_{n=1}^\infty R^n Z\right)$
and $\theta^{-1}\left(\bigcap_{n=1}^\infty R^n Z\right)$ respectively without affecting the validity of the proposition.
The restriction of $R$ to $\bigcap_{n=1}^\infty R^n Z$ is surjective.
Therefore we may assume that $R\colon Z\to Z$ is surjective from the beginning.

Next, we consider the natural extension $(\tilde{Z}, \tilde{R})$ of $(Z, R)$:
\[ \tilde{Z} = \{(z_n)_{n\in \mathbb{Z}} \in Z^{\mathbb{Z}}\mid R z_{n} = z_{n+1} \text{ for all } n\in \mathbb{Z} \}, \quad 
     \tilde{R}\left((z_n)_{n\in \mathbb{Z}}\right) = (z_{n+1})_{n\in \mathbb{Z}}. \]
Let $p^Z\colon \tilde{Z}\to Z$ be the projection to the zeroth coordinate.
Let $\tilde{X} := X\times_Z \tilde{Z}$ be the fiber product of $X$ and $\tilde{Z}$ over $Z$:
\[ \tilde{X} = \{(x, \tilde{z})\in X\times \tilde{Z} \mid \rho(x) = p^Z\left(\tilde{z}\right)\}. \]
Set $\tilde{T} = T\times \tilde{R}\colon \tilde{X}\to \tilde{X}$.
We also define $(\tilde{Y}, \tilde{S})$ similarly as the fiber product of $(Y,S)$ and $(\tilde{Z}, \tilde{R})$ over $(Z,R)$:
\[ \tilde{Y} = \{(y, \tilde{z})\in Y\times \tilde{Z}\mid \theta(y) = p^Z\left(\tilde{z}\right)\}, \quad 
    \tilde{S} = S\times \tilde{R}. \]
Set $\tilde{\pi} = \pi \times \mathrm{Id}_{\tilde{Z}} \colon \tilde{X}\to \tilde{Y}$, and let 
$\tilde{\rho}\colon \tilde{X}\to \tilde{Z}$ and $\tilde{\theta}\colon \tilde{Y}\to \tilde{Z}$ be the natural projections:
\begin{equation*} 
\begin{tikzcd}
(\tilde{X},\tilde{T}) \arrow[r,"\tilde{\pi}"] \arrow[dr, two heads, "\tilde{\rho}"'] & (\tilde{Y},\tilde{S}) \arrow[d,"\tilde{\theta}"] \\
& (\tilde{Z},\tilde{R})
\end{tikzcd}
\end{equation*}

Let $p^X\colon \tilde{X}\to X$ and $p^Y\colon \tilde{Y}\to Y$ be the natural projections.
It is known that, for any $\nu \in \mathscr{M}^{\tilde{R}}(\tilde{Z})$, we have 
$h_{\nu}(\tilde{R}|R) = 0$ \cite[p.190]{Downarowicz}. (This follows from the fact that $p^Z\colon \tilde{Z}\to Z$ is a principal factor map; 
see \S \ref{subsection: zero dimensional principal extension}.)
Therefore, for any $\mu \in \mathscr{M}^{\tilde{T}}(\tilde{X})$, we have 
$h_{\mu}(\tilde{T}|\tilde{R}) = h_{\mu}(\tilde{T}|R) \geq h_{p^X_*\mu}(T|R)$, where 
$h_{\mu}(\tilde{T}|R)$ denotes the conditional Kolmogorov--Sinai entropy of the map 
$p^Z\circ \tilde{\rho} = \rho\circ p^X\colon \tilde{X}\to Z$.
Similarly, $h_{\tilde{\pi}_*\mu}(\tilde{S}|\tilde{R}) \geq h_{\pi_*(p^X_*\mu)}(S|R)$.
Hence\footnote{Indeed, the maps $p^X$ and $p^Y$ are also principal factor maps; see Lemma \ref{lemma: fiber product}.
Therefore we have the equalities $h_{\mu}(\tilde{T}|\tilde{R}) = h_{p^X_*\mu}(T|R)$ and 
$h_{\tilde{\pi}_*\mu}(\tilde{S}|\tilde{R}) = h_{\pi_*(p^X_*\mu)}(S|R)$. But we do not need this here.}
\[ w h_{p^X_*\mu}(T|R) + (1-w) h_{\pi_*(p^X_*\mu)}(S|R) \leq w h_{\mu}(\tilde{T}|\tilde{R}) + (1-w) h_{\tilde{\pi}_*\mu}(\tilde{S}|\tilde{R}). \]

On the other hand, for any $\tilde{z}\in \tilde{Z}$, we have $\tilde{\rho}^{-1}(\tilde{z}) = \rho^{-1}\left(p^Z(\tilde{z})\right) \times \{\tilde{z}\}$ and 
hence (by the natural comparison of weighted Bowen balls of $\pi$ and $\tilde{\pi}$)
\[ \FHh^w\left(\tilde{\rho}^{-1}(\tilde{z}), \tilde{\pi}, \tilde{T}\right) = 
    \FHh^w\left(\rho^{-1}\left(p^Z(\tilde{z})\right), \pi, T\right). \]
Notice that $p^X_*\colon \mathscr{M}^{\tilde{T}}(\tilde{X}) \to \mathscr{M}^T(X)$ is surjective because 
$p^X\colon \tilde{X}\to X$ is surjective.    
Therefore it is enough to prove that 
\[  w h_{\mu}(\tilde{T}|\tilde{R}) + (1-w) h_{\tilde{\pi}_*\mu}(\tilde{S}|\tilde{R}) \leq 
    \int_{\tilde{Z}} \FHh^w\left(\tilde{\rho}^{-1}(\tilde{z}), \tilde{\pi}, \tilde{T}\right) d\left(\tilde{\rho}_*\mu\right)(\tilde{z}) \]
for $\mu \in \mathscr{M}^{\tilde{T}}(\tilde{X})$.
Since $(\tilde{Z}, \tilde{R})$ is an invertible dynamical system, 
we may assume that $R\colon Z\to Z$ is a homeomorphism 
by replacing $X, Y, Z$ with $\tilde{X}, \tilde{Y}, \tilde{Z}$.

We have done the reduction.
From now on, we assume that $R$ is a homeomorphism from the beginning of the argument, and we are going to prove 
\eqref{eq: one direction of variational principle in the proposition}.

By the ergodic decomposition theorem, it suffices to prove the proposition for ergodic $\mu$.
Set $\nu := \rho_*\mu$. Then $\nu$ is also ergodic with respect to $R$, and hence $\FHh^w\left(\rho^{-1}(z),\pi, T\right)$
is constant $\nu$-almost everywhere (Lemma \ref{lemma: measurability of fiber entropy}).
Let 
\[ \mu = \int_Z \mu_z \, d\nu(z) \]
be the disintegration of $\mu$ by the map $\rho\colon X\to Z$, where $\mu_z$ is a probability measure 
supported on $\rho^{-1}(z)$ for $\nu$-almost every $z\in Z$.

If $w h_\mu(T|R) + (1-w) h_{\pi_*\mu}(S|R)$ is zero, then the statement is trivial. So we assume that it is positive.
Take any $0<A< w h_\mu(T|R) + (1-w) h_{\pi_*\mu}(S|R)$.
We will prove that $\FHh^w\left(\rho^{-1}(z), \pi, T\right) \geq A$ for $\nu$-almost every $z\in Z$.

By the theorem of Wang and Huang (Theorem \ref{theorem: relativised and weighted Brin--Katok formula}),
\[ \int_X \lim_{\varepsilon \to 0} \left(\liminf_{n\to \infty} \frac{-\log \mu_{\rho(x)}(B^w_n(x,\pi, \varepsilon))}{n}\right) d\mu(x) >A. \]
Hence 
\[ \mu\left\{x\in X\middle|\, 
  \lim_{\varepsilon \to 0} \left(\liminf_{n\to \infty} \frac{-\log \mu_{\rho(x)}(B^w_n(x,\pi, \varepsilon))}{n}\right) >A\right\}
  >0. \]
Then we can find $\varepsilon>0$, $N>0$ and a Borel subset $E\subset X$ with $\mu(E)>0$ such that 
for every $x\in E$ and $n\geq N$
\[ - \frac{\log \mu_{\rho(x)}(B^w_n(x,\pi, \varepsilon))}{n} >A, \]
that is, $\mu_{\rho(x)}(B^w_n(x,\pi, \varepsilon)) < e^{-nA}$.
This implies that if $x\in X$ satisfies 
$B^w_n(x,\pi, \varepsilon/2) \cap E\cap \rho^{-1}(\rho(x)) \neq \emptyset$ for some $n\geq N$ then we have 
\[ \mu_{\rho(x)}\left(B^w_n(x,\pi, \varepsilon/2)\right) < e^{-nA}. \]
Indeed, if \(y\in B^w_n(x,\pi, \varepsilon/2) \cap E\cap \rho^{-1}(\rho(x))\),
then
$B^w_n(x,\pi,\varepsilon/2)\subset B^w_n(y,\pi,\varepsilon)$ and $\rho(y) = \rho(x)$, and hence
$\mu_{\rho(x)}\left(B^w_n(x,\pi, \varepsilon/2)\right) \leq \mu_{\rho(y)}\left(B^w_n(y,\pi,\varepsilon)\right) < e^{-nA}$.

Set $c := \mu(E)/2 >0$.
Since 
\[ \int_Z \mu_z(E) \, d\nu(z) = \mu(E) =2c, \]
there exists a Borel subset $Z^\prime\subset Z$ such that $\nu(Z^\prime) >0$ and $\mu_z$ is a probability measure 
supported on $\rho^{-1}(z)$ satisfying $\mu_z(E) > c$ for every $z\in Z^\prime$.

Let $z\in Z^\prime$.
We would like to prove $\Lambda^{w, A}_{N, \varepsilon/4}\left(\rho^{-1}(z)\right) \geq c$.
Let 
\[ \rho^{-1}(z) \subset \bigcup_k B^w_{n_k}(x_k,\pi, \varepsilon/4) \]
be an at most countable covering with $n_k\geq N$.
We can assume $\rho^{-1}(z) \cap B^w_{n_k}(x_k, \pi, \varepsilon/4)\neq \emptyset$ for all $k$.
Take a point $x^\prime_k$ in this intersection for each $k$.
Then $\rho^{-1}(z) \subset \bigcup_k B^w_{n_k}(x^\prime_k, \pi, \varepsilon/2)$.
We have 
\begin{align*}
 c & < \mu_z(E) \\
  & = \mu_z\left(E\cap \rho^{-1}(z)\right) \\
 & \leq \sum_{k: E\cap \rho^{-1}(z)\cap B^w_{n_k}(x^\prime_k,\pi, \varepsilon/2)\neq \emptyset} 
        \mu_z\left(B^w_{n_k}(x^\prime_k, \pi,  \varepsilon/2)\right)  \\
      &  < \sum_k e^{-n_k A}. 
\end{align*}        
Therefore 
\[ \Lambda^{w, A}_{N, \varepsilon/4}\left(\rho^{-1}(z)\right) \geq  c. \]
It follows that $\Lambda^{w, A}_{\varepsilon/4}\left(\rho^{-1}(z)\right) \geq  c$ and hence 
$\FHh^w(\rho^{-1}(z),\pi, T) \geq A$ for all $z\in Z^\prime$.
Since $\nu(Z^\prime)>0$ and $\FHh^w(\rho^{-1}(z),\pi, T)$ is constant $\nu$-almost everywhere, 
we conclude that 
$\FHh^w(\rho^{-1}(z),\pi, T) \geq A$ for $\nu$-almost every $z\in Z$.
\end{proof}

\section{Proof of relativised variational principle under additional assumptions} 
\label{section: proof of relativised variational principle under additional assumptions}

The purpose of this section is to prove that, in the setting of Theorem \ref{theorem: relativised variational principle}, 
\begin{equation} \label{eq: difficult part of variational principle} 
  \begin{split}
   & \int_Z \FHh^w\left(\rho^{-1}(z),\pi, T\right) \, d\nu(z) \\
   & \leq \sup\left\{w h_\mu(T|R) + (1-w) h_{\pi_*\mu}(S|R) \middle|\, 
    \mu \in \mathscr{M}^T(X)  \text{ with }\rho_*\mu = \nu\right\}
   \end{split}
\end{equation}
under the following two additional assumptions:
\begin{itemize}
  \item $X$ and $Y$ have zero topological dimension.
  \item $\nu\in \mathscr{M}^R(Z)$ is ergodic.
\end{itemize}
We will remove these assumptions in \S \ref{section: removing the ergodicity assumptions} and
\S \ref{section: completion of the proof of relativised variational principle}.

Let us briefly explain the role of the first assumption.
A compact metrizable space is said to have \textbf{zero topological dimension} if 
its clopen sets (closed and open sets) form a basis of the topology.
For example, the Cantor set has zero topological dimension.
This condition allows us to use clopen partitions in the study of 
measure-theoretic entropy.
Such partitions can be used simultaneously as measurable partitions and as open coverings.
This makes it easier to transfer information between measure-theoretic and topological dynamics.
This idea of \lq\lq{}zero dimensional trick\rq\rq{} was first systematically developed by 
Downarowicz \cite{Downarowicz}.

\subsection{Dynamical Frostman lemma} \label{subsection: dynamical Frostman lemma}

The proof of (\ref{eq: difficult part of variational principle}) is based on \textit{the dynamical Frostman lemma}
introduced by Feng and Huang \cite[Lemma 3.3]{Feng--Huang}.
The original Frostman lemma is a fundamental result in geometric measure theory.
It states that, given a compact metric space, there exists a probability measure on it which 
obeys the \lq\lq{}power law\rq\rq{} corresponding to the Hausdorff dimension. 
The dynamical Frostman lemma (Lemma \ref{lemma: dynamical Frostman lemma} below) 
is an analogous statement for the theory of weighted topological entropy.

Throughout this subsection we assume that $0\leq w \leq 1$, $(X, T)$ and $(Y, S)$ are dynamical systems
(with metrics $\mathbf{d}$ and $\mathbf{d}^\prime$ on $X$ and $Y$ respectively) and 
$\pi\colon X\to Y$ is an equivariant continuous map.
The following is the main result of this subsection.

\begin{lemma} \label{lemma: dynamical Frostman lemma}
Let $K$ be a (not necessarily invariant) closed subset of $X$.
For any $0<s<\FHh^w(K, T)$ there exist $\varepsilon >0$, $N>0$ and 
a (not necessarily invariant) Borel probability measure $\lambda$ supported on $K$
such that for every $x\in X$ and $n\geq N$ we have 
\[ \lambda\left(B^w_n(x, \varepsilon)\right) \leq e^{-sn}. \]
\end{lemma}

The case of $K=X$ is proved in \cite[Lemma 3.3]{Feng--Huang}.
The proof of this lemma is almost identical to that of \cite[Lemma 3.3]{Feng--Huang}.
However, we give the proof for completeness.

For $s\geq 0$, $\varepsilon>0$, a natural number $N$ and a function $f\colon X\to \mathbb{R}$, 
we define
\[ \mathcal{W}^{w, s}_{N, \varepsilon}(f) = \inf \sum_{k} c_k e^{-s n_k} \]
where the infimum is taken over all at most countable collections
$\{(n_k, x_k, c_k)\}$ that satisfy $n_k\geq N$, $x_k\in X$, $0<c_k<\infty$ and 
\[ \sum_k c_k 1_{B^w_{n_k}(x_k, \varepsilon)} \geq f. \]
Here $1_{B^w_{n_k}(x_k, \varepsilon)}$ denotes the characteristic function of the $w$-weighted Bowen ball 
$B^w_{n_k}(x_k, \varepsilon)$.

For a subset $\Omega \subset X$, we set $\mathcal{W}^{w, s}_{N,\varepsilon}(\Omega) = \mathcal{W}^{w, s}_{N,\varepsilon}(1_{\Omega})$.
The next lemma is a key result. This was proved in \cite[Proposition 3.5]{Feng--Huang}.

\begin{lemma} \label{lemma: average weighted topological entropy}
Let $s\geq 0$, $\varepsilon>0$, $\delta>0$ and $\Omega \subset X$.
For all sufficiently large $N$, we have
\[ \Lambda^{w, s+\delta}_{N,6\varepsilon}(\Omega) \leq 
    \mathcal{W}^{w, s}_{N,\varepsilon}(\Omega) \leq \Lambda^{w, s}_{N,\varepsilon}(\Omega). \]
\end{lemma}

The case of $K=X$ of the following lemma was given in \cite[Lemma 3.10]{Feng--Huang}.

\begin{lemma} \label{lemma: Howroyd method}
Let $s\geq 0$, $\varepsilon>0$, $N\in \mathbb{N}$ and $K\subset X$ be a closed subset.
Suppose that $c:= \mathcal{W}^{w, s}_{N,\varepsilon}(K) >0$.
Then there is a Borel probability measure $\lambda$ supported on $K$ such that for any $n\geq N$ and $x\in X$
\[  \lambda\left(B^w_n(x,\varepsilon)\right)  \leq  \frac{e^{-sn}}{c}\]
\end{lemma}

\begin{proof}
The proof borrows the method of Howroyd \cite{Howroyd}.
We have $0<c<\infty$.
Let $C(X)$ be the space of real-valued continuous functions on $X$.
We define a semi-norm $p$ on $C(X)$ by 
\[ p(f) = \frac{1}{c} \mathcal{W}^{w, s}_{N,\varepsilon}(f\cdot 1_K). \]
Here $1_K$ is the characteristic function of $K$.
This is a semi-norm (that is, it satisfies the triangle inequality and $p(tf) = tp(f)$ for any $t\geq 0$), 
$p(1_X) = 1$, $0\leq p(f) \leq \sup_{x\in X}|f(x)|$ and $p(g) = 0$ if $g\leq 0$.

We consider a linear functional 
\[ \mathbb{R}1_X\to \mathbb{R}, \quad t1_{X}\mapsto t. \]
Applying the Hahn--Banach theorem, we can extend it to a linear functional 
$L\colon C(X)\to \mathbb{R}$ that satisfies 
\[ L(1_X) = 1, \quad -p(-f) \leq L(f) \leq p(f) \text{ for $\forall f\in C(X)$}. \]
If $f\in C(X)$ is a nonnegative function, then $0=-p(-f) \leq L(f)$.
Therefore by the Riesz representation theorem, there is a Borel probability measure $\lambda$ on $X$
for which we have $L(f) = \int_X fd\lambda$.
If $f\in C(X)$ is zero on $K$, then $p(f) = p(-f) = 0$ and hence $L(f)= 0$.
Thus $\lambda$ is supported on $K$.

Let $x\in X$ be an arbitrary point and let $f\in C(X)$ be any continuous function that satisfies 
$0\leq f \leq 1$ and $\supp (f) \subset B^w_n(x, \varepsilon)$ for some $n \geq N$.
Since $f\leq 1_{B^w_n(x, \varepsilon)}$, we have 
$0\leq L(f) \leq p(f) \leq \frac{1}{c}e^{-sn}$.
This holds for any $f\in C(X)$ satisfying the above conditions. Therefore 
$\lambda\left(B^w_n(x, \varepsilon)\right) \leq \frac{1}{c} e^{-sn}$.
\end{proof}

Now the dynamical Frostman lemma (Lemma \ref{lemma: dynamical Frostman lemma}) follows from 
Lemmas \ref{lemma: average weighted topological entropy} and \ref{lemma: Howroyd method}.

\begin{proof}[Proof of Lemma \ref{lemma: dynamical Frostman lemma}]
We take $0<s < s^\prime < \FHh^w(K, T)$.
It follows from the definition of the weighted topological entropy and Lemma \ref{lemma: average weighted topological entropy}
that we can find $\varepsilon >0$ and $M\in \mathbb{N}$ for which we have $c:= \mathcal{W}^{w, s^\prime}_{M,\varepsilon}(K)>0$.
Applying Lemma \ref{lemma: Howroyd method}, there is a Borel probability measure $\lambda$ supported on $K$ that satisfies 
$\lambda\left(B^w_n(x, \varepsilon)\right) \leq e^{-s^\prime n}/c$ for any $x\in X$ and $n\geq M$. 
We choose $N\geq M$ so that $e^{-(s^\prime-s)N}/c \leq 1$.
Then we obtain $\lambda\left(B^w_n(x, \varepsilon)\right) \leq e^{-sn}$ for any $x\in X$ and $n\geq N$.
\end{proof}

\subsection{Construction of invariant probability measures} \label{subsection: construction of invariant probability measures}

Here we prove (\ref{eq: difficult part of variational principle}) under additional assumptions.
We need the following technical result on calculus.
This was proved in \cite[Lemma 5.4]{Feng--Huang}.

\begin{lemma} \label{lemma: calculus}
Let $u\colon \mathbb{N}\to \mathbb{R}$ be a bounded function defined on the set of natural numbers.
Suppose $\lim_{n\to \infty} |u(n+1)-u(n)| = 0$.
Then for any positive numbers $c$ and $r$
\[ \limsup_{n\to \infty} \left(u(\lceil cn\rceil)- u(\lceil rn \rceil)\right) \geq 0. \]
Here $\lceil x \rceil$ denotes the smallest integer greater than or equal to $x$.
\end{lemma}

Throughout the subsection we assume that 
\begin{enumerate}
  \item $0\leq w \leq 1$,
  \item $(X, T)$, $(Y, S)$ and $(Z, R)$ are dynamical systems,
  \item $\pi\colon X\to Y$, $\theta\colon Y\to Z$ and $\rho\colon X\to Z$ are equivariant continuous maps 
with $\rho=  \theta \circ \pi$,
  \item $\rho$ is surjective,
  \item $X$ and $Y$ have zero topological dimension.
\end{enumerate}
\[
\begin{tikzcd}
(X,T) \arrow[r,"\pi"] \arrow[dr, two heads, "\rho"'] & (Y,S) \arrow[d,"\theta"] \\
& (Z,R)
\end{tikzcd}
\]
The condition (5) means that clopen sets (closed and open sets) form a basis for the topology of both
$X$ and $Y$.

We take metrics $\mathbf{d}$ and $\mathbf{d}^\prime$ on $X$ and $Y$ respectively.
By replacing the metric $\mathbf{d}(x, y)$ with 
$\mathbf{d}(x, y)+ \mathbf{d}^\prime(\pi(x), \pi(y))$, if necessary, we can assume that the map 
$\pi\colon (X, \mathbf{d})\to (Y, \mathbf{d}^\prime)$ is one-Lipschitz (that is, 
$\mathbf{d}^\prime(\pi(x), \pi(y))\leq \mathbf{d}(x, y)$ for all $x, y\in X$).

\begin{proposition} \label{proposition: construction of invariant measure}
Let $\nu\in \mathscr{M}^R(Z)$ be an ergodic measure.
Then for $\nu$-almost every $z\in Z$ 
\[ \FHh^w\left(\rho^{-1}(z), \pi, T\right) \leq 
    \sup\{w h_\mu(T|R) + (1-w)h_{\pi_*\mu}(S|R) \mid \mu \in \mathscr{M}^T(X) \text{ with } \rho_*\mu = \nu\}. \]
\end{proposition}

This proposition provides (\ref{eq: difficult part of variational principle}) under the additional assumptions that 
$X$ and $Y$ have zero topological dimension and $\nu$ is ergodic.
In view of the removal of the ergodicity assumption in
\S\ref{section: removing the ergodicity assumptions}, we prove a slightly stronger statement than Proposition 
\ref{proposition: construction of invariant measure}.

We denote by $\mathcal{B}_Z$ the Borel $\sigma$-algebra of $Z$.
Let $s\geq 0$ and let $\alpha$ and $\beta$ be  finite clopen partitions of $X$ and $Y$ respectively.
We define $E^s_{\alpha, \beta}\subset \mathscr{M}^T(X)$ by
\begin{equation} \label{eq: definition of E^s_{alpha, beta}}
 E^s_{\alpha, \beta} = \left\{\mu \in \mathscr{M}^T(X)\middle|   \frac{w}{n} H_\mu\left(\alpha^{n-1}_0|\rho^{-1}\mathcal{B}_Z\right) 
+ \frac{1-w}{n} H_{\pi_*\mu}\left(\beta^{n-1}_0|\theta^{-1}\mathcal{B}_Z\right) \geq s 
\text{ for all $n\geq 1$}\right\}. 
\end{equation}
Since $\alpha$ and $\beta$ are clopen partitions, $E^s_{\alpha, \beta}$ is a closed subset of $\mathscr{M}^T(X)$
with respect to the weak$^*$ topology
by Lemma \ref{lemma: basic properties of relative measure theoretic entropy} (1).
Notice that, for $\mu\in E^s_{\alpha, \beta}$, we have  
\[ w h_\mu(T|R) + (1-w) h_{\pi_*\mu}(S|R) \geq s. \]
We define $\mesh(\alpha)$ as the maximum of $\diam(A, \mathbf{d})$ over $A\in \alpha$.
We define $\mesh(\beta)$ similarly.

Proposition \ref{proposition: construction of invariant measure} follows from the next result.
This is the main result of this section.

\begin{proposition} \label{proposition: construction of invariant measures more detailed}
For any ergodic measure $\nu\in \mathscr{M}^R(Z)$ and a real number $s$ with  
$0\leq s < \int_Z \FHh^w(\rho^{-1}(z),\pi, T) d\nu(z)$,
there exists a positive number $\varepsilon$ such that if finite clopen partitions $\alpha$ and $\beta$ of 
$X$ and $Y$ (respectively) satisfy $\mesh(\alpha) < \varepsilon$ and $\mesh(\beta)  < \varepsilon$ then 
$\nu \in \rho_*(E^s_{\alpha, \beta})$.
\end{proposition}

\begin{proof}
This proposition is the most important part of the proof of Theorem \ref{theorem: relativised variational principle}.
The basic structure of the argument comes from the ideas of Feng and Huang \cite[\S 5]{Feng--Huang}.
Here is the outline of the proof.
First, we use a dynamical Frostman lemma to construct a probability measure
on a typical fiber with good weighted Bowen-ball estimates.
Second, we average its iterates and extract an invariant measure.
The main difficulty is to obtain the correct weighted combination of
conditional entropies from these averaged measures.
Finally, a small technical argument, based on Lemma \ref{lemma: calculus},
removes an error term which appears because the two time scales
$\lceil wn\rceil$ and $n$ are different.

The endpoint cases $w=0, 1$ are simpler and can be treated similarly. 
Thus we assume $0<w<1$.
Since $\nu$ is ergodic, $\FHh^w(\rho^{-1}(z), \pi, T)$ is constant $\nu$-almost everywhere (Lemma \ref{lemma: measurability of fiber entropy}).
Take a point $z_0\in Z$ that satisfies the following two conditions.
 \begin{itemize}
   \item $z_0$ is a generic point of $\nu$, namely $\frac{1}{n}\sum_{k=0}^{n-1}\delta_{R^k z_0} \to \nu$ in the weak$^*$ topology. Here 
   $\delta_{R^k z_0}$ is the delta measure at $R^k z_0$.
   \item  $\FHh^w(\rho^{-1}(z_0), \pi, T) = \FHh^w(\rho^{-1}(z), \pi, T)$ for $\nu$-almost every $z\in Z$.
 \end{itemize}
Set $K=\rho^{-1}(z_0)$. This is a closed subset of $X$.
Applying Lemma \ref{lemma: dynamical Frostman lemma} (dynamical Frostman lemma) to $K$, we can find 
a positive number $\varepsilon$, a natural number $N$ and a Borel probability measure $\lambda$ supported on 
$K$ such that for any $x\in X$ and $n\geq N$ we have 
\[ \lambda\left(B^w_n(x, \pi, \varepsilon)\right) \leq e^{-sn}. \]
We will prove that the statement of the proposition holds for $\varepsilon$ chosen here.

Let $\alpha$ and $\beta$ be clopen partitions of $X$ and $Y$ respectively with $\mesh(\alpha) < \varepsilon$ and 
$\mesh(\beta)< \varepsilon$.
For a natural number $\ell$ we define $\mathscr{M}_\ell$ as the set of $\mu\in \mathscr{M}^T(X)$ that satisfies 
$\rho_*\mu = \nu$ and
\[ \frac{w}{\ell} H_\mu\left(\alpha_0^{\ell-1}| \rho^{-1}\mathcal{B}_Z\right) + \frac{1-w}{\ell} H_{\pi_*\mu}\left(\beta_0^{\ell-1}|\theta^{-1}\mathcal{B}_Z\right)
    \geq s. \]
$\mathscr{M}_\ell$ is a closed (and hence compact) subset of $\mathscr{M}^T(X)$ with respect to the weak$^*$ topology
by Lemma \ref{lemma: basic properties of relative measure theoretic entropy} (1).
We have $E^s_{\alpha, \beta}\cap (\rho_*)^{-1}(\nu) = \bigcap_{\ell=1}^\infty \mathscr{M}_\ell$.
The main task of the proof is to show that $\mathscr{M}_\ell$ is non-empty for all $\ell\geq 1$.
Suppose, for the moment, that this is the case.
Then the proof goes as follows.

Let $\ell_1$ and $\ell_2$ be natural numbers.
We claim that $\mathscr{M}_{\ell_1 \ell_2} \subset \mathscr{M}_{\ell_1}\cap \mathscr{M}_{\ell_2}$.
Indeed, let $\mu\in \mathscr{M}_{\ell_1 \ell_2}$.
Since $\alpha_0^{\ell_1 \ell_2 -1} = \alpha_0^{\ell_1-1} \vee \alpha_{\ell_1}^{2\ell_1-1}\vee \dots \vee \alpha_{(\ell_2-1)\ell_1}^{\ell_1 \ell_2-1}$,
we have 
\[ H_\mu\left(\alpha_0^{\ell_1\ell_2-1}|\rho^{-1}\mathcal{B}_Z\right) \leq 
   \sum_{k=0}^{\ell_2-1}H_\mu\left(\alpha_{k \ell_1}^{(k+1)\ell_1-1}\middle| \rho^{-1}\mathcal{B}_Z\right) \leq
   \ell_2 H_\mu\left(\alpha_0^{\ell_1-1}|\rho^{-1}\mathcal{B}_Z\right). \]
Here we have used $H_\mu\left(\alpha_r^{r+\ell_1-1}\middle| \rho^{-1}\mathcal{B}_Z\right) \leq 
H_\mu\left(\alpha_0^{\ell_1-1}\mid \rho^{-1}\mathcal{B}_Z\right)$ for any $r\geq 0$\footnote{
Since $T^{-r}\rho^{-1}\mathcal B_Z = \rho^{-1}R^{-r}\mathcal B_Z \subset \rho^{-1}\mathcal B_Z$, 
conditioning on \(\rho^{-1}\mathcal B_Z\) is finer than conditioning on
\(T^{-r}\rho^{-1}\mathcal B_Z\). Hence
\[
H_\mu\left(\alpha_r^{r+\ell_1-1}\mid \rho^{-1}\mathcal B_Z\right) =
H_\mu\left(T^{-r}\alpha_0^{\ell_1-1}\mid \rho^{-1}\mathcal B_Z\right) \leq
H_\mu\left(T^{-r}\alpha_0^{\ell_1-1}\mid T^{-r}\rho^{-1}\mathcal B_Z\right)
= H_\mu\left(\alpha_0^{\ell_1-1}\mid \rho^{-1}\mathcal B_Z\right).
\]}.
Hence 
\[ \frac{1}{\ell_1 \ell_2} H_\mu\left(\alpha_0^{\ell_1\ell_2-1}|\rho^{-1}\mathcal{B}_Z\right) 
\leq \frac{1}{\ell_1}H_\mu\left(\alpha_0^{\ell_1-1}|\rho^{-1}\mathcal{B}_Z\right). \]
Similarly we have 
\[ \frac{1}{\ell_1 \ell_2} H_{\pi_*\mu}\left(\beta_0^{\ell_1\ell_2-1}|\theta^{-1}\mathcal{B}_Z\right) 
\leq \frac{1}{\ell_1}H_{\pi_*\mu}\left(\beta_0^{\ell_1-1}|\theta^{-1}\mathcal{B}_Z\right). \]
Therefore 
\begin{align*}
  s & \leq \frac{w}{\ell_1 \ell_2} H_\mu\left(\alpha_0^{\ell_1\ell_2-1}|\rho^{-1}\mathcal{B}_Z\right) +
          \frac{1-w}{\ell_1 \ell_2} H_{\pi_*\mu}\left(\beta_0^{\ell_1\ell_2-1}|\theta^{-1}\mathcal{B}_Z\right) \\
    & \leq  \frac{w}{\ell_1}H_\mu\left(\alpha_0^{\ell_1-1}|\rho^{-1}\mathcal{B}_Z\right) +
        \frac{1-w}{\ell_1}H_{\pi_*\mu}\left(\beta_0^{\ell_1-1}|\theta^{-1}\mathcal{B}_Z\right). 
\end{align*}
This shows $\mu \in \mathscr{M}_{\ell_1}$. Similarly $\mu \in \mathscr{M}_{\ell_2}$.
Thus $\mathscr{M}_{\ell_1 \ell_2} \subset \mathscr{M}_{\ell_1}\cap \mathscr{M}_{\ell_2}$.

Since each $\mathscr{M}_\ell$ is a nonempty compact subset, we can conclude that
\[ E^s_{\alpha, \beta} \cap (\rho_*)^{-1}(\nu) = \bigcap_{\ell=1}^\infty \mathscr{M}_\ell \neq \emptyset.  \]
Namely we have $\nu \in \rho_*(E^s_{\alpha, \beta})$.

Now we move to the proof of $\mathscr{M}_\ell \neq \emptyset$.
Fix a positive number $M$ greater than both the cardinalities of $\alpha$ and $\beta$.
We set $\gamma = \pi^{-1}(\beta)$. This is a finite clopen partition of $X$.
Let $n\geq N$ be a natural number (we will eventually let $n\to \infty$), and consider the partition
\[ 
\alpha_0^{\lceil wn\rceil -1}\vee \gamma_{\lceil wn \rceil}^{n-1}. 
\]
Since $\mesh(\alpha)$ and $\mesh(\beta)$ are smaller than $\varepsilon$ and the map 
$\pi\colon (X, \mathbf{d})\to (Y, \mathbf{d}^\prime)$ is one-Lipschitz, every element $P$
of the partition $\alpha_0^{\lceil wn\rceil -1}\vee \gamma_{\lceil wn \rceil}^{n-1}$ satisfies 
$P\subset B^w_n(x, \pi, \varepsilon)$ for any $x\in P$.
Hence $\lambda(P) \leq \lambda\left(B^w_n(x, \pi, \varepsilon)\right) \leq e^{-sn}$.
Therefore 
\begin{equation} \label{eq: entropy bound of lambda}
  \begin{split}
  sn & \leq    -\sum_{P\in \alpha_0^{\lceil wn\rceil -1}\vee \gamma_{\lceil wn \rceil}^{n-1}} \lambda(P) \log \lambda(P) \\
  & =  H_\lambda\left(\alpha_0^{\lceil wn\rceil -1}\vee \gamma_{\lceil wn \rceil}^{n-1} \middle|\rho^{-1}\mathcal{B}_Z\right),
  \quad 
  \text{since $\lambda$ is supported on $K = \rho^{-1}(z_0)$} \\
  & \leq  H_\lambda\left(\alpha_0^{\lceil wn \rceil -1}\middle| \rho^{-1}\mathcal{B}_Z\right) + 
            H_\lambda\left(\gamma_{\lceil wn \rceil}^{n-1}\middle| \rho^{-1}\mathcal{B}_Z\right) \\
  & \leq       H_\lambda\left(\alpha_0^{\lceil wn \rceil -1}\middle| \rho^{-1}\mathcal{B}_Z\right) + 
            H_{T^{\lceil wn\rceil}_*\lambda}\left(\gamma^{n-\lceil wn \rceil-1}_{0}\middle| \rho^{-1}\mathcal{B}_Z\right).
  \end{split}
\end{equation}
Let $\ell$ be any natural number. (We assume that $n$ is much larger than $\ell$.)
For each $0\leq i < \ell$, we set 
\[ m_i := \left\lfloor \frac{\lceil wn\rceil-i-1}{\ell} \right\rfloor, \quad 
   m^\prime_i :=  \left\lfloor \frac{n-\lceil wn \rceil -i -1}{\ell}\right\rfloor   \]
and consider 
\begin{align*}
  \alpha_0^{\lceil wn \rceil-1} 
  &= \alpha_0^{i-1} \vee \alpha_i^{i+\ell-1} \vee \alpha_{i+\ell}^{i+2\ell-1} \vee \cdots 
       \vee \alpha_{i+(m_i-1)\ell}^{i+m_i\ell-1} \vee \alpha_{i+m_i\ell}^{\lceil wn\rceil-1} \\
   \gamma_0^{n-\lceil wn \rceil -1} & = \gamma_0^{i-1}\vee \gamma_i^{i+\ell-1}\vee \gamma_{i+\ell}^{i+2\ell-1} \vee \cdots
       \vee \gamma_{i+(m^\prime_i-1)\ell}^{i+m^\prime_i \ell-1} \vee \gamma_{i+m^\prime_i\ell}^{n-\lceil wn \rceil -1},
\end{align*}
where the terms $\alpha_0^{i-1}$ and $\gamma_0^{i-1}$ should be omitted when $i=0$.
Then 
\begin{align*}
 H_\lambda\left(\alpha_0^{\lceil wn\rceil-1}\middle| \rho^{-1}\mathcal{B}_Z\right)  \leq & \,
 H_\lambda\left(\alpha_0^{i-1}\middle| \rho^{-1}\mathcal{B}_Z\right) 
 + \sum_{j=0}^{m_i-1} H_\lambda\left(\alpha_{i+j\ell}^{i+(j+1)\ell-1}\middle|\rho^{-1}\mathcal{B}_Z\right) \\
 &+ H_\lambda\left(\alpha_{i+m_i\ell}^{\lceil wn\rceil-1}\middle| \rho^{-1}\mathcal{B}_Z\right).
\end{align*} 
We have $H_\lambda\left(\alpha_{i+j\ell}^{i+(j+1)\ell-1}\middle|\rho^{-1}\mathcal{B}_Z\right) 
\leq  H_{T^{i+j\ell}_*\lambda}\left(\alpha_0^{\ell-1}\middle| \rho^{-1}\mathcal{B}_Z\right)$.
The terms $H_\lambda\left(\alpha_0^{i-1}\middle| \rho^{-1}\mathcal{B}_Z\right)$ and 
$H_\lambda\left(\alpha_{i+m_i\ell}^{\lceil wn\rceil-1}\middle| \rho^{-1}\mathcal{B}_Z\right)$ are bounded from above by 
$\ell\log M$ because $|\alpha| < M$.
Hence 
\[ H_\lambda\left(\alpha_0^{\lceil wn\rceil-1}\middle| \rho^{-1}\mathcal{B}_Z\right) \leq 
   2\ell \log M + \sum_{j=0}^{m_i-1}H_{T^{i+j\ell}_*\lambda}\left(\alpha_0^{\ell-1}\middle| \rho^{-1}\mathcal{B}_Z\right). \]
Similarly we have 
\[ H_{T^{\lceil wn \rceil}_*\lambda}\left(\gamma_0^{n-\lceil wn \rceil -1}\middle| \rho^{-1}\mathcal{B}_Z\right)
   \leq 2\ell \log M + \sum_{j=0}^{m^\prime_i-1}H_{T^{\lceil wn \rceil+i+j\ell}_*\lambda}\left(\gamma_0^{\ell-1}\middle| \rho^{-1}\mathcal{B}_Z\right). \]
Adding these inequalities over $i=0, 1, \dots, \ell-1$, we obtain 
\begin{equation}  \label{eq: entropy estimate by partitioning}
  \begin{split}
  & \ell H_\lambda\left(\alpha_0^{\lceil wn\rceil-1}\middle| \rho^{-1}\mathcal{B}_Z\right) 
  + \ell H_{T^{\lceil wn \rceil}_*\lambda}\left(\gamma_0^{n-\lceil wn \rceil -1}\middle| \rho^{-1}\mathcal{B}_Z\right) \\
  &\leq   4\ell^2 \log M + \sum_{k=0}^{\lceil wn \rceil-1} H_{T^k_*\lambda}\left(\alpha_0^{\ell-1}\middle| \rho^{-1}\mathcal{B}_Z\right)
             + \sum_{k=\lceil wn \rceil}^{n-1} H_{T^k_*\lambda} \left(\gamma_0^{\ell-1}\middle| \rho^{-1}\mathcal{B}_Z\right).
  \end{split}
\end{equation}
Since the entropy is concave in the underlying measures 
(see the inequality \eqref{eq: concavity of entropy} in Lemma \ref{lemma: concavity of entropy and reverse}),
\[ \sum_{k=0}^{\lceil wn \rceil-1} H_{T^k_*\lambda}\left(\alpha_0^{\ell-1}\middle| \rho^{-1}\mathcal{B}_Z\right)
   \leq \lceil wn \rceil H_{\frac{1}{\lceil wn\rceil}\sum_{k=0}^{\lceil wn\rceil-1} T^k_*\lambda}\left(\alpha_0^{\ell-1}\middle| \rho^{-1}\mathcal{B}_Z\right), \]
\[  \sum_{k=\lceil wn \rceil}^{n-1} H_{T^k_*\lambda} \left(\gamma_0^{\ell-1}\middle| \rho^{-1}\mathcal{B}_Z\right) \leq 
    (n-\lceil wn \rceil) H_{\frac{1}{n-\lceil wn \rceil}\sum_{k=\lceil wn \rceil}^{n-1} T^k_*\lambda}\left(\gamma_0^{\ell-1}\middle| \rho^{-1}\mathcal{B}_Z\right). \]
Set 
\[ \mu_n = \frac{1}{n}\sum_{k=0}^{n-1} T^k_*\lambda. \]
Then the first inequality provides 
\[ \sum_{k=0}^{\lceil wn \rceil-1} H_{T^k_*\lambda}\left(\alpha_0^{\ell-1}\middle| \rho^{-1}\mathcal{B}_Z\right) \leq 
    \lceil wn \rceil H_{\mu_{\lceil wn \rceil}}\left(\alpha_0^{\ell-1}\middle| \rho^{-1}\mathcal{B}_Z\right). \]
Note that 
\[ \mu_n = \frac{\lceil wn \rceil}{n} \mu_{\lceil wn \rceil} 
  + \frac{n-\lceil wn \rceil}{n}\left(\frac{1}{n-\lceil wn \rceil}\sum_{k=\lceil wn \rceil}^{n-1} T^k_*\lambda\right). \]
Using the concavity (the inequality \eqref{eq: concavity of entropy} in Lemma \ref{lemma: concavity of entropy and reverse}) again,
\begin{align*}
 & \frac{\lceil wn \rceil}{n}H_{\mu_{\lceil wn \rceil}}\left(\gamma_0^{\ell-1}\middle|\rho^{-1}\mathcal{B}_Z\right) +
  \frac{n-\lceil wn \rceil}{n} 
  H_{\frac{1}{n-\lceil wn \rceil}\sum_{k=\lceil wn \rceil}^{n-1} T^k_*\lambda}\left(\gamma_0^{\ell-1}\middle|\rho^{-1}\mathcal{B}_Z\right)
    \\
 & \leq H_{\mu_n}\left(\gamma_0^{\ell-1}\middle|\rho^{-1}\mathcal{B}_Z\right). 
\end{align*}  
Therefore 
\[ \sum_{k=\lceil wn \rceil}^{n-1} H_{T^k_*\lambda} \left(\gamma_0^{\ell-1}\middle| \rho^{-1}\mathcal{B}_Z\right)
\leq n H_{\mu_n}\left(\gamma_0^{\ell-1}\middle| \rho^{-1}\mathcal{B}_Z\right) 
- \lceil wn \rceil H_{\mu_{\lceil wn \rceil}}\left(\gamma_0^{\ell-1}\middle| \rho^{-1}\mathcal{B}_Z\right). \]
Combining these estimates with \eqref{eq: entropy estimate by partitioning}, we obtain 
\begin{align*}
  &\ell H_\lambda\left(\alpha_0^{\lceil wn\rceil-1}\middle| \rho^{-1}\mathcal{B}_Z\right) 
  + \ell H_{T^{\lceil wn \rceil}_*\lambda}\left(\gamma_0^{n-\lceil wn \rceil -1}\middle| \rho^{-1}\mathcal{B}_Z\right) \\
   & \leq  4 \ell^2 \log M + \lceil wn \rceil H_{\mu_{\lceil wn \rceil}}\left(\alpha_0^{\ell-1}\middle| \rho^{-1}\mathcal{B}_Z\right) \\
    & +n H_{\mu_n}\left(\gamma_0^{\ell-1}\middle| \rho^{-1}\mathcal{B}_Z\right) 
- \lceil wn \rceil H_{\mu_{\lceil wn \rceil}}\left(\gamma_0^{\ell-1}\middle| \rho^{-1}\mathcal{B}_Z\right). 
\end{align*}
From \eqref{eq: entropy bound of lambda}
\begin{equation} \label{eq: lower bound on the combination of entropy}
   \begin{split}
    s \leq & \frac{4\ell \log M}{n} + \frac{\lceil wn \rceil}{n\ell}H_{\mu_{\lceil wn \rceil}}\left(\alpha_0^{\ell-1}\middle|\rho^{-1}\mathcal{B}_Z\right) \\
     &+ \frac{1}{\ell} H_{\mu_n}\left(\gamma_0^{\ell-1}\middle| \rho^{-1}\mathcal{B}_Z\right)
     - \frac{\lceil wn \rceil}{n \ell}  H_{\mu_{\lceil wn \rceil}}\left(\gamma_0^{\ell-1}\middle| \rho^{-1}\mathcal{B}_Z\right) \\
      = & \frac{4\ell \log M}{n} + \frac{\lceil wn \rceil}{n\ell}H_{\mu_{\lceil wn \rceil}}\left(\alpha_0^{\ell-1}\middle|\rho^{-1}\mathcal{B}_Z\right) + 
       \frac{n-\lceil wn \rceil}{n\ell} H_{\mu_{\lceil wn \rceil}}\left(\gamma_0^{\ell-1}\middle| \rho^{-1}\mathcal{B}_Z\right) \\
       &+ \frac{1}{\ell}H_{\mu_n}\left(\gamma_0^{\ell-1}\middle| \rho^{-1}\mathcal{B}_Z\right) 
       - \frac{1}{\ell} H_{\mu_{\lceil wn \rceil}}\left(\gamma_0^{\ell-1}\middle| \rho^{-1}\mathcal{B}_Z\right).
   \end{split}
\end{equation}
This estimate is close to the desired one. 
If the last terms 
\[ \frac{1}{\ell}H_{\mu_n}\left(\gamma_0^{\ell-1}\middle| \rho^{-1}\mathcal{B}_Z\right) 
       - \frac{1}{\ell} H_{\mu_{\lceil wn \rceil}}\left(\gamma_0^{\ell-1}\middle| \rho^{-1}\mathcal{B}_Z\right)\]
were not present, we could immediately obtain the conclusion by letting $n\to \infty$.
So the remaining task is to show that this error term does not affect the limiting lower bound.

Set $u(n) = \frac{1}{\ell}H_{\mu_n}\left(\gamma_0^{\ell-1}\middle|\rho^{-1}\mathcal{B}_Z\right)$.
We have $0\leq u(n) \leq \log M$, namely it is bounded.

\begin{claim} \label{claim: removing an error term by concavity and reverse}
$|u(n+1)-u(n)|\to 0$ as $n\to \infty$.
\end{claim}

\begin{proof}
We have $\mu_{n+1} = \frac{n}{n+1}\mu_n + \frac{1}{n+1}T^n_*\lambda$.
We apply Lemma \ref{lemma: concavity of entropy and reverse} with $t=\frac{1}{n+1}$ and obtain 
  \begin{align*}
  & \frac{n}{n+1}H_{\mu_n}\left(\gamma_0^{\ell-1}\middle|\, \rho^{-1}\mathcal{B}_Z\right) 
  + \frac{1}{n+1} H_{T^n_*\lambda}\left(\gamma_0^{\ell-1}\middle|\, \rho^{-1}\mathcal{B}_Z\right) \\
  & \leq H_{\mu_{n+1}}\left(\gamma_0^{\ell-1}\middle|\, \rho^{-1}\mathcal{B}_Z\right) \\
  & \leq   \frac{n}{n+1}H_{\mu_n}\left(\gamma_0^{\ell-1}\middle|\, \rho^{-1}\mathcal{B}_Z\right) 
  + \frac{1}{n+1} H_{T^n_*\lambda}\left(\gamma_0^{\ell-1}\middle|\, \rho^{-1}\mathcal{B}_Z\right)
  + h\left(\frac{1}{n+1}\right),
  \end{align*}
where $h(\frac{1}{n+1}) = \frac{1}{n+1}\log (n+1) + \frac{n}{n+1}\log \frac{n+1}{n}$.
Therefore 
\[ |u(n+1)-u(n)| \leq \frac{u(n)}{n+1} + \frac{1}{\ell(n+1)}H_{T^n_*\lambda}\left(\gamma_0^{\ell-1}\middle|\, \rho^{-1}\mathcal{B}_Z\right)
    + \frac{1}{\ell} h\left(\frac{1}{n+1}\right). \]
The right-hand side goes to zero as $n\to \infty$ because 
$u(n)$ and 
$\frac{1}{\ell}H_{T^n_*\lambda}\left(\gamma_0^{\ell-1}\middle|\, \rho^{-1}\mathcal{B}_Z\right)$ are both bounded by 
$\log M$ and $h(\frac{1}{n+1}) \to 0$.
\end{proof}

Now we can apply Lemma \ref{lemma: calculus} to $u(n)$ and obtain 
\[ \limsup_{n\to \infty} \left(u(\lceil wn\rceil) - u(n)\right) \geq 0. \]
So there exists a subsequence $n_1<n_2<n_3<\dots$ for which we have 
\[  \lim_{k\to \infty} \left(u(\lceil wn_k\rceil) - u(n_k)\right) \geq 0. \]
Namely
\[ 
 \lim_{k\to \infty} \left\{\frac{1}{\ell}H_{\mu_{\lceil w n_k\rceil}}\left(\gamma_0^{\ell-1}\middle|\, \rho^{-1}\mathcal{B}_Z\right)
  - \frac{1}{\ell} H_{\mu_{n_k}}\left(\gamma_0^{\ell-1}\middle|\, \rho^{-1}\mathcal{B}_Z\right)\right\} \geq 0.
  \]

Passing to a further subsequence (which we also denote by $\{n_k\}$), we can assume that 
$\mu_{\lceil w n_k\rceil}$ converges to $\mu\in \mathscr{M}^T(X)$ in the weak$^*$ topology.
We have \(\rho_*\mu=\nu\), because
\[
\rho_*\mu_{\lceil wn_k\rceil}
=
\frac{1}{\lceil wn_k\rceil}
\sum_{j=0}^{\lceil wn_k\rceil-1}\delta_{R^jz_0}
\to \nu
\]
as \(k\to\infty\), by the genericity of \(z_0\).

Since $\alpha$ and $\gamma$ are clopen, 
by Lemma \ref{lemma: basic properties of relative measure theoretic entropy} (1)
\begin{align*}
 \limsup_{k\to \infty} H_{\mu_{\lceil w n_k\rceil}} \left(\alpha_0^{\ell-1}\middle| \rho^{-1}\mathcal{B}_Z\right) 
 & \leq H_\mu\left(\alpha_0^{\ell-1}\middle| \rho^{-1}\mathcal{B}_Z\right), \\
  \limsup_{k\to \infty} H_{\mu_{\lceil w n_k\rceil}} \left(\gamma_0^{\ell-1}\middle| \rho^{-1}\mathcal{B}_Z\right) 
 & \leq H_\mu\left(\gamma_0^{\ell-1}\middle| \rho^{-1}\mathcal{B}_Z\right).
\end{align*}
Now we combine these estimates with \eqref{eq: lower bound on the combination of entropy} and conclude 
\[ s \leq \frac{w}{\ell} H_\mu \left(\alpha_0^{\ell-1}\middle| \rho^{-1}\mathcal{B}_Z\right) 
  + \frac{1-w}{\ell} H_\mu \left(\gamma_0^{\ell-1}\middle| \rho^{-1}\mathcal{B}_Z\right). \]
Since $\gamma = \pi^{-1}\beta$, we have $H_\mu \left(\gamma_0^{\ell-1}\middle| \rho^{-1}\mathcal{B}_Z\right)
= H_{\pi_*\mu} \left(\beta_0^{\ell-1}\middle| \theta^{-1}\mathcal{B}_Z\right)$.
Therefore we obtain $\mu \in \mathscr{M}_\ell$.
This shows that $\mathscr{M}_\ell$ is non-empty and, as we have already observed, this finishes the proof.
\end{proof}

\section{Removing the ergodicity assumption} \label{section: removing the ergodicity assumptions}

In this section, we remove the ergodicity assumption on $\nu$.
(For the application to Bedford--McMullen carpets, it is enough to consider only the case that 
$\nu$ is ergodic. Therefore readers may skip this section if their primary interest is in Theorem 
\ref{theorem: intersection of Bedford--McMullen carpets}.)
At first sight, this may seem to follow immediately from
Proposition \ref{proposition: construction of invariant measure}
and the ergodic decomposition theorem.
There is, however, a minor measurability issue: one has to choose,
for almost every ergodic component, a nearly optimal lift in a measurable way.
We handle this using the measurable selection theorem of Kuratowski and Ryll--Nardzewski.
We first explain where the difficulty lies.

Throughout this section,
as in \S \ref{subsection: construction of invariant probability measures}, we assume the following.
\begin{enumerate}
  \item $0\leq w \leq 1$,
  \item $(X, T)$, $(Y, S)$ and $(Z, R)$ are dynamical systems,
  \item $\pi\colon X\to Y$, $\theta\colon Y\to Z$ and $\rho\colon X\to Z$ are equivariant continuous maps 
with $\rho= \theta \circ \pi$,
  \item $\rho$ is surjective,
  \item $X$ and $Y$ have zero topological dimension.
\end{enumerate}
\[
\begin{tikzcd}
(X,T) \arrow[r,"\pi"] \arrow[dr, two heads, "\rho"'] & (Y,S) \arrow[d,"\theta"] \\
& (Z,R)
\end{tikzcd}
\]

We assume that $\mathscr{M}^T(X), \mathscr{M}^S(Y)$ and $\mathscr{M}^R(Z)$ are endowed with the 
weak$^*$ topology and the Borel $\sigma$-algebra.
We would like to prove 

\begin{proposition} \label{proposition: construction of invariant measures without ergodicity}
For any (not necessarily ergodic) $\nu\in \mathscr{M}^R(Z)$
\begin{equation*} 
  \begin{split}
   & \int_Z \FHh^w\left(\rho^{-1}(z),\pi, T\right)  d\nu(z) \\
   & \leq \sup\left\{w h_\mu(T|R) + (1-w) h_{\pi_*\mu}(S|R) \middle|\, 
    \mu \in \mathscr{M}^T(X)  \text{ with }\rho_*\mu = \nu\right\}.
   \end{split}
\end{equation*}
\end{proposition}

Let $\mathscr{M}^R_{\mathrm{erg}}(Z)$ be the set of ergodic measures on $(Z, R)$.
This is a Borel subset of $\mathscr{M}^R(Z)$.
For $N\geq 1$, we set 
\[ f_N(z) = \min\left(\FHh^w\left(\rho^{-1}(z), \pi, T\right),N\right), \quad  (z\in Z). \]
Proposition \ref{proposition: construction of invariant measure} implies that for any $\delta>0$ and 
$\lambda \in \mathscr{M}^R_{\mathrm{erg}}(Z)$
there exists $\mu\in \mathscr{M}^T(X)$ such that $\rho_*\mu = \lambda$ and
\[ w h_\mu(T|R) + (1-w) h_{\pi_*\mu}(S|R) \geq \int_Z f_N(z) d\lambda(z) -\delta. \]
Then it is natural to expect that the following lemma holds.

\begin{lemma} \label{lemma: existence of measurable section}
For any $\delta>0$ there exists a measurable map $\phi\colon \mathscr{M}^R_{\mathrm{erg}}(Z) \to \mathscr{M}^T(X)$ 
that satisfies $\rho_*\left(\phi(\lambda)\right) = \lambda$ and
\[ w h_{\phi(\lambda)}(T|R) + (1-w) h_{\pi_*\phi(\lambda)}(S|R) \geq \int_Z f_N(z) d\lambda(z) -\delta \]
for every $\lambda \in \mathscr{M}^R_{\mathrm{erg}}(Z)$.
\end{lemma}

The point of the lemma is the measurability of $\phi$. 
If we do not assume it, 
then the statement is an immediate consequence of Proposition \ref{proposition: construction of invariant measure}
and the axiom of choice.

Assume this lemma for the moment. Then the proof of Proposition \ref{proposition: construction of invariant measures without ergodicity}
goes as follows.
Let $\nu\in \mathscr{M}^R(Z)$ be a (not necessarily ergodic) invariant probability measure.
We consider its ergodic decomposition:
\begin{equation*}
   \nu = \int_{\mathscr{M}^R_{\mathrm{erg}}(Z)} \lambda \, dp(\lambda).
\end{equation*} 
Here $p$ is a Borel probability measure on $\mathscr{M}^R_{\mathrm{erg}}(Z)$.
We define $\mu \in \mathscr{M}^T(X)$ by 
\[ \mu = \int_{\mathscr{M}^R_{\mathrm{erg}}(Z)} \phi(\lambda) \, dp(\lambda). \]
Then 
\begin{align*}
  w h_\mu(T|R) + (1-w) h_{\pi_*\mu}(S|R) & = \int_{\mathscr{M}^R_{\mathrm{erg}}(Z)}
  \left(w h_{\phi(\lambda)}(T|R) + (1-w) h_{\pi_*\phi(\lambda)}(S|R)\right) dp(\lambda) \\
 & \geq \int_{\mathscr{M}^R_{\mathrm{erg}}(Z)}\left(\int_Z f_N(z)  d\lambda(z) -\delta\right) dp(\lambda) \\
 &= \int_Z f_N(z)  d\nu(z) -\delta.
\end{align*}
Letting $N\to \infty$ and $\delta\to 0$, 
we obtain Proposition \ref{proposition: construction of invariant measures without ergodicity}.

Now the remaining problem is to prove Lemma \ref{lemma: existence of measurable section}.
A main ingredient of our proof is the following \textit{measurable selection theorem} due to Kuratowski and 
Ryll--Nardzewski; see the book of Srivastava \cite[p.189, Theorem 5.2.1]{Srivastava}
for the proof. (The proof is short and interesting.)

\begin{theorem}[Measurable selection theorem] \label{theorem: measurable selection theorem}
Let $\mathcal{X}$ be a complete separable metric space.
Let $\mathcal{B}$ be its Borel $\sigma$-algebra, and $\mathrm{Cl}(\mathcal{X})$ be the set of 
nonempty closed subsets of $\mathcal{X}$.
Let $(\Omega, \mathcal{F})$ be a measurable space, and 
$F\colon \Omega \to \mathrm{Cl}(\mathcal{X})$ a map.
Suppose $F$ is weakly measurable, that is, for any open set $U\subset \mathcal{X}$, the set 
$\{\omega\in \Omega\mid F(\omega) \cap U \neq \emptyset\}$ belongs to the $\sigma$-algebra $\mathcal{F}$.
Then there exists a measurable map $\phi\colon \Omega \to \mathcal{X}$, with respect to $\mathcal{B}$ and $\mathcal{F}$, 
that satisfies 
$\phi(\omega) \in F(\omega)$ for all $\omega\in \Omega$.
\end{theorem}

As in \S \ref{subsection: construction of invariant probability measures}, we take metrics $\mathbf{d}$ and $\mathbf{d}^\prime$
on $X$ and $Y$ respectively so that the map $\pi\colon (X, \mathbf{d})\to (Y, \mathbf{d}^\prime)$ is one-Lipschitz.

Let $s\geq 0$, and let $\alpha$ and $\beta$ be finite clopen partitions of $X$ and $Y$ respectively.
We define a closed subset $E^s_{\alpha, \beta}\subset \mathscr{M}^T(X)$ by (\ref{eq: definition of E^s_{alpha, beta}}). 
Let $\rho_*\left(E^s_{\alpha, \beta}\right)$ be the image of $E^s_{\alpha, \beta}$ by the map 
$\rho_*\colon \mathscr{M}^T(X)\to \mathscr{M}^R(Z)$.
For each $\nu\in \rho_*\left(E^s_{\alpha, \beta}\right)$ we define 
\[ E^s_{\alpha, \beta, \nu} = \{\mu\in E^s_{\alpha, \beta}\mid \rho_*\mu = \nu\}. \]
This is a nonempty closed subset of $\mathscr{M}^T(X)$.

\begin{claim} \label{claim: weak measurability}
The map 
\[ \rho_*\left(E^s_{\alpha, \beta}\right) \to \mathrm{Cl}\left(\mathscr{M}^T(X)\right),
 \quad \nu \mapsto E^s_{\alpha, \beta, \nu} \]
 is weakly measurable.
\end{claim}

\begin{proof}
Let $U\subset \mathscr{M}^T(X)$ be an open set.
We need to show that the set 
\[ \{\nu\in \rho_*\left(E^s_{\alpha, \beta}\right)\mid E^s_{\alpha, \beta, \nu}\cap U \neq \emptyset\}
     = \rho_*\left(E^s_{\alpha, \beta}\cap U\right) \]
is a measurable subset of $\rho_*\left(E^s_{\alpha, \beta}\right)$.

Since $\mathscr{M}^T(X)$ is a compact metrizable space, there are compact subsets $F_n\subset \mathscr{M}^T(X)$ satisfying 
$U = \bigcup_{n=1}^\infty F_n$.
Then 
\[ \rho_*\left(E^s_{\alpha, \beta}\cap U\right) = \rho_*\left(\bigcup_{n=1}^\infty (F_n\cap E^s_{\alpha, \beta})\right)
   = \bigcup_{n=1}^\infty \rho_*(F_n\cap E^s_{\alpha, \beta}). \]
Since $\rho_*\colon \mathscr{M}^T(X)\to \mathscr{M}^R(Z)$ is continuous and 
$F_n\cap E^s_{\alpha, \beta}$ is a compact subset of $\mathscr{M}^T(X)$, the set 
$\rho_*\left(F_n\cap E^s_{\alpha, \beta}\right)$ is a compact subset of $\mathscr{M}^R(Z)$.
In particular, it is measurable.
Thus $\rho_*\left(E^s_{\alpha, \beta}\cap U\right)$ is also measurable.
\end{proof}

Now we can use Theorem \ref{theorem: measurable selection theorem} and obtain a measurable map 
$\phi^s_{\alpha, \beta}\colon \rho_*\left(E^s_{\alpha, \beta}\right) \to \mathscr{M}^T(X)$ such that 
$\phi^s_{\alpha, \beta}(\nu) \in E^s_{\alpha, \beta, \nu}$ for every $\nu\in \rho_*\left(E^s_{\alpha, \beta}\right)$.

Fix a natural number $N$ and a positive number $\delta$.
Set $J = \{0, \delta, 2 \delta, 3 \delta, 4 \delta, \dots\}$.
For $n\geq 1$, we take finite clopen partitions $\alpha_n$ and $\beta_n$ of $X$ and $Y$ respectively that satisfy 
$\mesh(\alpha_n) \to 0$ and $\mesh(\beta_n)\to 0$ as $n\to \infty$.

For $s\in J$ with $s>0$ and $n\geq 1$ we define 
\[ F_n^s = \left\{\nu\in \rho_*\left(E^s_{\alpha_n, \beta_n}\right) \middle| s < \int_Z f_N(z) d\nu(z) \leq s+ \delta\right\}. \] 
For $s=0$ we set 
\[ F_n^0 = \left\{\nu\in \mathscr{M}^R(Z) \middle| 0 \leq \int_Z f_N(z) d\nu(z) \leq \delta\right\}. \] 
(Notice that $\rho_*\left(E^0_{\alpha_n, \beta_n}\right) = \rho_*(\mathscr{M}^T(X)) = \mathscr{M}^R(Z)$ and that 
$F_n^0$ is independent of the choice of $n$.)
The sets $F_n^s$ are measurable subsets of $\mathscr{M}^R(Z)$.
If $s\neq s^\prime$ then $F_n^s\cap F_{n^\prime}^{s^\prime} = \emptyset$.

It follows from Proposition \ref{proposition: construction of invariant measures more detailed} that 
every $\nu\in \mathscr{M}_{\mathrm{erg}}^R(Z)$ belongs to $F^s_n$ for some $s\in J$ and $n\geq 1$.
Set $G_1^s = F_1^s$ and
\[ G_{n+1}^s = F_{n+1}^s\setminus \left(F_1^s\cup F_2^s\cup \dots \cup F_n^s\right) \quad (n\geq 1). \]
Then $G_n^s$ $(s\in J, n\geq 1)$ are pairwise disjoint and 
$\bigcup_{s\in J}\bigcup_{n=1}^\infty F_n^s = \bigcup_{s\in J}\bigcup_{n=1}^\infty G_n^s$.
We have 
\[ \mathscr{M}_{\mathrm{erg}}^R(Z) \subset \bigcup_{s\in J}\bigcup_{n=1}^\infty G_n^s. \]
Since $G_n^s \subset \rho_*\left(E^s_{\alpha_n, \beta_n}\right)$, we can define a measurable map 
$\phi\colon \bigcup_{s\in J}\bigcup_{n=1}^\infty G_n^s \to \mathscr{M}^T(X)$ by 
\[ \phi(\nu) = \phi^s_{\alpha_n,\beta_n}(\nu), \quad (\nu\in G_n^s). \]
This satisfies
\[  \rho_*\left(\phi(\nu)\right) = \nu, \quad 
    w h_{\phi(\nu)}(T|R) + (1-w) h_{\pi_*\phi(\nu)}(S|R) \geq \int_Z f_N(z) d\nu(z) -\delta.\]
Restricting $\phi$ to $\mathscr{M}_{\mathrm{erg}}^R(Z)$, we have proved Lemma \ref{lemma: existence of measurable section}.

\section{Completion of the proof of Theorem \ref{theorem: relativised variational principle}}
\label{section: completion of the proof of relativised variational principle}

In this section we establish Theorem \ref{theorem: relativised variational principle}.

\subsection{Zero dimensional principal extension} \label{subsection: zero dimensional principal extension}

We have so far proved Theorem \ref{theorem: relativised variational principle} under the additional assumption that 
the spaces $X$ and $Y$ have zero topological dimension.
The remaining problem is how to remove this artificial assumption.
In this subsection we prepare a technical device for that purpose.

Let $(X, T)$ and $(Y, S)$ be dynamical systems.
An equivariant continuous map $\pi\colon X\to Y$ is called a \textbf{factor map} if $\pi$ is surjective.
Let $\pi\colon X\to Y$ be a factor map and let $\mathbf{d}$ be a metric on $X$.
We define the \textbf{topological conditional entropy} of $\pi$ by 
\[ \h(X, T|Y,S) = \lim_{\varepsilon\to 0}
\left(\lim_{N\to \infty} \frac{\sup_{y \in Y} \log \#(\pi^{-1}(y), \mathbf{d}_N, \varepsilon)}{N}\right). \]
Here $\#(\pi^{-1}(y), \mathbf{d}_N, \varepsilon)$ is the $\varepsilon$-covering number of $\pi^{-1}(y)$ with respect to $\mathbf{d}_N$.
By the relativised variational principle of Ledrappier and Walters (and some additional variational principle) \cite[Theorem 6.8.8]{Downarowicz}
\begin{equation} \label{eq: conditional variational principle}
 \h(X, T|Y,S) =\sup_{\mu\in \mathscr{M}^T(X)} h_\mu(T|S). 
\end{equation} 
The factor map $\pi$ is said to be \textbf{principal} if $\h(X, T|Y, S) = 0$.
In this case, $(X, T)$ is called a \textbf{principal extension} of $(Y, S)$.
If $\pi\colon (X, T)\to (Y, S)$ and $\theta\colon (Y, S)\to (Z, R)$ are both principal factor maps, then 
the composition $\theta\circ \pi\colon (X, T)\to (Z, R)$ is also principal.
(This follows from the variational principle \eqref{eq: conditional variational principle} and 
$h_\mu(T|R) = h_\mu(T|S) + h_{\pi_*\mu}(S|R)$.)

\begin{lemma} \label{lemma: principal extension and weighted entropy sum}
Let $(X_i, T_i)$, $(Y_i, S_i)$ $(i=1,2)$ and $(Z, R)$ be dynamical systems.
Suppose we are given the following commutative diagram.
\[
\begin{tikzcd}
(X_1,T_1) \arrow[r,"\pi_1"] \arrow[d, two heads, "f"']
  & (Y_1,S_1) \arrow[d, two heads, "g"] \\
(X_2,T_2) \arrow[r,"\pi_2"'] \arrow[dr, "\rho"']
  & (Y_2,S_2) \arrow[d,"\theta"] \\
& (Z,R)
\end{tikzcd}
\]
We assume that $f$ and $g$ are principal factor maps.
Then for any $0\leq w\leq 1$ and $\nu\in \mathscr{M}^R(Z)$
\begin{align*}
  & \sup\{w h_\mu(T_1|R) + (1-w)h_{(\pi_1)_*\mu}(S_1|R)\mid \mu\in \mathscr{M}^{T_1}(X_1) \text{ with } (\rho\circ f)_*\mu = \nu\} \\
  & =  \sup\{w h_\lambda(T_2|R) + (1-w)h_{(\pi_2)_*\lambda}(S_2|R)\mid \lambda \in \mathscr{M}^{T_2}(X_2) \text{ with }
         \rho_*\lambda = \nu\}.
\end{align*}
\end{lemma}

\begin{proof}
The map $f_*\colon \mathscr{M}^{T_1}(X_1) \to \mathscr{M}^{T_2}(X_2)$ is surjective by the Hahn--Banach theorem.
By the definition of principal factor maps and the variational principle \eqref{eq: conditional variational principle},
for $\mu \in \mathscr{M}^{T_1}(X_1)$
\begin{align*}
  h_\mu(T_1|R) &= h_\mu(T_1|T_2) + h_{f_*\mu}(T_2|R) = h_{f_*\mu}(T_2|R), \\
  h_{(\pi_1)_*\mu}(S_1|R) &= h_{(\pi_1)_*\mu}(S_1|S_2) + h_{(g\circ \pi_1)_*\mu}(S_2|R) = h_{(\pi_2\circ f)_*\mu}(S_2|R).
\end{align*}
Since \(f_*\colon \mathscr M^{T_1}(X_1)\to \mathscr M^{T_2}(X_2)\) is surjective,
taking the supremum over all admissible \(\mu\) gives the desired equality.
\end{proof}

\begin{lemma} \label{lemma: fiber product}
Let $(X,T)$, $(Y, S)$ and $(Y^\prime, S^\prime)$ be dynamical systems.
Let $\pi\colon (X,T)\to (Y, S)$ be an equivariant continuous map and $\phi\colon (Y^\prime, S^\prime)\to (Y, S)$
a principal factor map.
We define the fiber product 
\[ X\times_{Y} Y^\prime := \{(x, y)\in X\times Y^\prime\mid \pi(x) = \phi(y)\}. \]
The pair $(X\times_Y Y^\prime, T\times S^\prime)$ becomes a dynamical system.
Then the map 
\[ \varphi\colon X\times_Y Y^\prime \to X, \quad (x, y)\mapsto x \]
is a principal factor map.
\[
\begin{tikzcd}
(X\times_Y Y',\,T\times S') \arrow[r] \arrow[d,"\varphi"']
  & (Y',S') \arrow[d,"\phi: \text{ principal}"] \\
(X,T) \arrow[r,"\pi"']
  & (Y,S)
\end{tikzcd}
\]
\end{lemma}

\begin{proof}
This was proved in \cite[Lemma 5.3]{Tsukamoto}. We repeat the proof here for completeness.
Take metrics $\mathbf{d}$ and $\mathbf{d}^\prime$ on $X$ and $Y^\prime$ respectively.
We define a metric $D$ on $X\times_Y Y^\prime$ by 
\[ D\left((x_1, y_1), (x_2, y_2)\right) = \max\left(\mathbf{d}(x_1, x_2), \mathbf{d}^\prime(y_1, y_2)\right). \]
For any $N\geq 1$ and $x\in X$, the metric space $(\varphi^{-1}(x), D_N)$ is 
isometric to $(\phi^{-1}(\pi(x)), \mathbf{d}^\prime_N)$.
Hence for any $\varepsilon>0$ 
\[ \#\left(\varphi^{-1}(x), D_N, \varepsilon\right) = \#\left(\phi^{-1}(\pi(x)), \mathbf{d}^\prime_N,\varepsilon\right). \]
Then $\h\left(X\times_Y Y^\prime, T\times S^\prime\middle| X, T\right) \leq \h(Y^\prime, S^\prime|Y, S) = 0$.
\end{proof}

The next theorem is due to Downarowicz and Huczek (\cite[Theorem 7.6.1]{Downarowicz}, 
\cite{Downarowicz--Huczek}).
Here recall again that a compact metrizable space is said to be zero (topological) dimensional if 
clopen sets form a basis of the topology.

\begin{theorem}[Existence of zero dimensional principal extension] \label{theorem: zero dimensional principal extension}
For any dynamical system $(X, T)$, there exists a principal extension 
$\phi \colon (X^\prime, T^\prime) \to (X, T)$ such that $X^\prime$ is zero dimensional.
\end{theorem}

\begin{corollary} \label{cor: zero dimensional trick}
Let $\pi\colon (X, T)\to (Y, S)$ be an equivariant continuous map between dynamical systems.
Then we can construct the following commutative diagram so that 
$f$ and $g$ are principal factor maps and that $X^\prime$ and $Y^\prime$ are zero dimensional.
\[
\begin{tikzcd}
(X^\prime,T^\prime) \arrow[r,"\pi^\prime"] \arrow[d, two heads,"f"'] &  (Y^\prime,S^\prime)    \arrow[d,two heads, "g"] \\
 (X,T) \arrow[r,"\pi"'] & (Y,S)
\end{tikzcd}
\]
\end{corollary}

\begin{proof}
By Theorem \ref{theorem: zero dimensional principal extension}, there exists a zero dimensional principal 
extension $g\colon (Y^\prime, S^\prime)\to (Y, S)$.
Consider the fiber product $X\times_Y Y^\prime$ and the following commutative diagram.
\[
\begin{tikzcd}
(X\times_Y Y^\prime,T\times S^\prime) \arrow[r,"\Pi_2"] \arrow[d, two heads,"\Pi_1"'] &  (Y^\prime,S^\prime)    \arrow[d,two heads, "g"] \\
 (X,T) \arrow[r,"\pi"'] & (Y,S)
\end{tikzcd}
\]
Here $\Pi_1$ and $\Pi_2$ are the natural projections.
By Lemma \ref{lemma: fiber product}, the map $\Pi_1$ is principal.
Applying Theorem \ref{theorem: zero dimensional principal extension} to $X\times_Y Y^\prime$, we obtain 
a principal extension $\phi\colon (X^\prime, T^\prime)\to (X\times_Y Y^\prime, T\times S^\prime)$.
We define $\pi^\prime = \Pi_2\circ \phi$ and $f=\Pi_1\circ \phi$.
\[
\begin{tikzcd}[column sep=large, row sep=large]
(X',T') 
  \arrow[rd, two heads, "\phi"] 
  \arrow[rdd, "f"'] 
  \arrow[rrd, "\pi'", bend left=8]
& & \\
& (X\times_Y Y',\, T\times S') 
  \arrow[d, two heads, "\Pi_1"] 
  \arrow[r, "\Pi_2"]
& (Y',S') 
  \arrow[d, two heads, "g"] \\
& (X,T) 
  \arrow[r, "\pi"']
& (Y,S)
\end{tikzcd}
\]
Since the composition of principal factor maps is also principal, the map $f$ is principal.
\end{proof}

\subsection{Proof of Theorem \ref{theorem: relativised variational principle}}
\label{subsection: proof of relativised variational principle}

Here we prove Theorem \ref{theorem: relativised variational principle}.
We repeat the statement for the convenience of readers.

\begin{theorem}[$=$ Theorem \ref{theorem: relativised variational principle}]
Let $0\leq w\leq 1$.
Let $(X, T)$, $(Y, S)$ and $(Z, R)$ be dynamical systems.
Let $\pi\colon X\to Y$, $\rho\colon X\to Z$ and $\theta\colon Y\to Z$ be equivariant continuous maps such that 
$\rho = \theta \circ \pi$ and that $\rho$ is surjective.
\[
\begin{tikzcd}
(X,T) \arrow[r,"\pi"] \arrow[dr, two heads, "\rho"'] & (Y,S) \arrow[d,"\theta"] \\
& (Z,R)
\end{tikzcd}
\]
For any $\nu\in \mathscr{M}^R(Z)$, we have
\begin{equation*} 
  \begin{split}
   & \int_Z \FHh^w\left(\rho^{-1}(z),\pi, T\right) \, d\nu(z) \\
   & = \sup\left\{w h_\mu(T|R) + (1-w) h_{\pi_*\mu}(S|R) \middle|\, 
    \mu \in \mathscr{M}^T(X)  \text{ with }\rho_*\mu = \nu\right\}.
   \end{split}
\end{equation*}    
\end{theorem}

\begin{proof}
By Proposition \ref{proposition: half of the variational principle}
\begin{align*}
  & \sup\left\{w h_\mu(T|R) + (1-w) h_{\pi_*\mu}(S|R) \middle|\, 
      \mu \in \mathscr{M}^T(X)  \text{ with }\rho_*\mu = \nu\right\} \\
     &\leq \int_Z \FHh^w\left(\rho^{-1}(z),\pi, T\right) \, d\nu(z).
\end{align*}     
We would like to prove the reverse inequality.
By Corollary \ref{cor: zero dimensional trick}, we can construct the following commutative diagram 
so that 
$f$ and $g$ are principal factor maps and that $X^\prime$ and $Y^\prime$ are zero dimensional.
\[
\begin{tikzcd}
(X^\prime,T^\prime) \arrow[r,"\pi^\prime"] \arrow[d, two heads,"f"'] &  (Y^\prime,S^\prime)    \arrow[d,two heads, "g"] \\
 (X,T) \arrow[r,"\pi"'] & (Y,S)
\end{tikzcd}
\]
We apply Lemma \ref{lemma: principal extension and weighted entropy sum} to the commutative diagram 
\[
\begin{tikzcd}
(X^\prime,T^\prime) \arrow[r,"\pi^\prime"] \arrow[d, two heads, "f"']
  & (Y^\prime,S^\prime) \arrow[d, two heads, "g"] \\
(X,T) \arrow[r,"\pi"'] \arrow[dr, "\rho"']
  & (Y,S) \arrow[d,"\theta"] \\
& (Z,R)
\end{tikzcd}
\]
and we obtain
\begin{align*}
    & \sup\left\{w h_\mu(T|R) + (1-w) h_{\pi_*\mu}(S|R) \middle|\, 
    \mu \in \mathscr{M}^T(X)  \text{ with }\rho_*\mu = \nu\right\} \\
    & = \sup\left\{w h_\mu(T^\prime|R) + (1-w) h_{(\pi^\prime)_*\mu}(S^\prime|R) \middle|\, 
    \mu \in \mathscr{M}^{T^\prime}(X^\prime)  \text{ with }(\rho\circ f)_*\mu = \nu\right\}.
\end{align*}
By Corollary \ref{corollary: factor map and weighted topological entropy}
\[ \int_Z \FHh^w\left(\rho^{-1}(z),\pi, T\right) \, d\nu(z) 
\leq \int_Z \FHh^w\left((\rho \circ f)^{-1}(z),\pi^\prime, T^\prime\right) \, d\nu(z).\]
Since $X^\prime$ and $Y^\prime$ are zero dimensional, 
we can apply Proposition \ref{proposition: construction of invariant measures without ergodicity}
to the diagram 
\[
\begin{tikzcd}
(X^\prime,T^\prime) \arrow[r,"\pi^\prime"] \arrow[dr, two heads, "\rho\circ f"'] & (Y^\prime,S^\prime) \arrow[d,"\theta\circ g"] \\
& (Z,R)
\end{tikzcd}
\]
and we obtain
\begin{align*}
  & \int_Z \FHh^w\left((\rho \circ f)^{-1}(z),\pi^\prime, T^\prime \right) \, d\nu(z) \\
  &\leq \sup\left\{w h_\mu(T^\prime|R) + (1-w) h_{(\pi^\prime)_*\mu}(S^\prime|R) \middle|\, 
    \mu \in \mathscr{M}^{T^\prime}(X^\prime)  \text{ with }(\rho\circ f)_*\mu = \nu\right\}.
\end{align*}
Combining these estimates, we conclude
\begin{align*}
 & \int_Z \FHh^w\left(\rho^{-1}(z),\pi, T\right) \, d\nu(z) \\
 &\leq \sup\left\{w h_\mu(T|R) + (1-w) h_{\pi_*\mu}(S|R) \middle|\, 
    \mu \in \mathscr{M}^T(X)  \text{ with }\rho_*\mu = \nu\right\}.
\end{align*}
\end{proof}

\subsection{Useful corollary} \label{subsection: useful corollary}

In this subsection we study a modification of Corollary \ref{corollary: equivalence of two approaches}.
The result here will be applied to the dimension theory of intersections of randomly translated Bedford--McMullen carpets.
Let $(X, T)$, $(Y, S)$ and $(Z, R)$ be dynamical systems, and let 
$\pi\colon X\to Y$ and $\rho\colon X\to Z$ be equivariant continuous maps.
We assume that $\rho$ is surjective.
\begin{equation*} 
\begin{tikzcd}
(X,T) \arrow[r,"\pi"] \arrow[dr, two heads, "\rho"']
  & (Y,S) \\
  & (Z,R)
\end{tikzcd}
\end{equation*}
Notice that here we do not consider a map from $Y$ to $Z$.
Define $\Pi\colon X\to Y\times Z$ and $\theta\colon Y\times Z\to Z$ by 
$\Pi(x) = (\pi(x), \rho(x))$ and $\theta(y, z) = z$.
Then we have the following commutative diagram.
\begin{equation}  \label{eq: commutative diagram}
\begin{tikzcd}
(X,T) \arrow[r,"\Pi"] \arrow[dr, two heads, "\rho"'] & (Y\times Z, S\times R) \arrow[d,"\theta"] \\
& (Z,R)
\end{tikzcd}
\end{equation}

We compare the weighted topological entropy of $\pi$ and $\Pi$:

\begin{lemma} \label{lemma: reduction to previous Corollary}
Let $0\leq w \leq 1$ and $z\in Z$.
\begin{enumerate}
 \item We have 
\[ \FHh^w\left(\rho^{-1}(z), \pi, T\right)= \FHh^w\left(\rho^{-1}(z), \Pi, T\right). \]
Here the left-hand side is the $w$-weighted topological entropy of the set $\rho^{-1}(z)$ with respect to the 
map $\pi\colon X\to Y$, whereas the right-hand side is the $w$-weighted topological entropy of $\rho^{-1}(z)$ 
with respect to the map $\Pi \colon X\to Y\times Z$.
 \item We have 
\[ \h^w\left(\rho^{-1}(z), \pi, T\right)= \h^w\left(\rho^{-1}(z), \Pi, T\right). \]
Here $\h^w\left(\rho^{-1}(z), \pi, T\right)$ and $\h^w\left(\rho^{-1}(z), \Pi, T\right)$ are the combinatorial version of 
the $w$-weighted topological entropy of $\rho^{-1}(z)$
(defined in \S \ref{subsection: relativised variational principle for weighted topological entropy}) 
with respect to $\pi$ and $\Pi$ respectively.
\end{enumerate}
\end{lemma}

\begin{proof}
(1) Take metrics $\mathbf{d}$, $\mathbf{d}^\prime$ and $\mathbf{d}^{\prime\prime}$ on $X$, $Y$ and $Z$ respectively.
We define a metric $D$ on $Y\times Z$ by 
$D\left((y_1, z_1), (y_2, z_2)\right)  = \max\left(\mathbf{d}^\prime(y_1, y_2), \mathbf{d}^{\prime\prime}(z_1, z_2)\right)$.
Now it is easy to check that for any $x\in \rho^{-1}(z)$, $n\geq 1$ and $\varepsilon >0$
\begin{equation} \label{eq: coincidence of weighted Bowen balls of pi and Pi}
   B^w_n(x, \pi, \varepsilon) \cap \rho^{-1}(z) = B^w_n(x,\Pi, \varepsilon) \cap \rho^{-1}(z). 
\end{equation}   
Here $B^w_n(x, \pi, \varepsilon)$ and $B^w_n(x,\Pi, \varepsilon)$ are the $w$-weighted Bowen balls with respect to $\pi$ and $\Pi$ 
respectively.
Notice that we have $B^w_n(x,\Pi, \varepsilon)\subset B^w_n(x, \pi, \varepsilon)$ for all $x\in X$ 
and they do not coincide in general.
However, their intersections with $\rho^{-1}(z)$ do coincide if $x\in \rho^{-1}(z)$.

Now we recall the observation in Remark \ref{remark: variation of the definition of FH entropy} that one 
obtains the same values of \(\FHh^w(\rho^{-1}(z),\pi, T)\) and $\FHh^w\left(\rho^{-1}(z), \Pi, T\right)$
even if the definition of the Feng--Huang entropy is modified by requiring all centers of weighted Bowen balls to belong to \(\rho^{-1}(z)\).
Then we conclude from \eqref{eq: coincidence of weighted Bowen balls of pi and Pi}
that $\FHh^w\left(\rho^{-1}(z), \pi, T\right)= \FHh^w\left(\rho^{-1}(z), \Pi, T\right)$.

(2) We have $\Pi\left(\rho^{-1}(z)\right) = \pi(\rho^{-1}(z)) \times \{z\}$.
Then, following the definitions given in \S \ref{subsection: relativised variational principle for weighted topological entropy}, it is 
straightforward to check that 
$\#^w\left(\rho^{-1}(z), \pi, N, \varepsilon\right) = \#^w\left(\rho^{-1}(z), \Pi, N, \varepsilon\right)$ and hence 
$\h^w\left(\rho^{-1}(z), \pi, T\right)= \h^w\left(\rho^{-1}(z), \Pi, T\right)$.
\end{proof}

The next corollary is the main result of this subsection.
This will be used in \S \ref{section: application to Bedford--McMullen carpets}.

\begin{corollary} \label{corollary: geometric and combinatorial entropies}
Let $(X, T)$, $(Y, S)$ and $(Z, R)$ be dynamical systems, and let 
$\pi\colon X\to Y$ and $\rho\colon X\to Z$ be equivariant continuous maps.
We assume that $\rho$ is surjective.
\begin{equation*} 
\begin{tikzcd}
(X,T) \arrow[r,"\pi"] \arrow[dr, two heads, "\rho"']
  & (Y,S) \\
  & (Z,R)
\end{tikzcd}
\end{equation*}
Let $0\leq w\leq 1$ and $\nu\in \mathscr{M}^R(Z)$. Then 
\[ \FHh^w\left(\rho^{-1}(z), \pi, T\right)  = \h^w\left(\rho^{-1}(z), \pi, T\right) \]
for $\nu$-almost every $z\in Z$.
\end{corollary}

\begin{proof}
Applying Corollary \ref{corollary: equivalence of two approaches} to the 
commutative diagram \eqref{eq: commutative diagram}, we obtain 
\[ \FHh^w\left(\rho^{-1}(z), \Pi, T\right)  = \h^w\left(\rho^{-1}(z), \Pi, T\right) \]
for $\nu$-almost every $z\in Z$.
Now the conclusion follows from Lemma \ref{lemma: reduction to previous Corollary}.
\end{proof}

\begin{remark} \label{remark: another relativised variational principle}
In the setting of Corollary \ref{corollary: geometric and combinatorial entropies}, 
applying the variational principle (Theorem \ref{theorem: relativised variational principle})
to the diagram \eqref{eq: commutative diagram}, we obtain
\begin{align*}
 & \int_Z \FHh^w\left(\rho^{-1}(z), \pi, T\right) d\nu(z) 
  = \int_Z \FHh^w\left(\rho^{-1}(z), \Pi, T\right) d\nu(z) \\
 & = \sup\left\{w h_\mu(T|R) + (1-w) h_{\Pi_*\mu}(S\times R|R) \middle| \mu\in \mathscr{M}^T(X) \text{ with } \rho_*\mu = \nu\right\},
\end{align*}
where $h_{\Pi_*\mu}(S\times R|R)$ denotes 
the conditional entropy of $\Pi_*\mu \in \mathscr{M}^{S\times R}(Y\times Z)$ with respect to 
the natural projection $\theta\colon Y\times Z\to Z$.
\end{remark}

\section{Application to Bedford--McMullen carpets} \label{section: application to Bedford--McMullen carpets}

In this section, we prove Theorem \ref{theorem: intersection of Bedford--McMullen carpets}.
The argument has three steps.
First, Corollary \ref{corollary: geometric and combinatorial entropies}
identifies the Feng--Huang weighted entropy and the combinatorial weighted entropy
on typical fibres.
Second, we compute the Feng--Huang entropy geometrically and relate it to
Hausdorff dimension.
Finally, we compute the combinatorial entropy by a digit-by-digit analysis,
which leads to products of the matrices $Q_{\tau,v}$.

Recall that $\mathbb{T} = \mathbb{R}/\mathbb{Z}$ and $\mathbb{T}^2 = \mathbb{R}^2/\mathbb{Z}^2$.
Let $a\geq b> 1$ be integers and set $A= \{0,1,2,\dots,a-1\}$ and 
$B=\{0,1,2,\dots,b-1\}$.
We also set $w = \log_a b$.
Let $D_1$ and $D_2$ be nonempty subsets of $A\times B$.
For $r=1,2$,
let $X_r$ be the Bedford--McMullen carpets defined by the digit set $D_r$:
\[ X_r = \left\{\left(\sum_{n=1}^\infty \frac{x_n}{a^n}, \sum_{n=1}^\infty\frac{y_n}{b^n}\right) \in \mathbb{T}^2\middle|\, 
    (x_n, y_n)\in D_r \text{ for all $n\geq1$}\right\}.  \]
Define $T_r\colon X_r\to X_r$ by $T_r(x, y) = (ax, by)$.
Let $\pi_r\colon X_r\to \mathbb{T}$ be the projection to the second coordinate ($\pi_r(x, y) = y$), and set 
$Y_r = \pi_r(X_r)$. We define $S_r\colon Y_r\to Y_r$ by $S_r y = by$.
We would like to calculate the Hausdorff dimension of the intersection $(X_1+\mathbf{t})\cap X_2$ for Lebesgue almost every
$\mathbf{t} \in \mathbb{T}^2$.

We define $X := X_1\times X_2$, $Y := Y_1\times Y_2$ and $Z := \mathbb{T}^2$
with maps $T\colon X\to X$, $S\colon Y\to Y$ and $R\colon Z\to Z$ defined by 
\[ T := T_1\times T_2, \quad S := S_1\times S_2, \quad R(x,y) = (ax, by). \] 
We also define $\pi\colon X\to Y$ and $\rho\colon X\to Z$ by 
\[ \pi = \pi_1\times \pi_2, \quad \rho\left(\mathbf{x}, \mathbf{x}^\prime\right) = \mathbf{x}^\prime-\mathbf{x} \quad 
    (\mathbf{x}\in X_1, \mathbf{x}^\prime\in X_2). \]
Now we have    
\begin{equation*} 
\begin{tikzcd}
(X,T) \arrow[r,"\pi"] \arrow[dr, "\rho"']
  & (Y,S) \\
  & (Z,R)
\end{tikzcd}.
\end{equation*}
It is sometimes convenient to extend the maps $S$ and $T$ to $\mathbb{T}^2$ and $\mathbb{T}^4$.
Namely we will occasionally consider the maps 
\[ \mathbb{T}^2\ni (y_1, y_2) \mapsto (b y_1, b y_2) \in \mathbb{T}^2, \quad 
    \mathbb{T}^4 \ni (x_1, y_1, x_2, y_2) \mapsto (ax_1, b y_1, a x_2, b y_2) \in \mathbb{T}^4 \]
and denote them also by $S$ and $T$ respectively.     
    
Let $\nu\in \mathscr{M}^R(Z)$ be the Lebesgue measure.
We apply Corollary \ref{corollary: geometric and combinatorial entropies} to this setting and obtain 
\begin{equation} \label{eq: geometric and combinatorial entropies of fibers}
 \FHh^w\left(\rho^{-1}(\mathbf{t}), \pi, T\right) = \h^w\left(\rho^{-1}(\mathbf{t}), \pi, T\right)  \quad 
    \text{for $\nu$-almost every $\mathbf{t}\in Z$}. 
\end{equation}    
Strictly speaking, Corollary \ref{corollary: geometric and combinatorial entropies}
assumes that $\rho$ is surjective.
If $\rho(X)\neq Z$, then $\rho(X)$ is a proper closed $R$-invariant subset of $Z$.
Since Lebesgue measure on $Z$ is ergodic and has full support, this implies
$\nu(\rho(X))=0$.
Thus $\rho^{-1}(\mathbf t)=\emptyset$ for $\nu$-a.e. $\mathbf t$, and the desired
identity \eqref{eq: geometric and combinatorial entropies of fibers} is trivial. 
(We use the convention\footnote{Recall our convention that the Hausdorff dimension of the empty set is $-\infty$.}
that the entropy of the empty set is $-\infty$.)
Hence we may assume, without loss of generality, that
$\rho$ is surjective.

The task of this section is to calculate both sides of \eqref{eq: geometric and combinatorial entropies of fibers}
in terms of geometric and combinatorial quantities
and deduce Theorem \ref{theorem: intersection of Bedford--McMullen carpets}.

\subsection{Calculating the Feng--Huang entropy} \label{subsection: calculating the Feng--Huang entropy}

First we calculate the Feng--Huang entropy $\FHh^w\left(\rho^{-1}(\mathbf{t}), \pi, T\right)$ in terms of Hausdorff dimension.
The proof is based on a direct comparison between standard balls and weighted Bowen balls.

We need to introduce metrics.
We define a metric $\mathbf{d}$ on the circle $\mathbb{T} = \mathbb{R}/\mathbb{Z}$ by 
$\mathbf{d}(x, y) = \min_{n\in \mathbb{Z}}|x-y-n|$.
We also define metrics $\mathbf{d}^\prime$ and $\mathbf{d}^{\prime\prime}$ 
on $\mathbb{T}^2$ and $\mathbb{T}^2\times \mathbb{T}^2$ respectively by 
\begin{align*}
   \mathbf{d}^\prime\left((x, y), (x^\prime, y^\prime)\right) & = \max\left(\mathbf{d}(x,x^\prime), \mathbf{d}(y, y^\prime)\right), \\
    \mathbf{d}^{\prime\prime}\left((\mathbf{x}, \mathbf{y}), (\mathbf{x}^\prime, \mathbf{y}^\prime)\right) & =
    \max\left(\mathbf{d}^\prime(\mathbf{x}, \mathbf{x}^\prime), \mathbf{d}^\prime(\mathbf{y}, \mathbf{y}^\prime)\right)
\end{align*}    
where $x, y, x^\prime, y^\prime\in \mathbb{T}$ and $\mathbf{x}, \mathbf{y}, \mathbf{x}^\prime, \mathbf{y}^\prime\in \mathbb{T}^2$.
We consider Hausdorff dimension of various subsets of $\mathbb{T}^2$ and $\mathbb{T}^2\times \mathbb{T}^2$ with respect to these metrics.

\begin{lemma} \label{lemma: Hausdorff dimension of intersection}
For every $\mathbf{t} \in \mathbb{T}^2$
\[ \dim_{\mathrm{H}} \left((X_1+\mathbf{t})\cap X_2\right) = \dim_{\mathrm{H}} \rho^{-1}(\mathbf{t}). \]
\end{lemma}

\begin{proof}
The map
\[
   (X_1+\mathbf{t})\cap X_2 \to \rho^{-1}(\mathbf{t}), \quad \mathbf{x} \mapsto (\mathbf{x}-\mathbf{t},\mathbf{x})
\]
is isometric. (The inverse is the projection $ \rho^{-1}(\mathbf{t}) \to  (X_1+\mathbf{t})\cap X_2$, 
$(\mathbf{x}, \mathbf{x}^\prime) \mapsto  \mathbf{x}^\prime$.)
Therefore $(X_1+\mathbf{t})\cap X_2$ and $\rho^{-1}(\mathbf{t})$ have the same Hausdorff dimension.
\end{proof}

Recall $w= \log_a b$.

\begin{proposition} \label{proposition: Feng--Huang entropy and dimension}
For any subset $\Omega\subset X_1\times X_2$
\[ \FHh^w\left(\Omega, \pi, T\right) = (\log b) \dim_{\mathrm{H}}\Omega. \]
\end{proposition}

\begin{proof}
For $r>0$ and $p\in X_1\times X_2$ we denote $\mathrm{B}(p, r) = \{q\in X_1\times X_2\mid \mathbf{d}^{\prime\prime}(p, q) < r\}$.
For $s\geq 0$ and $\varepsilon>0$ we define $\mathcal{H}^s_\varepsilon(\Omega)$ as the infimum of 
$\sum_j r_j^s$ over all at most countable covering $\Omega \subset \bigcup_j \mathrm{B}(p_j, r_j)$ with $0<r_j < \varepsilon$.
We set $\mathcal{H}^s(\Omega) = \lim_{\varepsilon\to 0} \mathcal{H}^s_\varepsilon(\Omega)$.
The Hausdorff dimension $\dim_{\mathrm H}\Omega$ is the critical value of $s$ at which
$\mathcal H^s(\Omega)$ changes from $\infty$ to $0$.

The point of the proof is to compare $\mathrm{B}(p, r)$ with the weighted Bowen ball $B^w_n(p,\pi, r)$.
For $0<r<\frac{1}{2a}$ and $p=\left((x_1, y_1), (x_2, y_2)\right) \in X_1\times X_2$, the $w$-weighted Bowen ball 
$B^w_n(p,\pi, r)$ is the set of $q = \left((x^\prime_1, y^\prime_1), (x^\prime_2, y^\prime_2)\right) \in X_1\times X_2$
satisfying $\mathbf{d}(x_i, x^\prime_i)<r a^{-\lceil wn \rceil +1}$ and $\mathbf{d}(y_i, y^\prime_i) < r b^{-n+1}$ for $i=1,2$.
Since $w=\log_a b$ we have 
\[ b^{-n} < a^{-\lceil wn\rceil +1} \leq a b^{-n}. \]
Hence 
\[  \mathrm{B}(p, rb^{-n}) \subset B^w_n(p,\pi, r) \subset \mathrm{B}(p, rab^{-n}). \]

Let $0<\varepsilon <\frac{1}{2a}$ and 
let $N$ be a natural number.
The quantity $\Lambda^{w, s\log b}_{N,\varepsilon}(\Omega)$ is defined as the infimum of
$\sum_{j} b^{-sn_j}$ over all at most countable covering 
$\Omega \subset \bigcup_j  B^w_{n_j}(p_j,\pi, \varepsilon)$ with $n_j\geq N$.
We have $B^w_{n_j}(p_j,\pi, \varepsilon)\subset \mathrm{B}(p_j, \varepsilon ab^{-n_j})$ and hence 
\[  \mathcal{H}^s_{\varepsilon a b^{-N+1}}(\Omega) \leq (\varepsilon a)^s \Lambda^{w, s\log b}_{N,\varepsilon}(\Omega). \]
Letting $N\to \infty$, we obtain 
\begin{equation} \label{eq: Hausdorff measure and Feng--Huang entropy}
   \mathcal{H}^s(\Omega) \leq (\varepsilon a)^s \Lambda^{w, s\log b}_{\varepsilon}(\Omega). 
\end{equation}   

On the other hand, suppose we are given at most countable covering 
$\Omega \subset \bigcup_j \mathrm{B}(p_j, r_j)$ with $r_j < \varepsilon b^{-N}$.
Let $n_j$ be the largest integer satisfying $\varepsilon b^{-n_j} > r_j$.
Then $n_j\geq N$ and 
\[ \mathrm{B}(p_j, r_j) \subset \mathrm{B}(p_j, \varepsilon b^{-n_j}) \subset B^w_{n_j}(p_j,\pi, \varepsilon). \]
It follows that 
\[ \Lambda^{w, s\log b}_{N,\varepsilon}(\Omega) \leq \sum_j b^{-sn_j}  \leq \left(\frac{b}{\varepsilon}\right)^s \sum_j r_j^s,  \quad 
     (\text{since $\varepsilon b^{-n_j-1} \leq r_j$}). \]
Therefore $\Lambda^{w, s\log b}_{N,\varepsilon}(\Omega) \leq (b/\varepsilon)^s \mathcal{H}^s_{\varepsilon b^{-N}}(\Omega)$.
Letting $N\to \infty$
\[ \Lambda^{w, s\log b}_{\varepsilon}(\Omega) \leq  \left(\frac{b}{\varepsilon}\right)^s \mathcal{H}^s(\Omega). \]
This inequality together with \eqref{eq: Hausdorff measure and Feng--Huang entropy} implies that for any $0<\varepsilon <\frac{1}{2a}$
\[  \FHh^w(\Omega, \pi, T, \varepsilon) = (\log b) \dim_{\mathrm{H}} \Omega. \]
Letting $\varepsilon \to 0$, we obtain the conclusion 
$\FHh^w\left(\Omega, \pi, T\right) = (\log b) \dim_{\mathrm{H}}\Omega$.
\end{proof}

Combining Lemma \ref{lemma: Hausdorff dimension of intersection} and 
Proposition \ref{proposition: Feng--Huang entropy and dimension}, we can calculate 
$\FHh^w\left(\rho^{-1}(\mathbf{t}), \pi, T\right)$ in terms of Hausdorff dimension.

\begin{corollary} \label{corollary: calculating the Feng--Huang entropy}
For all $\mathbf{t}\in \mathbb{T}^2$
\[ \FHh^w\left(\rho^{-1}(\mathbf{t}), \pi, T\right) = (\log b) \dim_{\mathrm{H}}\left((X_1+\mathbf{t})\cap X_2\right). \]
\end{corollary}

\subsection{Combinatorial preparations}  \label{subsection: combinatorial preparations}

Our final task is to compute $\h^w(\rho^{-1}(\mathbf{t}), \pi, T)$ in terms of matrix products.
This will be carried out in the next subsection.
The necessary combinatorial preparation is given in this subsection.

First we recall the definitions of the matrices $Q_{\tau, v}$ introduced in 
\S \ref{subsection: intersecting random translates of Bedford--McMullen carpets}.
We consider $4\times 4$-matrices indexed by $I = \{(0,0), (0,1), (1,0), (1,1)\}$.
Let $q\colon \mathbb{R}^2\to \mathbb{R}$ be the projection to the second coordinate.
For $\tau\in A\times B$ and $v\in B$, 
we define a $4\times 4$-matrix $Q_{\tau, v}$ by 
\[ Q_{\tau, v}\bigl((i,j), (k,\ell)\bigr) 
   = \left|\left(D_1+\tau+(i,j)\right)\cap \left(D_2+(ka,\ell b)\right)\cap q^{-1}(\ell b+ v)\right|, \]
where $(i,j), (k,\ell)\in I$.   
We also set
\[ Q_\tau  = \sum_{v\in B} Q_{\tau, v}. \]

For $r=1,2$ and $n\geq 1$, we define 
\[ X_r^{(n)} 
  = \left\{\left(\sum_{m=1}^n \frac{x_m}{a^m}, \sum_{m=1}^n \frac{y_m}{b^m}\right) \in \mathbb{R}^2
  \middle|\, (x_m, y_m) \in D_r \> (1\leq m \leq n)\right\}. \]
Notice that this is a subset of $\mathbb{R}^2$ (not the torus $\mathbb{T}^2$).

The following lemma is a key result that connects the intersection $(X_1+\mathbf{t})\cap X_2$
to the products of $Q_{\tau,v}$.
This may be viewed as a self-affine analogue of \cite[Lemma~4.2]{Kenyon--Peres_intersection}.

\begin{lemma} \label{lemma: discrete carpets}
For $\tau_m = (\alpha_m, \beta_m)\in A\times B$, $v_m\in B$ $(1\leq m \leq n)$ and $(i,j), (k,\ell)\in I$,
the cardinality of the set
\begin{equation} \label{eq: discrete intersection}
   \left(X_1^{(n)}+\left(\sum_{m=1}^n \frac{\alpha_m}{a^m}, \sum_{m=1}^n \frac{\beta_m}{b^m}\right) 
  + \left(\frac{i}{a^n}, \frac{j}{b^n}\right)\right) \cap \left(X_2^{(n)}+ (k,\ell)\right) 
  \cap q^{-1}\left(\ell+ \sum_{m=1}^n \frac{v_m}{b^m}\right)
\end{equation}
is equal to the $\left((i,j), (k,\ell)\right)$-entry of the matrix 
\[ Q_{\tau_n, v_n} Q_{\tau_{n-1}, v_{n-1}} \cdots Q_{\tau_1, v_1}. \] 
\end{lemma}

\begin{proof}
We proceed by induction on $n$.
The case $n=1$ follows from the definition of $Q_{\tau, v}$.
Indeed, the cardinality 
\begin{equation} \label{eq: discrete intersection case n=1}
 \left|\left(X_1^{(1)}+\left(\frac{\alpha_1}{a}, \frac{\beta_1}{b}\right) + \left(\frac{i}{a}, \frac{j}{b}\right)\right) 
\cap \left(X_2^{(1)}+ (k,\ell)\right) \cap q^{-1}\left(\ell+\frac{v_1}{b}\right)\right| 
\end{equation}
is equal to the number of $(x_1, y_1)\in D_1$ for which there exists $u\in A$ with $(u, v_1)\in D_2$ and  
\[ \left(\frac{x_1+\alpha_1+i}{a}, \frac{y_1+\beta_1+j}{b}\right) = \left(\frac{u}{a}+k, \frac{v_1}{b}+ \ell\right). \]
This equation is equivalent to 
$(x_1+\alpha_1+i, y_1+\beta_1+j) = (u+ak, v_1+\ell b)$.
Therefore \eqref{eq: discrete intersection case n=1} is equal to 
\[ \left|\left(D_1+(\alpha_1, \beta_1)+(i,j)\right)\cap (D_2+(ka, \ell b)) \cap q^{-1}(\ell b+v_1)\right|. \]
This coincides with the definition of $Q_{\tau_1, v_1}\left((i,j), (k,\ell)\right)$.

Now we suppose that $n>1$ and the statement holds for $n-1$.
The cardinality of \eqref{eq: discrete intersection} is equal to the number of $(x_m, y_m)\in D_1$ $(1\leq m \leq n)$
for which there exist $u_m\in A$ $(1\leq m \leq n)$ with 
$(u_m, v_m)\in D_2$ and 
\[ \left(\sum_{m=1}^n\frac{x_m}{a^m}+ \sum_{m=1}^n\frac{\alpha_m}{a^m} + \frac{i}{a^n},
    \sum_{m=1}^n \frac{y_m}{b^m} + \sum_{m=1}^n \frac{\beta_m}{b^m} + \frac{j}{b^n}\right) 
    = \left(\sum_{m=1}^n \frac{u_m}{a^m} + k, \sum_{m=1}^n \frac{v_m}{b^m} + \ell\right). \] 
Multiplying the first and second coordinates by $a^n$ and $b^n$ respectively, this equation is equivalent to 
\begin{equation} \label{eq: discrete intersection another form}
 \begin{split}
 & \left(\sum_{m=1}^n (x_m+\alpha_m) a^{n-m} + i, \sum_{m=1}^n (y_m + \beta_m)b^{n-m} + j\right) \\
  & = \left(\sum_{m=1}^n u_m a^{n-m} + k a^n, \sum_{m=1}^n v_m b^{n-m} + \ell b^n\right).
 \end{split} 
\end{equation}
Taking the two coordinates of \eqref{eq: discrete intersection another form} modulo $a$ and $b$, respectively,
we obtain that 
\begin{equation} \label{eq: discrete intersection top component}
 (x_n +\alpha_n + i, y_n+\beta_n+j) = (u_n + a k^\prime, v_n + b \ell^\prime) 
\end{equation}
for some $(k^\prime, \ell^\prime) \in I$.
The two equations \eqref{eq: discrete intersection another form} and \eqref{eq: discrete intersection top component}
imply 
\begin{equation} \label{eq: discrete intersection case n-1}
  \begin{split}
   & \left(\sum_{m=1}^{n-1}(x_m+\alpha_m)a^{n-1-m} + k^\prime, \sum_{m=1}^{n-1}(y_m+\beta_m)b^{n-1-m} + \ell^\prime\right) \\
   & = \left(\sum_{m=1}^{n-1}u_m a^{n-1-m} + k a^{n-1}, \sum_{m=1}^{n-1}v_m b^{n-1-m} + \ell b^{n-1}\right). 
  \end{split}
\end{equation}
On the other hand, the equation \eqref{eq: discrete intersection case n-1} together with 
\eqref{eq: discrete intersection top component} implies \eqref{eq: discrete intersection another form}.
Then it follows from the induction hypothesis that the cardinality of \eqref{eq: discrete intersection} is equal to 
\begin{align*}
  \sum_{(k^\prime, \ell^\prime) \in I} & \left|\left(D_1+\tau_n+(i,j)\right)\cap (D_2+(a k^\prime, b \ell^\prime))\cap q^{-1}(v_n+b\ell^\prime)\right| \\
   &\times \left(Q_{\tau_{n-1}, v_{n-1}} Q_{\tau_{n-2}, v_{n-2}} \cdots Q_{\tau_1, v_1}\right)\left((k^\prime, \ell^\prime), (k, \ell)\right),
\end{align*}
which is equal to $\left(Q_{\tau_n, v_n} Q_{\tau_{n-1}, v_{n-1}} \cdots Q_{\tau_1, v_1}\right)\left((i, j), (k, \ell)\right)$.
\end{proof}

Next we formulate how to pass from the \lq\lq{}approximate overlap\rq\rq{} \eqref{eq: discrete intersection}
to real intersection points of two carpets.
For $v_1, \dots, v_n\in B$ and $\mathbf{t} \in \mathbb{T}^2$, we define $\mathcal{N}(\mathbf{t}; v_1, \dots, v_n)$
as the number of $(u_1, \dots, u_n)\in A^n$ that satisfies the following two conditions.
\begin{itemize}
  \item $(u_m, v_m)\in D_2$ for all $1\leq m \leq n$.
  \item There exist $(u_m, v_m)\in D_2$ for $m\geq n+1$ satisfying
\begin{equation} \label{eq: defining equations of N(t,v_1,dots,v_n)}
 \left(\sum_{m=1}^\infty \frac{u_m}{a^m}, \sum_{m=1}^\infty \frac{v_m}{b^m}\right) \in X_1 + \mathbf{t}.
\end{equation} 
\end{itemize}
Notice that the condition \eqref{eq: defining equations of N(t,v_1,dots,v_n)} is equivalent to 
\[ \left(\sum_{m=1}^\infty \frac{u_m}{a^m}, \sum_{m=1}^\infty \frac{v_m}{b^m}\right) \in (X_1 + \mathbf{t})\cap X_2. \]

For $\mathbf{i}, \mathbf{k}\in I$, we define the $4\times 4$-matrix $E_{\mathbf{i} \mathbf{k}}$ by 
\[ E_{\mathbf{i} \mathbf{k}}(\mathbf{j}, \mathbf{l}) = \begin{cases} 1 & \left((\mathbf{j}, \mathbf{l}) = (\mathbf{i}, \mathbf{k})\right) \\
                                                                                                0 & \left((\mathbf{j}, \mathbf{l}) \neq (\mathbf{i}, \mathbf{k})\right)
                                                                                                \end{cases}. \] 

\begin{lemma} \label{lemma: approximate overlap and real intersection point}
Let $\tau_m = (\alpha_m, \beta_m) \in A\times B$ ($m\geq 1$) be given, and set 
\[ \mathbf{t} = \left(\sum_{m=1}^\infty \frac{\alpha_m}{a^m}, \sum_{m=1}^\infty \frac{\beta_m}{b^m}\right). \]
Let $n\geq 1$.
Let $\mathbf{i} = (i, j)$ and $\mathbf{k}= (k, \ell)$ be two elements of $I$ such that 
\begin{equation} \label{eq: extendability}
 Q_{\tau_{n+h}} Q_{\tau_{n+h-1}}\cdots Q_{\tau_{n+1}} E_{\mathbf{i} \mathbf{k}} \neq 0 \quad 
    \text{for any $h>0$}. 
\end{equation}   
Suppose that $(u_1, v_1), \dots, (u_n, v_n) \in D_2$ are given so that 
$\left(\sum_{m=1}^n \frac{u_m}{a^m}, \sum_{m=1}^n \frac{v_m}{b^m}\right) + (k, \ell)$ belongs to the set 
\begin{equation} \label{eq: discrete intersection revisited} 
  \left(X_1^{(n)}+\left(\sum_{m=1}^n \frac{\alpha_m}{a^m}, \sum_{m=1}^n \frac{\beta_m}{b^m}\right) 
  + \left(\frac{i}{a^n}, \frac{j}{b^n}\right)\right) \cap \left(X_2^{(n)}+ (k,\ell)\right) 
  \cap q^{-1}\left(\ell+ \sum_{m=1}^n \frac{v_m}{b^m}\right). 
\end{equation}  
Then there exist $(u_m, v_m) \in D_2$ $(m\geq n+1)$ for which we have
\[  \left(\sum_{m=1}^\infty \frac{u_m}{a^m}, \sum_{m=1}^\infty \frac{v_m}{b^m}\right) \in (X_1+\mathbf{t})\cap X_2. \]
In particular, under the condition \eqref{eq: extendability}, we have 
\[ \left(Q_{\tau_n, v_n} Q_{\tau_{n-1}, v_{n-1}} \cdots Q_{\tau_1, v_1}\right)(\mathbf{i}, \mathbf{k}) \leq \mathcal{N}(\mathbf{t}; v_1, \dots, v_n) \]
for any $v_1, \dots, v_n\in B$.
\end{lemma}

Roughly speaking, the lemma claims that the condition \eqref{eq: extendability} guarantees that any point of 
\eqref{eq: discrete intersection revisited} can be \lq\lq{}extended\rq\rq{} to a point of $(X_1+\mathbf{t})\cap X_2$.

\begin{proof}
By the condition \eqref{eq: extendability} and the diagonal argument, we can find sequences
$\mathbf{i}_1, \mathbf{i}_2, \mathbf{i}_3, \dots$ in $I$ and $v_{n+1}, v_{n+2}, v_{n+3}, \dots$ in $B$ such that 
$Q_{\tau_{n+h}, v_{n+h}}(\mathbf{i}_h, \mathbf{i}_{h-1}) >0$ for any $h>0$, where
we have set $\mathbf{i}_0 = (i, j)$.

Let $\mathbf{i}_h = (i_h, j_h)$.
By the definition of the matrix $Q_{\tau_{n+h}, v_{n+h}}$, we can find $(x_{n+h}, y_{n+h})\in D_1$ and 
$u_{n+h}\in A$ such that $(u_{n+h}, v_{n+h}) \in D_2$ and 
\[ (x_{n+h}+\alpha_{n+h}+i_h, y_{n+h}+\beta_{n+h}+j_h) = 
    (u_{n+h}+ a i_{h-1}, v_{n+h}+ b j_{h-1}). \]
Then, as in the proof of Lemma \ref{lemma: discrete carpets}
(namely, by the argument that the equations \eqref{eq: discrete intersection top component} and 
\eqref{eq: discrete intersection case n-1} imply \eqref{eq: discrete intersection another form}), we can show that 
the point $\left(\sum_{m=1}^{n+h} \frac{u_m}{a^m}, \sum_{m=1}^{n+h} \frac{v_m}{b^m}\right)$ belongs to the set
\[  \left(X_1^{(n+h)}+\left(\sum_{m=1}^{n+h} \frac{\alpha_m}{a^m}, \sum_{m=1}^{n+h} \frac{\beta_m}{b^m}\right) 
  + \left(\frac{i_h}{a^{n+h}}, \frac{j_h}{b^{n+h}}\right)\right) \cap \left(X_2^{(n+h)}+ (k,\ell)\right) 
  \cap q^{-1}\left(\ell+ \sum_{m=1}^{n+h} \frac{v_m}{b^m}\right). \]
This holds for any $h>0$.
Letting $h\to \infty$, we obtain
\[  \left(\sum_{m=1}^\infty \frac{u_m}{a^m}, \sum_{m=1}^\infty \frac{v_m}{b^m}\right) \in (X_1+\mathbf{t})\cap X_2. \]

Now we have seen that any point of \eqref{eq: discrete intersection revisited} can be extended to a point of 
$(X_1+\mathbf{t})\cap X_2$.
Therefore the cardinality of the set \eqref{eq: discrete intersection revisited} (which is equal to 
$\left(Q_{\tau_n, v_n} \cdots Q_{\tau_1, v_1}\right)(\mathbf{i}, \mathbf{k})$ by Lemma \ref{lemma: discrete carpets})
is bounded from above by $\mathcal{N}(\mathbf{t}; v_1, \dots, v_n)$.
\end{proof}

We need a convenient condition that guarantees \eqref{eq: extendability}.
The next definition provides one.
Let $M$ be any nonnegative $4\times 4$-matrix.
(Here \lq\lq{}nonnegative\rq\rq{} means that every entry of $M$ is nonnegative.)
We say that $M$ is \textbf{persistent} if $Q_{\tau_n} Q_{\tau_{n-1}}\cdots Q_{\tau_1} M \neq 0$
for any choices of $n\geq 1$ and $\tau_1, \dots, \tau_n\in A\times B$.
$M$ is said to be non-persistent if $Q_{\tau_n} Q_{\tau_{n-1}}\cdots Q_{\tau_1} M =0$
for some $n\geq 1$ and $\tau_1, \dots, \tau_n\in A\times B$.
If $M$ is persistent then $Q_\tau M$ is also persistent for any $\tau\in A\times B$.
The role of persistence is to record which carry states can be extended
indefinitely, independently of the future translation digits.
It enables us to use Lemma \ref{lemma: approximate overlap and real intersection point} and compare 
the finite matrix product with the actual combinatorial covering number of the fiber.

For a nonnegative $4\times 4$-matrix $M$, we define a $4\times 4$-matrix $\iota(M)$
by 
\[ \iota(M)(\mathbf{i}, \mathbf{k}) = \begin{cases}  1 & (M(\mathbf{i}, \mathbf{k})  >0) \\
                                                                      0 & (M(\mathbf{i}, \mathbf{k}) = 0) 
                                                                      \end{cases}, \]
for $\mathbf{i}, \mathbf{k}\in I$.
$\iota(M)$ is persistent if and only if $M$ is persistent.
For two nonnegative matrices $M_1$ and $M_2$, we have $\iota(M_1 M_2) = \iota\left(\iota(M_1) \iota(M_2)\right)$.
Notice that there exist only $2^{16}$ possibilities for the value of $\iota(M)$.

\begin{lemma} \label{lemma: extinction in definite time}
A nonnegative $4\times 4$-matrix $M$ is non-persistent if and only if 
there exist $n < 2^{16}$ and $\tau_1, \dots, \tau_n \in A\times B$ that satisfy 
$Q_{\tau_n} Q_{\tau_{n-1}}\cdots Q_{\tau_1} M =0$.
\end{lemma}

\begin{proof}
The \lq\lq{}if\rq\rq{} part is trivial.
We prove the \lq\lq{}only if\rq\rq{} part.
Suppose $M$ is non-persistent.
We choose the minimum $n\geq 1$ for which there exist $\tau_1, \dots, \tau_n\in A\times B$ that satisfy 
$Q_{\tau_n} \cdots Q_{\tau_1} M =0$.
We would like to show $n<2^{16}$.
Suppose the contrary (that is, $n\geq 2^{16}$).
Consider the following $n+1$ matrices 
\[ \iota(M), \> \iota(Q_{\tau_1}M), \> \iota(Q_{\tau_2}Q_{\tau_1}M), \dots, \iota(Q_{\tau_n}\cdots Q_{\tau_1}M). \]
By the pigeon-hole principle, there exist $0\leq \ell < m \leq n$ for which we have
$\iota\left(Q_{\tau_{\ell}}\cdots Q_{\tau_1}M\right) = \iota\left(Q_{\tau_m}\cdots Q_{\tau_1}M\right)$.
Then 
\begin{align*}
  \iota\left(Q_{\tau_n}\cdots Q_{\tau_{m+1}}\cdot Q_{\tau_{\ell}}\cdots Q_{\tau_1}M\right) & =
  \iota\left(\iota\left(Q_{\tau_n}\cdots Q_{\tau_{m+1}}\right) \cdot \iota\left(Q_{\tau_{\ell}}\cdots Q_{\tau_1}M\right)\right) \\
   & = \iota\left(\iota\left(Q_{\tau_n}\cdots Q_{\tau_{m+1}}\right) \cdot \iota\left(Q_{\tau_m}\cdots Q_{\tau_1}M\right)\right) \\
    & = \iota\left(Q_{\tau_n}\cdots Q_{\tau_1}M\right) = 0, 
\end{align*}    
which implies $Q_{\tau_n}\cdots Q_{\tau_{m+1}}\cdot Q_{\tau_{\ell}}\cdots Q_{\tau_1}M=0$.
However this contradicts the minimality of $n$.
\end{proof}

Now we assume that $\tau_1, \tau_2, \tau_3, \dots$ are independent random 
variables such that each $\tau_n$ is uniformly distributed over $A\times B$.
For notational convenience, we set $c=2^{16}-1$.
From Lemma \ref{lemma: extinction in definite time} we can find $p>0$ such that, if $M$ is a non-persistent nonnegative 
$4\times 4$-matrix, then 
\[ \mathbb{P}\left(Q_{\tau_{c}}Q_{\tau_{c-1}}\cdots Q_{\tau_1}M = 0\right) \geq p. \]

\begin{lemma} \label{lemma: extinction or not}
The following statement holds true almost surely.
For any $\varepsilon>0$, there exists a natural number $n_0$ such that, for any $n\geq n_0$ and any nonnegative $4\times 4$-matrix $M$,
the product $Q_{\tau_{n+\lfloor n\varepsilon\rfloor}} Q_{\tau_{n+\lfloor n\varepsilon\rfloor -1}}\cdots Q_{\tau_{n+2}} Q_{\tau_{n+1}}M$ 
is either zero or persistent.
\end{lemma}

\begin{proof}
First we notice that if the statement holds for some $\varepsilon'<\varepsilon$,
then it also holds for $\varepsilon$, because multiplying further on the left
keeps a zero matrix zero and sends a persistent matrix to a persistent matrix.
So it is enough to consider only countably many $\varepsilon = 1, 1/2, 1/3, 1/4,\dots$.
Hence we fix one such $\varepsilon$.
We can replace $M$ with $\iota(M)$, which has only $2^{16}$ possibilities.
So we also fix one such $M$, and it is enough to prove that the event
\begin{equation} \label{eq: probability of zero or persistent}
  \{\text{$Q_{\tau_{n+\lfloor n\varepsilon\rfloor}} Q_{\tau_{n+\lfloor n\varepsilon\rfloor -1}}\cdots Q_{\tau_{n+1}}M$
is either zero or persistent for sufficiently large $n$}\} 
\end{equation}
has probability one.
We define an event $\Omega_n$ by 
\[ \Omega_n = \{\text{$Q_{\tau_{n+\lfloor n\varepsilon\rfloor}} Q_{\tau_{n+\lfloor n\varepsilon\rfloor -1}}\cdots Q_{\tau_{n+1}}M$
  is nonzero and non-persistent}\}. \]
We have 
\[ \mathbb{P}(\Omega_n) \leq (1-p)^{\lfloor \lfloor n\varepsilon\rfloor/c \rfloor}. \]
Hence $\sum_{n=1}^\infty \mathbb{P}(\Omega_n) <\infty$.
By the Borel--Cantelli lemma,
\[ \mathbb{P}\left(\bigcap_{k=1}^\infty \bigcup_{n\geq k} \Omega_n\right)  = 0. \]
Namely 
\[ \mathbb{P}\left(\bigcup_{k=1}^\infty \bigcap_{n\geq k} \Omega_n^c\right)  = 1. \]
This shows that the event \eqref{eq: probability of zero or persistent} has probability one.
\end{proof}

Let $\tau_n = (\alpha_n, \beta_n)$.
We set 
\[ \mathbf{t} = \left(\sum_{n=1}^\infty \frac{\alpha_n}{a^n}, \sum_{n=1}^\infty \frac{\beta_n}{b^n}\right) \in \mathbb{T}^2. \]
Since $\tau_n$ $(n\geq 1)$ are assumed to be independent and uniformly distributed over $A\times B$, the distribution of 
$\mathbf{t}$ is given by the Lebesgue measure.

For a $4\times 4$-matrix $M = \left(M(\mathbf{i}, \mathbf{k})\right)_{\mathbf{i}, \mathbf{k} \in I}$ we denote 
\[ \norm{M} = \sum_{\mathbf{i}, \mathbf{k} \in I} |M(\mathbf{i}, \mathbf{k})|. \]
(Any norm will essentially do the same work, but this one is the most convenient.)

\begin{lemma} \label{lemma: counting intersection and matrix product}
Almost surely, we have
\[ \mathcal{N}(\mathbf{t}; v_1, \dots, v_n) \leq \norm{Q_{\tau_n, v_n} Q_{\tau_{n-1}, v_{n-1}}\cdots Q_{\tau_1, v_1}} \]
for any $n\geq 1$ and any $v_1, \dots, v_n\in B$.
\end{lemma}

\begin{proof}
This is basically a consequence of Lemma \ref{lemma: discrete carpets}.
Discarding a measure zero event, we can assume that for every $n\geq 1$
\begin{equation} \label{eq: genericity condition of t}
   0< \sum_{m=n}^\infty \frac{\alpha_m}{a^m} < \frac{1}{a^{n-1}}, \quad
    0 < \sum_{m=n}^\infty \frac{\beta_m}{b^m} < \frac{1}{b^{n-1}}. 
\end{equation}   
Here we have considered the sums $\sum_{m=n}^\infty \frac{\alpha_m}{a^m}$, $\sum_{m=n}^\infty \frac{\beta_m}{b^m}$
as points in $\mathbb{R}$ (not $\mathbb{T}$).

Suppose that $(x_m, y_m)\in D_1$ and $(u_m, v_m)\in D_2$ $(m\geq 1)$ are given and satisfy 
\[ \left(\sum_{m=1}^\infty \frac{x_m}{a^m}, \sum_{m=1}^\infty \frac{y_m}{b^m}\right)
           + \left(\sum_{m=1}^\infty \frac{\alpha_m}{a^m}, \sum_{m=1}^\infty \frac{\beta_m}{b^m}\right) = 
           \left(\sum_{m=1}^\infty \frac{u_m}{a^m},
           \sum_{m=1}^\infty \frac{v_m}{b^m}\right) \quad \text{in $\mathbb{T}^2$}. \]
By \eqref{eq: genericity condition of t}
\[ 0< \sum_{m=1}^\infty \frac{x_m}{a^m} + \sum_{m=1}^\infty \frac{\alpha_m}{a^m} < 2, \quad 
    0< \sum_{m=1}^\infty \frac{y_m}{b^m} +  \sum_{m=1}^\infty \frac{\beta_m}{b^m} < 2 \quad \text{in $\mathbb{R}$}. \]
Hence there exists $(k, \ell) \in I = \{(0,0), (0,1), (1,0), (1,1)\}$ satisfying
\[  \left(\sum_{m=1}^\infty \frac{x_m}{a^m} + \sum_{m=1}^\infty \frac{\alpha_m}{a^m},
           \sum_{m=1}^\infty \frac{y_m}{b^m} +  \sum_{m=1}^\infty \frac{\beta_m}{b^m}\right)
     =   \left(\sum_{m=1}^\infty \frac{u_m}{a^m}, \sum_{m=1}^\infty \frac{v_m}{b^m}\right) + (k, \ell) \quad 
     \text{in $\mathbb{R}$}. \]
Let $n \geq 1$. We have
\begin{align*}
     & \left(\sum_{m=1}^n \frac{x_m}{a^m} + \sum_{m=1}^n \frac{\alpha_m}{a^m},
           \sum_{m=1}^n \frac{y_m}{b^m} +  \sum_{m=1}^n \frac{\beta_m}{b^m}\right)  \\ 
       &+ \left(\sum_{m=n+1}^\infty \frac{x_m}{a^m} + \sum_{m=n+1}^\infty \frac{\alpha_m}{a^m},
           \sum_{m=n+1}^\infty \frac{y_m}{b^m} +  \sum_{m=n+1}^\infty \frac{\beta_m}{b^m}\right)    \\
   &  = \left(\sum_{m=1}^n \frac{u_m}{a^m}, \sum_{m=1}^n \frac{v_m}{b^m}\right) + 
        \left(\sum_{m=n+1}^\infty \frac{u_m}{a^m}, \sum_{m=n+1}^\infty \frac{v_m}{b^m}\right) + (k, \ell).
\end{align*}
We have $0\leq \sum_{m=n+1}^\infty \frac{u_m}{a^m} \leq \frac{1}{a^n}$ and 
$0\leq \sum_{m=n+1}^\infty \frac{v_m}{b^m} \leq \frac{1}{b^n}$.
By \eqref{eq: genericity condition of t}
\[ 0< \sum_{m=n+1}^\infty \frac{x_m}{a^m} + \sum_{m=n+1}^\infty \frac{\alpha_m}{a^m} < \frac{2}{a^n}, \quad 
    0<  \sum_{m=n+1}^\infty \frac{y_m}{b^m} +  \sum_{m=n+1}^\infty \frac{\beta_m}{b^m} < \frac{2}{b^n}. \]
Then there exists $(i, j)\in I$ such that 
\[ \left(\sum_{m=1}^n \frac{x_m}{a^m} + \sum_{m=1}^n \frac{\alpha_m}{a^m},
           \sum_{m=1}^n \frac{y_m}{b^m} +  \sum_{m=1}^n \frac{\beta_m}{b^m}\right) 
    + \left(\frac{i}{a^n}, \frac{j}{b^n}\right) 
    = \left(\sum_{m=1}^n \frac{u_m}{a^m}, \sum_{m=1}^n \frac{v_m}{b^m}\right) + (k,\ell). \]
This means that the point 
\[ \left(\sum_{m=1}^n \frac{u_m}{a^m}, \sum_{m=1}^n \frac{v_m}{b^m}\right) + (k,\ell) \]
belongs to the set 
\[ \left(X_1^{(n)}+\left(\sum_{m=1}^n \frac{\alpha_m}{a^m}, \sum_{m=1}^n \frac{\beta_m}{b^m}\right) 
  + \left(\frac{i}{a^n}, \frac{j}{b^n}\right)\right) \cap \left(X_2^{(n)}+ (k,\ell)\right) 
  \cap q^{-1}\left(\ell+ \sum_{m=1}^n \frac{v_m}{b^m}\right) \]
in Lemma \ref{lemma: discrete carpets}.
This implies that 
\[ \mathcal{N}(\mathbf{t}; v_1, \dots, v_n) \leq \norm{Q_{\tau_n, v_n} Q_{\tau_{n-1}, v_{n-1}}\cdots Q_{\tau_1, v_1}}. \]
\end{proof}

\subsection{Proof of Theorem \ref{theorem: intersection of Bedford--McMullen carpets}}
\label{subsection: proof of Theorem_intersection of Bedford--McMullen carpets}

In this subsection, we calculate the entropy $\h^w\left(\rho^{-1}(\mathbf{t}), \pi, T\right)$ and 
complete the proof of Theorem \ref{theorem: intersection of Bedford--McMullen carpets}.

Recall the definition of $\mathcal{N}(\mathbf{t}; v_1, \dots, v_n)$:
For $v_1, \dots, v_n\in B$ and $\mathbf{t} \in \mathbb{T}^2$, we defined $\mathcal{N}(\mathbf{t}; v_1, \dots, v_n)$
as the number of $(u_1, \dots, u_n)\in A^n$ such that $(u_m, v_m)\in D_2$ for all $1\leq m \leq n$ and that
there exist $(u_m, v_m)\in D_2$ for $m\geq n+1$ satisfying
$\left(\sum_{m=1}^\infty \frac{u_m}{a^m}, \sum_{m=1}^\infty \frac{v_m}{b^m}\right) \in X_1 + \mathbf{t}$.

\begin{lemma} \label{lemma: expression of combinatorial entropy}
For every $\mathbf{t}\in \mathbb{T}^2$
\[ \h^w\left(\rho^{-1}(\mathbf{t}), \pi, T\right) 
  = \limsup_{n\to \infty} \frac{1}{n} \log \left(\sum_{(v_1, \dots, v_n)\in B^n} \mathcal{N}(\mathbf{t}; v_1, \dots, v_n)^w\right).
  \]
\end{lemma}

\begin{proof}
Recall that $\pi\colon X\to Y$ and $\rho\colon X\to \mathbb{T}^2$ were defined by 
\[ \pi(x_1, y_1, x_2, y_2) = (y_1, y_2), \quad 
    \rho(x_1, y_1, x_2, y_2) = (x_2-x_1, y_2-y_1) \]
for $(x_1, y_1)\in X_1$ and $(x_2, y_2)\in X_2$.
We use metrics $\mathbf{d}$, $\mathbf{d}^\prime$ and $\mathbf{d}^{\prime\prime}$ on $\mathbb{T}$, $\mathbb{T}^2$ and 
$\mathbb{T}^2\times \mathbb{T}^2$ respectively introduced in \S \ref{subsection: calculating the Feng--Huang entropy}.
For $n\geq 1$, we denote $\mathbf{d}^\prime_n = (\mathbf{d}^\prime)^S_n$ and 
$\mathbf{d}^{\prime\prime}_n = (\mathbf{d}^{\prime\prime})^{T}_n$.
We define a subset $B_2 \subset B$ as the set of $v\in B$ for which there exists $u\in A$ with $(u, v)\in D_2$.

We fix $\mathbf{t} = (s, t)\in \mathbb{T}^2$.
First notice that if $\rho^{-1}(\mathbf{t})$ is empty then $\mathcal{N}(\mathbf{t}; v_1, \dots, v_n)$ is zero 
for all $(v_1, \dots, v_n)$ and the 
above identity trivially holds. (Both sides are $-\infty$.) So we assume that $\rho^{-1}(\mathbf{t})$ is not empty.

For $\mathbf{v} = (v_1, \dots, v_n)\in B^n$, we introduce 
$\mathbf{v}^{\pm} = (v^{\pm}_1, \dots, v^{\pm}_n)\in B^n$ by 
\[ \sum_{m=1}^n \frac{v_m}{b^m} + \frac{1}{b^n} = \sum_{m=1}^n \frac{v_m^+}{b^m}, \quad 
    \sum_{m=1}^n \frac{v_m}{b^m} - \frac{1}{b^n} = \sum_{m=1}^n \frac{v_m^-}{b^m}, \quad \text{in $\mathbb{T}$}.\]
We define a closed subset $\mathcal{B}(\mathbf{v}) \subset \mathbb{T}^2$ by
\[ \mathcal{B}(\mathbf{v}) = \left\{\left(-t+\sum_{m=1}^\infty \frac{v_m}{b^m},
     \sum_{m=1}^\infty \frac{v_m}{b^m}  \right)\in \mathbb{T}^2 \middle|\, 
     v_m\in B_2 \text{ for all $m\geq n+1$}\right\}. \]
We have $\pi\left(\rho^{-1}(\mathbf{t})\right)\subset \bigcup_{\mathbf{v}\in B^n} \mathcal{B}(\mathbf{v})$.
We define $\mathcal{A}(\mathbf{v})$ as the set of $\mathbf{u} = (u_1, \dots, u_n)\in A^n$ such that
$(u_m, v_m)\in D_2$ for all $1\leq m \leq n$ and that
there exist $(u_m, v_m)\in D_2$ for $m\geq n+1$ satisfying
$\left(\sum_{m=1}^\infty \frac{u_m}{a^m}, \sum_{m=1}^\infty \frac{v_m}{b^m}\right) \in X_1 + \mathbf{t}$.
We have $\mathcal{N}(\mathbf{t}; \mathbf{v}) = |\mathcal{A}(\mathbf{v})|$.

For $\mathbf{u}\in \mathcal{A}(\mathbf{v})$ we define a closed subset 
$\mathcal{X}(\mathbf{u}, \mathbf{v})$ of $\mathbb{T}^2\times \mathbb{T}^2$ by
\begin{align*}
   & \mathcal{X}(\mathbf{u}, \mathbf{v})= \\
  & \left\{\left(-s+\sum_{m=1}^\infty \frac{u_m}{a^m}, -t+\sum_{m=1}^\infty \frac{v_m}{b^m},
   \sum_{m=1}^\infty \frac{u_m}{a^m}, \sum_{m=1}^\infty \frac{v_m}{b^m}\right) \middle|\, 
  (u_m, v_m)\in D_2 \text{ for all $m\geq n+1$}\right\}.
\end{align*}
Then 
\[ 
   \pi^{-1}\left(\mathcal{B}(\mathbf{v})\right) \cap \rho^{-1}(\mathbf{t}) \subset 
    \bigcup_{\mathbf{u}\in \mathcal{A}(\mathbf{v})} \mathcal{X}(\mathbf{u}, \mathbf{v}) 
    \cup \bigcup_{\mathbf{u}\in \mathcal{A}(\mathbf{v}^-)} \mathcal{X}(\mathbf{u}, \mathbf{v}^-) 
    \cup \bigcup_{\mathbf{u}\in \mathcal{A}(\mathbf{v}^+)} \mathcal{X}(\mathbf{u}, \mathbf{v}^+).
 \]   
 
For every $N\geq 1$ and $\varepsilon>0$, if $n>N+\log_{b}(1/\varepsilon)$, then the diameters of 
$\mathcal{B}(\mathbf{v})$ and $\mathcal{X}(\mathbf{u}, \mathbf{v})$ with respect to $\mathbf{d}^\prime_N$ and 
$\mathbf{d}^{\prime\prime}_N$ are 
smaller than $\varepsilon$.
Then (noting Lemma \ref{lemma: weighted covering number})
\begin{align*}
   \h^w\left(\rho^{-1}(\mathbf{t}), \pi, T\right) & \leq \limsup_{n\to \infty} 
   \frac{1}{n}\log \sum_{\mathbf{v}\in B^n}
   \left(\mathcal{N}(\mathbf{t}; \mathbf{v}) + \mathcal{N}(\mathbf{t}; \mathbf{v}^+) + \mathcal{N}(\mathbf{t}; \mathbf{v}^-)\right)^w \\
  & \leq 
  \limsup_{n\to \infty} \frac{1}{n}\log \sum_{\mathbf{v}\in B^n}
   \left(\mathcal{N}(\mathbf{t}; \mathbf{v})^w 
   + \mathcal{N}(\mathbf{t}; \mathbf{v}^+)^w + \mathcal{N}(\mathbf{t}; \mathbf{v}^-)^w\right) \\
   & = \limsup_{n\to \infty} \frac{1}{n}\log \sum_{\mathbf{v}\in B^n}
   \mathcal{N}(\mathbf{t}; \mathbf{v})^w.
\end{align*}
In the second inequality, we have used $0 \leq w \leq 1$.
Next we consider the reverse inequality.

For each $n\geq 1$, by a simple greedy algorithm, we can choose a subset $\Omega_n\subset B^n$
for which the following conditions hold.
  \begin{itemize}
   \item We have $\mathcal{B}(\mathbf{v})\cap \pi\left(\rho^{-1}(\mathbf{t})\right) \neq \emptyset$ for every $\mathbf{v}\in \Omega_n$.
   \item $\mathcal{B}(\mathbf{v}) \cap \mathcal{B}(\mathbf{v}^\prime) = \emptyset$ for any distinct $\mathbf{v}, \mathbf{v}^\prime\in \Omega_n$.  
            (This implies $\mathbf{d}^\prime(\mathbf{x}, \mathbf{x}^\prime) \geq 1/b^n$ for any $\mathbf{x}\in \mathcal{B}(\mathbf{v})$
            and $\mathbf{x}^\prime\in \mathcal{B}(\mathbf{v}^\prime)$.)
   \item $\sum_{\mathbf{v}\in \Omega_n} \mathcal{N}(\mathbf{t}; \mathbf{v})^w \geq 
             \frac{1}{3}\sum_{\mathbf{v}\in B^n}  \mathcal{N}(\mathbf{t}; \mathbf{v})^w$.
  \end{itemize}
(Here the \lq\lq{}greedy algorithm\rq\rq{} is as follows. We first choose $\mathbf{v}_1\in B^n$ which attains the maximum 
$\mathcal{N}(\mathbf{t}; \mathbf{v})$, and subtract $\{\mathbf{v}_1, \mathbf{v}^-_1, \mathbf{v}^+_1\}$ from $B^n$.
Next we choose $\mathbf{v}_2 \in B^n \setminus \{\mathbf{v}_1, \mathbf{v}^-_1, \mathbf{v}^+_1\}$ which attains the maximum 
$\mathcal{N}(\mathbf{t}; \mathbf{v})$ among the remaining ones, and subtract 
$\{\mathbf{v}_2, \mathbf{v}^-_2, \mathbf{v}^+_2\}$ from $B^n \setminus \{\mathbf{v}_1, \mathbf{v}^-_1, \mathbf{v}^+_1\}$.
We continue this process until it terminates, and we set $\Omega_n := \{\mathbf{v}_1, \mathbf{v}_2, \dots\}$.
The factor \(1/3\) comes from the fact that each \(b\)-adic interval can meet
only its two neighbors.)

For every $\mathbf{v}\in \Omega_n$, we can take a finite subset 
$\Lambda(\mathbf{v})\subset \pi^{-1}(\mathcal{B}(\mathbf{v}))\cap \rho^{-1}(\mathbf{t})$
with $\left|\Lambda(\mathbf{v})\right| \geq (1/3) \mathcal{N}(\mathbf{t}; \mathbf{v})$ 
such that $\mathbf{d}^{\prime\prime}_n(p, p^\prime) \geq 1/a$ for any distinct $p, p^\prime\in \Lambda(\mathbf{v})$
(that is, $\Lambda(\mathbf{v})$ is a $(1/a)$-separated set with respect to $\mathbf{d}^{\prime\prime}_n$).

Now, let $0<\varepsilon<1/a$ and suppose we are given an open cover 
$\pi\left(\rho^{-1}(\mathbf{t})\right) \subset V_1\cup V_2\cup \dots \cup V_K$ with 
$\diam(V_j, \mathbf{d}^\prime_n) < \varepsilon$ for all $1\leq j \leq K$.
For each $\mathbf{v}\in \Omega_n$, let $J_\mathbf{v}$ be the set of $j\in \{1,2,\dots, K\}$ with 
$V_j\cap \mathcal{B}(\mathbf{v}) \neq \emptyset$.
For $\mathbf{v}\neq \mathbf{v}^\prime$, we have $J_\mathbf{v}\cap J_{\mathbf{v}^\prime} = \emptyset$ because
$\diam(V_j, \mathbf{d}^\prime_n) < 1/b$ and 
the distance between $\mathcal{B}(\mathbf{v})$ and $\mathcal{B}(\mathbf{v}^\prime)$ with respect to $\mathbf{d}^\prime_n$
is greater than or equal to $1/b$.

Since $\Lambda(\mathbf{v}) \subset \bigcup_{j\in J_{\mathbf{v}}} \pi^{-1}(V_j) \cap \rho^{-1}(\mathbf{t})$ and 
$\Lambda(\mathbf{v})$ is $\varepsilon$-separated with respect to $\mathbf{d}^{\prime\prime}_n$,
\[ \left|\Lambda(\mathbf{v})\right| \leq \sum_{j\in J_\mathbf{v}}
    \#\left(\pi^{-1}(V_j)\cap \rho^{-1}(\mathbf{t}), \mathbf{d}^{\prime\prime}_n, \varepsilon\right). \]
From $0\leq w \leq 1$ and $\left|\Lambda(\mathbf{v})\right| \geq (1/3) \mathcal{N}(\mathbf{t}; \mathbf{v})$,
\[ 3^{-w} \mathcal{N}(\mathbf{t}; \mathbf{v})^w \leq 
    \sum_{j\in J_{\mathbf{v}}} \#\left(\pi^{-1}(V_j)\cap \rho^{-1}(\mathbf{t}), \mathbf{d}^{\prime\prime}_n, \varepsilon\right)^w. \]
Summing this over $\mathbf{v}\in \Omega_n$,
\[ 3^{-w} \sum_{\mathbf{v}\in \Omega_n} \mathcal{N}(\mathbf{t}; \mathbf{v})^w \leq 
    \sum_{j=1}^K \#\left(\pi^{-1}(V_j)\cap \rho^{-1}(\mathbf{t}), \mathbf{d}^{\prime\prime}_n, \varepsilon\right)^w. \]
Since $\sum_{\mathbf{v}\in \Omega_n} \mathcal{N}(\mathbf{t}; \mathbf{v})^w \geq 
             \frac{1}{3}\sum_{\mathbf{v}\in B^n}  \mathcal{N}(\mathbf{t}; \mathbf{v})^w$, we obtain
\[ 3^{-1-w} \sum_{\mathbf{v}\in B^n} \mathcal{N}(\mathbf{t}; \mathbf{v})^w \leq
    \#^w\left(\rho^{-1}(\mathbf{t}), n, \varepsilon\right), \quad 
    (0< \varepsilon < 1/a). \]
Therefore we conclude 
\[ \limsup_{n\to \infty} \frac{1}{n}\log \sum_{\mathbf{v}\in B^n}
   \mathcal{N}(\mathbf{t}; \mathbf{v})^w \leq \h^w\left(\rho^{-1}(\mathbf{t}), \pi, T\right). \]
\end{proof}

For $\mathbf{t} \in \mathbb{T}^2$, we introduce $\tau_n = (\alpha_n, \beta_n) \in A\times B$ $(n\geq 1)$ by 
\[ \mathbf{t} = \left(\sum_{n=1}^\infty \frac{\alpha_n}{a^n}, \sum_{n=1}^\infty \frac{\beta_n}{b^n}\right). \]
They are uniquely determined for Lebesgue almost every $\mathbf{t} \in \mathbb{T}^2$.
(We ignore points $\mathbf{t}$ for which $\tau_n$ are not uniquely determined.)
When $\mathbf{t}$ is uniformly distributed over $\mathbb{T}^2$, the variables $\tau_n$ become independent and uniformly distributed 
over $A\times B$.

Now we can calculate the combinatorial version of the weighted topological entropy 
in terms of matrix products.

\begin{proposition}  \label{proposition: combinatorial entropy and matrix product}
For Lebesgue almost every $\mathbf{t}\in \mathbb{T}^2$ 
\begin{equation} \label{eq: combinatorial entropy and matrix product}
   \h^w\left(\rho^{-1}(\mathbf{t}), \pi, T\right) = \lim_{n\to \infty} 
    \frac{1}{n}\log \left(\sum_{(v_1, \dots, v_n)\in B^n} \norm{Q_{\tau_n, v_n} Q_{\tau_{n-1}, v_{n-1}}\cdots Q_{\tau_1, v_1}}^w\right). 
\end{equation}    
\end{proposition}

\begin{proof}
First notice that this statement is claimed only for Lebesgue almost every $\mathbf{t}$ while 
Lemma \ref{lemma: expression of combinatorial entropy} holds true for every $\mathbf{t}$.
We also notice that the limit in the right-hand side of \eqref{eq: combinatorial entropy and matrix product}
exists for almost every $\mathbf{t}$ because of the subadditive ergodic theorem.

Set 
$\mathcal{N}_0(\mathbf{t}; v_1, \dots, v_n) = \norm{Q_{\tau_n, v_n} Q_{\tau_{n-1}, v_{n-1}}\cdots Q_{\tau_1, v_1}}$.
It follows from Lemma \ref{lemma: counting intersection and matrix product} that
$\mathcal{N}(\mathbf{t}; v_1, \dots, v_n) \leq \mathcal{N}_0(\mathbf{t}; v_1, \dots, v_n)$
for all $v_1, \dots, v_n$ and almost every $\mathbf{t}\in \mathbb{T}^2$.
In particular, from Lemma \ref{lemma: expression of combinatorial entropy},
\[ \h^w\left(\rho^{-1}(\mathbf{t}), \pi, T\right) \leq 
    \lim_{n\to \infty} 
    \frac{1}{n}\log \left(\sum_{(v_1, \dots, v_n)\in B^n} \mathcal{N}_0(\mathbf{t}; v_1, \dots, v_n)^w\right) \quad 
    \text{for a.e. $\mathbf{t}\in \mathbb{T}^2$}. \]
We would like to prove the reverse inequality.

By Lemma \ref{lemma: extinction or not}, removing a null set of $\mathbf{t} \in \mathbb{T}^2$, we can assume that 
for any $\varepsilon>0$ there exists a natural number $n_0 = n_0(\mathbf{t}, \varepsilon)$ 
such that, for any nonnegative $4\times 4$-matrix $M$,
the product $Q_{\tau_{n+\lfloor n\varepsilon \rfloor}} Q_{\tau_{n+\lfloor n\varepsilon \rfloor -1}}\cdots Q_{\tau_{n+1}}M$
is either zero or persistent for $n\geq n_0$.

Recall that, for $\mathbf{i}, \mathbf{k}\in I$, the $4\times 4$-matrix $E_{\mathbf{i} \mathbf{k}}$ was given by 
\[ E_{\mathbf{i} \mathbf{k}}(\mathbf{j}, \mathbf{l}) = \begin{cases} 1 & \left((\mathbf{j}, \mathbf{l}) = (\mathbf{i}, \mathbf{k})\right) \\
                                                                                                0 & \left((\mathbf{j}, \mathbf{l}) \neq (\mathbf{i}, \mathbf{k})\right)
                                                                                                \end{cases}. \]
We define a subset $J(\mathbf{t}, n, \varepsilon)\subset I\times I$ as the set of 
$(\mathbf{i}, \mathbf{k})\in I\times I$ such that the product
$Q_{\tau_{n+\lfloor n\varepsilon \rfloor}} Q_{\tau_{n+\lfloor n\varepsilon \rfloor -1}}\cdots Q_{\tau_{n+1}}E_{\mathbf{i} \mathbf{k}}$
is persistent.
For a nonnegative $4\times 4$-matrix $M$ we define 
\[ M^{(\mathbf{t}, n,\varepsilon)} 
   = \sum_{(\mathbf{i}, \mathbf{k})\in J(\mathbf{t}, n, \varepsilon)} M(\mathbf{i}, \mathbf{k}) E_{\mathbf{i} \mathbf{k}}. \]
Then we have 
\[ Q_{\tau_{n+\lfloor n\varepsilon \rfloor}} Q_{\tau_{n+\lfloor n\varepsilon \rfloor -1}}\cdots Q_{\tau_{n+1}} M 
   = Q_{\tau_{n+\lfloor n\varepsilon \rfloor}} Q_{\tau_{n+\lfloor n\varepsilon \rfloor -1}}\cdots 
   Q_{\tau_{n+1}} \cdot M^{(\mathbf{t}, n,\varepsilon)}  \]
for $n\geq n_0(\mathbf{t},\varepsilon)$.   

Let $(\mathbf{i}, \mathbf{k})\in J(\mathbf{t}, n, \varepsilon)$.
By the definition of persistence, we have
\[ Q_{\tau_{n+h}} Q_{\tau_{n+h-1}}\cdots Q_{\tau_{n+1}} E_{\mathbf{i} \mathbf{k}} \neq 0 \quad 
    \text{for any $h>0$}. \]
Now we can apply Lemma \ref{lemma: approximate overlap and real intersection point} and obtain that
\[ \left(Q_{\tau_n, v_n} Q_{\tau_{n-1}, v_{n-1}} \cdots Q_{\tau_1, v_1}\right)(\mathbf{i}, \mathbf{k})
    \leq \mathcal{N}(\mathbf{t}; v_1, \dots, v_n) \]
for all $v_1, \dots, v_n\in B$. 
It follows that
\[ \text{Every entry of } \left(Q_{\tau_n, v_n}Q_{\tau_{n-1},v_{n-1}}\cdots Q_{\tau_1, v_1}\right)^{(\mathbf{t}, n, \varepsilon)}
   \leq \mathcal{N}(\mathbf{t}; v_1, \dots, v_n). \]
For $n\geq n_0(\mathbf{t},\varepsilon)$,
\begin{align*}
  &Q_{\tau_{n+\lfloor n\varepsilon \rfloor}} Q_{\tau_{n+\lfloor n\varepsilon \rfloor -1}}\cdots Q_{\tau_{n+1}} \cdot
   Q_{\tau_n, v_n}\cdots Q_{\tau_1, v_1}  \\
  & = Q_{\tau_{n+\lfloor n\varepsilon \rfloor}} Q_{\tau_{n+\lfloor n\varepsilon \rfloor -1}}\cdots 
   Q_{\tau_{n+1}} \cdot \left(Q_{\tau_n, v_n}\cdots Q_{\tau_1, v_1}\right)^{(\mathbf{t}, n,\varepsilon)}. 
\end{align*}   
Every entry of this matrix is bounded from above by $(4ab)^{n\varepsilon} \mathcal{N}(\mathbf{t}; v_1, \dots, v_n)$.
Hence 
\[ \norm{Q_{\tau_{n+\lfloor n\varepsilon \rfloor}} Q_{\tau_{n+\lfloor n\varepsilon \rfloor -1}}\cdots Q_{\tau_{n+1}} \cdot
   Q_{\tau_n, v_n}\cdots Q_{\tau_1, v_1}} \leq 16(4ab)^{n\varepsilon} \mathcal{N}(\mathbf{t}; v_1, \dots, v_n). \]
In particular, for any $v_1, \dots, v_{n+\lfloor n\varepsilon \rfloor}\in B$ with $n\geq n_0(\mathbf{t}, \varepsilon)$
\[ \mathcal{N}_0(\mathbf{t}; v_1, \dots, v_n, v_{n+1}, \dots, v_{n+\lfloor n\varepsilon \rfloor}) 
    \leq 16(4ab)^{n\varepsilon} \mathcal{N}(\mathbf{t}; v_1, \dots, v_n). \]
Hence 
\[ \sum_{v_1, \dots, v_{n+\lfloor n\varepsilon \rfloor} \in B} \mathcal{N}_0(\mathbf{t}; v_1, \dots, v_{n+\lfloor n\varepsilon  \rfloor})^w
    \leq b^{n\varepsilon}\cdot 16^w \cdot (4ab)^{n\varepsilon w} \sum_{v_1, \dots, v_n\in B} 
    \mathcal{N}(\mathbf{t}; v_1, \dots, v_n)^w. \]
Therefore 
\[ \lim_{n\to \infty} 
    \frac{1+\varepsilon}{n}\log \left(\sum_{(v_1, \dots, v_n)\in B^n} \mathcal{N}_0(\mathbf{t}; v_1, \dots, v_n)^w\right) \leq 
    \varepsilon \log b + w \varepsilon \log (4ab) + \h^w\left(\rho^{-1}(\mathbf{t}), \pi, T\right). \]
We can let $\varepsilon \to 0$ and obtain the conclusion.
\end{proof}

Applying Corollary \ref{corollary: calculating the Feng--Huang entropy} and
Proposition \ref{proposition: combinatorial entropy and matrix product} to our fundamental identity 
\[ \FHh^w\left(\rho^{-1}(\mathbf{t}), \pi, T\right) = \h^w\left(\rho^{-1}(\mathbf{t}), \pi, T\right), \]
we obtain Theorem \ref{theorem: intersection of Bedford--McMullen carpets}:

\begin{corollary}[$=$ Theorem \ref{theorem: intersection of Bedford--McMullen carpets}]
For Lebesgue almost every $\mathbf{t}\in \mathbb{T}^2$
\[ \dim_{\mathrm{H}}\left((X_1+\mathbf{t})\cap X_2\right) 
    = \lim_{n\to \infty} 
    \frac{1}{n}\log_b \left(\sum_{(v_1, \dots, v_n)\in B^n} \norm{Q_{\tau_n, v_n} Q_{\tau_{n-1}, v_{n-1}}\cdots Q_{\tau_1, v_1}}^w\right). \]
\end{corollary}

\begin{remark}
We have so far considered only the case that $\mathbf{t}\in \mathbb{T}^2$ is chosen according to the 
Lebesgue measure. However, this is just for simplicity of the exposition.
The same conclusion (and the same proof)
work for other \lq\lq{}Bernoulli measures\rq\rq{} as well:
Let $p$ be a probability measure on $A\times B$ such that
 $p(\mathbf{u}) >0$ for every $\mathbf{u}\in A\times B$.
Let $\tau_1, \tau_2, \tau_3, \dots$ be independent and identically distributed random variables that 
take values in $A\times B$ according to $p$.
Let $\tau_n = (\alpha_n, \beta_n)$ and set
\[ \mathbf{t} := \left(\sum_{n=1}^\infty \frac{\alpha_n}{a^n}, \sum_{n=1}^\infty \frac{\beta_n}{b^n}\right) \in \mathbb{T}^2. \]
Then, almost surely, we have 
\begin{equation} \label{eq: dimension formula for general Bernoulli measure}
 \dim_{\mathrm{H}}\left((X_1+\mathbf{t})\cap X_2\right) 
    = \lim_{n\to \infty} 
    \frac{1}{n}\log_b \left(\sum_{(v_1, \dots, v_n)\in B^n} \norm{Q_{\tau_n, v_n} Q_{\tau_{n-1}, v_{n-1}}\cdots Q_{\tau_1, v_1}}^w\right), 
\end{equation}    
where $w = \log_a b$.
We notice that the full-support assumption of $p$ cannot simply be omitted. 
For example, let $a=b=3$, $D_1 = \{(0,0), (0,1), (0,2)\}$ and $D_2 = \{(2,0), (2,1), (2,2)\}$.
We suppose that $p$ is the Dirac measure at $\{(0,0)\}$. In this case $X_1 = X_2 = \{0\}\times \mathbb{T}$, and
we almost surely have $(X_1+\mathbf{t})\cap X_2 = \{0\}\times \mathbb{T}$, which has dimension one. However 
$Q_{(0,0)} = 0$ and hence the right-hand side of \eqref{eq: dimension formula for general Bernoulli measure} is $-\infty$.
The point is that a statement corresponding to Lemma \ref{lemma: counting intersection and matrix product} does not hold for 
this example. A sufficient condition for the validity of \eqref{eq: dimension formula for general Bernoulli measure} is that, 
for every $n\geq 1$, the distribution of the random variable 
\[  \left(\sum_{m=n}^\infty \frac{\alpha_m}{a^m}, \sum_{m=n}^\infty \frac{\beta_m}{b^m}\right) \]
has measure zero on $\{0\}\times \mathbb{T} \cup \mathbb{T}\times \{0\}$.
\end{remark}

\end{document}